\documentclass[12pt]{article}
\usepackage{appendix}
\usepackage{amsthm,bbm,amssymb,amsmath,bm,mathrsfs}
\usepackage{graphicx}
\usepackage{subfigure}
\usepackage{float}
\usepackage{epsfig}
\usepackage[colorlinks=true, citecolor=blue]{hyperref}
\usepackage{abstract}
\usepackage{natbib}
\usepackage{array}
\usepackage{booktabs}
\usepackage[english]{babel}
\usepackage[utf8]{inputenc}
\usepackage{algorithm,algpseudocode}  %for pseudo code writing
\allowdisplaybreaks[4]
\newtheorem{thm}{\textbf{Theorem}}

\newtheorem{definition}{\textbf{Definition}}
\newtheorem{lem}{\textbf{Lemma}}

\newtheorem{remark}{\textbf{Remark}}
\newtheorem{example}{\textbf{Example}}

\DeclareMathOperator*{\argmax}{argmax}

\begin{document}

\title{\vspace{-2.6cm}
Mean-Variance Optimization and Algorithm for Finite-Horizon Markov
Decision Processes}
\author{{Li Xia$^{1}$}, \ {Zhihui Yu$^1$} \\
\small \vspace{-0.4cm}
    {$^1$School of Business, Sun Yat-Sen University, Guangzhou, China}\\
%\small \vspace{-0.4cm}
    %{$^2$Guangdong Province Key Laboratory of Computational Science, Guangzhou, China}
%    \\
%    \small \vspace{-0.4cm}
%    {$^3$Department of Management Science and Engineering, Stanford University, CA 94305, USA}
    }
\date{}
\maketitle
\begin{abstract}
Multi-period mean-variance optimization is a long-standing problem,
caused by the failure of dynamic programming principle. This paper
studies the mean-variance optimization in a setting of
finite-horizon discrete-time Markov decision processes (MDPs), where
the objective is to maximize the combined metrics of mean and
variance of the accumulated rewards at terminal stage. By
introducing the concepts of pseudo mean and pseudo variance, we
convert the original mean-variance MDP to a bilevel MDP, where the
outer is a single parameter optimization of the pseudo mean and the
inner is a standard finite-horizon MDP with an augmented state space
by adding an auxiliary state of accumulated rewards. We further
study the properties of this bilevel MDP, including the optimality
of history-dependent deterministic policies and the piecewise
quadratic concavity of the inner MDPs' optimal values with respect
to the pseudo mean. To efficiently solve this bilevel MDP, we
propose an iterative algorithm that alternatingly updates the inner
optimal policy and the outer pseudo mean. We prove that this
algorithm converges to a local optimum. We also derive a sufficient
condition under which our algorithm converges to the global optimum.
Furthermore, we apply this approach to study the mean-variance
optimization of multi-period portfolio selection problem, which
shows that our approach exactly coincides with the classical result
by \cite{li2000optimal} in financial engineering. Our approach
builds a new avenue to solve mean-variance optimization problems and
has wide applicability to any problem modeled by MDPs, which is
further demonstrated by examples of mean-variance optimization for
queueing control and inventory management.
\end{abstract}


\section{Introduction}\label{sec:intro}
Mean-variance optimization is a classical problem in finance, which
was proposed by Nobel laureate \cite{markowitz1952} to control the
return and risk of portfolio, originally in a static optimization
regime. Since variance is a widely adopted metric to measure the
deviation of random variables, mean-variance optimization is also
studied in other fields, such as the safety control in renewable
power systems \citep{Li2014}, fairness control in queueing systems
\citep{Avi-Itzhak2004}, and risk management in inventory and supply
chain management \citep{Chio2016}. It is natural to extensively
study the mean-variance optimization in a stochastic dynamic regime.
However, this problem is challenging since the dynamic programming
principle fails and the time consistency does not hold, which is
caused by the non-separable (we would rather call it non-additive
and non-Markovian) property of variance function in dynamic
programming \citep{Ruszczynski10,Shapiro09,sobel1994mean}. The
mean-variance optimization of stochastic dynamic systems is a
long-standing open problem continually attracting research attention
in the literature
\citep{bauerle2024,chung1994mean,dai2021dynamic,sobel1982variance}.

One of the main research streams of multi-period mean-variance
optimization focuses on the portfolio selection in financial
engineering, from the perspective of stochastic control. The seminal
work by \cite{li2000optimal} proposed an embedding method to compute
the optimal policy with analytical forms. Then \cite{Zhou00}
extended this work to a continuous-time linear quadratic model and
derived even more elegant results analytically. These works
motivated a series of following researches by using the same idea of
the embedding method. For example, \cite{Zhou04} studied the
continuous-time model with regime-switching, \cite{Zhu04} studied
the risk control over bankruptcy in a more general formulation,
\cite{Yi08} studied the asset-liability management with uncertain
investment horizon, \cite{Gao13} extended this approach to study the
cardinality constrained portfolio selection with mean-variance
optimization. A more complete introduction on the related topics can
be referred to a recent survey paper \citep{Cui22}. Although the
work by \cite{li2000optimal} provides an analytical way to study the
multi-period mean-variance optimization, this approach heavily
depends on the specific model of portfolio selection. It is hardly
applicable to other problems except portfolio selection.

Another research stream on mean-variance optimization is from the
perspective of Markov decision processes (MDPs), since MDP is a
widely adopted methodology to study stochastic dynamic optimization
problems. Classical optimization criteria of MDPs focus on the
expectation of discounted or long-run average of accumulated rewards
\citep{Bertsekas05,Puterman94}. Since the rewards in MDPs are random
variables, it is natural to concern their higher order quantities
rather than expectations. The paper by \cite{sobel1982variance} is
one of the pioneering works on MDPs with a variance-related
optimality criterion. He focused on the variance minimization of
discounted rewards in infinite-horizon MDPs, where some property
analysis were presented for both MDPs and semi-MDPs.
\cite{filar1989variance} studied the variance-penalized MDP with a
penalty for the variability of rewards, by formulating it as a
non-convex mathematical program in the space of state-action
frequencies. \cite{sobel1994mean} and \cite{chung1994mean}
separately studied the mean-variance tradeoff in undiscounted MDPs,
both from the viewpoint of mathematical programming to analyze the
Pareto optima that minimize the variance among the policies with
mean greater than a given value. There are numerous works following
this research stream of MDPs with variance-related criteria. Some
excellent works can be referred to
\cite{Haskell13,Hernandez-Lerma1996,hernandez1999sample,guo2009mean,guo2015first,xia2016,xia18,xia2020,xia2025global},
and references therein, just to name a few.  These aforementioned
works either study steady-state variance MDPs through mathematical
programming approaches
\citep{chung1994mean,filar1989variance,Haskell13,sobel1994mean} and
sensitivity-based optimization methods
\citep{xia2016,xia18,xia2020,xia2025global}, or focus on variance
optimization of accumulated rewards by considering policies whose
mean performance already achieves the optimum, thereby converting
the problem into a standard expected MDP
\citep{hernandez1999sample,guo2009mean,xia18}. Recently,
\cite{bauerle2024} proposed a new approach to analyze the
mean-variance optimization of the instantaneous reward at the
terminal stage of finite-horizon MDPs through a so-called population
version MDP by replacing the original state space with the set of
probability measures on it. The solution of this new MDP model meets
the Bellman optimality principle and is time consistent, but the
computational complexity is intractable since the state has a high
dimensional continuous space. There does not exist an approach to
efficiently analyze and solve the mean-variance optimization of
accumulated rewards in a finite-horizon MDP, which is a very common
motivation since many decision-making problems focus on finite
horizon, such as multi-period portfolio selection and inventory
management.

The community of reinforcement learning has also been paying
attention to the mean-variance optimization of stochastic dynamic
systems, which is called risk-sensitive reinforcement learning. With
the great success of AlphaGo, deep reinforcement learning becomes a
hot research topic where the policy and the value function are
approximated by deep neural networks and policy gradients are
utilized to do optimization. The early work of variance-related
reinforcement learning focuses on improving the sampling efficiency
of gradient estimators for variance-related performance metrics
\citep{Borkar10,Prashanth13,Tamar12}. Some recent works reformulate
the mean-variance optimization with the Fenchel duality
\citep{Xie18} and propose gradient-based algorithms to find local
optima \citep{Bisi20,Zhang21}. A more comprehensive viewpoint on the
risk-sensitive reinforcement learning can be referred to a recent
survey book by \cite{Prashanth22}. In addition, recent studies have
investigated reinforcement learning algorithms in continuous-time
and continuous-state settings, providing a novel perspective for
solving multi-period mean-variance portfolio optimization problems
\citep{huang2024mean,wang2020continuous}. However, all these
reinforcement learning approaches focus on approximated algorithms
for sample path learning, which suffer from slow and local
convergence of gradient algorithms and huge sample size.
Reinforcement learning is algorithm centric and is not applicable to
rigorously study the property of mean-variance optimization. How to
effectively analyze and solve the mean-variance optimization in
finite-horizon MDPs is still an open problem.

Although the mean-variance optimization of stochastic dynamic
systems has been studied in different disciplines including
stochastic control, MDPs, and reinforcement learning, these
approaches focus on different aspects and it seems that they are
hardly merged. In this paper, we aim to study the mean-variance
optimization of accumulated rewards in a finite-horizon
discrete-time MDP, which has been relatively underexplored in the
literature. Our objective is to find an optimal policy among
history-dependent randomized policies to simultaneously maximize the
mean and minimize the variance of accumulated rewards at the
terminal stage. We first formulate a fairly general model of
finite-horizon discrete-time MDPs with mean-variance optimality
criterion. To resolve the challenge of non-additivity and
non-Markovian (or called non-separability) of variance metrics, we
introduce the concepts called pseudo mean and pseudo variance, and
convert the mean-variance MDP (MV-MDP) to a bilevel MDP. The inner
level is a pseudo mean-variance optimization of MDPs and the outer
level is a single parameter optimization of the pseudo mean. The
inner problem of pseudo mean-variance optimization is not a standard
finite-horizon MDP. Considering the fact that the pseudo variance
term contains history rewards, we treat the anticipation of
accumulated rewards from the current stage to the terminal stage as
an auxiliary state and derive an augmented MDP. We show that the
inner pseudo MV-MDP with the augmented state is a standard
finite-horizon MDP and it can be solved by dynamic programming. The
optimality of history-dependent deterministic policies is proved
based on the bilevel formulation, which indicates that Markov
policies may not be optimal any more for this problem. We further
prove that the optimal value function of the inner pseudo MV-MDP is
piecewisely quadratic concave with respect to the outer pseudo mean.
By utilizing these optimality properties, we develop an iterative
algorithm that alternatingly updates the inner optimal policy and
the outer pseudo mean. Our iterative algorithm has a form similar to
policy iteration which exhibits fast convergence in most cases. We
prove that this algorithm converges to local optima, in the sense of
a parameterized space (mixed policy space or parameter space of
pseudo mean). Furthermore, we derive a sufficient condition that can
guarantee the global convergence of our algorithm. We show that the
multi-period portfolio selection problem satisfies this sufficient
condition and our approach exactly coincides with the classical
result by \cite{li2000optimal}. Finally, we use the numerical
experiments in portfolio selection, queueing control, and inventory
management to demonstrate the effectiveness of our approach. The
numerical results show that our approach always finds the global
optimum of the multi-period mean-variance portfolio selection
problem and mean-variance queueing control, while it finds the local
optima of the multi-period mean-variance inventory management
problem.

The contribution of this paper is threefold. First, we derive an
effective approach to study the mean-variance optimization of
accumulated rewards in finite-horizon discrete-time MDPs. To the
best of our knowledge, our work is the first to solve this
long-standing problem in the literature on MDPs. With the concepts
of pseudo mean and pseudo variance, we convert the original problem
to a bilevel MDP where the inner is a state-augmented MDP and the
outer is the optimization of pseudo mean. Different from most MDPs
in the literature, the optimum of our MV-MDP is attainable by
history-dependent deterministic policies, not by Markov
deterministic policies. Second, we propose an efficient policy
iteration type algorithm to solve this MV-MDP problem. We prove that
the algorithm can converge to a local optimum after a finite number
of iterations. We also derive a sufficient condition under which the
algorithm can find the global optimum. Third, we show that our
approach can unify the classical result of multi-period portfolio
selection by \cite{li2000optimal}. As a comparison, our approach has
a much wider applicability since Markov model is much more general
than portfolio selection model, which is also demonstrated by
numerical examples of mean-variance optimization for queueing
control and inventory management.

The rest of the paper is organized as follows. In
Section~\ref{sec:model}, we give the problem formulation of
finite-horizon MDPs with mean-variance criterion.
Section~\ref{sec:opt} presents the main theoretical results,
including the bilevel MDP framework and the optimality analysis of
this problem. In Section~\ref{sec:alg}, we derive the iterative
algorithm and the convergence analysis for mean-variance
finite-horizon MDPs. In Section~\ref{sec:exp}, we apply our approach
and algorithm to solve the multi-period mean-variance optimization
for portfolio selection, queueing control, and inventory management,
respectively. Finally, we conclude this paper in
Section~\ref{sec:con}.

\section{Problem Formulation}\label{sec:model}
A finite-horizon discrete-time MDP is denoted by a collection
$\mathcal{M} := \langle \mathcal T, \mathcal S, \mathcal A,
(\mathcal{A}(s) \subset \mathcal A, s \in \mathcal{S}), (P_t, t \in
\mathcal T), (r_t, t \in \mathcal T) \rangle$, where $\mathcal
T:=\left\{0,1,\ldots,T-1\right\}$ is the set of decision epochs with
terminal stage $T < \infty$; $\mathcal S$ and $\mathcal A$ represent
the finite spaces of states and actions, respectively;
$\mathcal{A}(s)$ denotes the admissible action set at state $s \in
\mathcal{S}$ with $\cup_{s \in
\mathcal{S}}\mathcal{A}(s)=\mathcal{A}$; $P_t$ denotes the Markov
kernel at decision epoch $t$ and $P_t(\cdot|s,a)$ is a probability
measure on $\mathcal S$ for each given $(s,a) \in \mathcal K$, where
$\mathcal{K}:=\left\{(s,a): s \in \mathcal{S}, a \in
\mathcal{A}(s)\right\}$ is defined as the set of admissible
state-action pairs; and $r_t : \mathcal K \rightarrow \mathbb{R}$ is
the reward function with minimum $\underline r$ and maximum $\bar
r$, where $r_t(s,a)$ denotes the reward at decision epoch $t$
determined by the current state-action pair $(s,a) \in \mathcal{K}$.
Suppose the system state is $s_t \in \mathcal S$ at the current time
$t$, and an action $a_t \in \mathcal{A}(s_t)$ is adopted, the system
will receive an instantaneous reward $r_t(s_t,a_t)$, and then move
to a new state $s_{t+1} \in \mathcal S$ at the next time $t+1$
according to the transition probability $P_t(s_{t+1}|s_t,a_t)$. The
policy $u$ prescribes the action-selection rule at each decision
time epoch based on either history or just the current state, where
the former refers to a \emph{history-dependent policy} while the
latter refers to a \emph{Markov policy}. Specifically, a
\emph{history-dependent randomized policy} $u:=(u_t; t \in \mathcal
T)$ is a sequence of stochastic kernels $u_t$ which is a probability
distribution on action space $\mathcal{A}$ given history $h_t :=
\left\{s_0,a_0,s_1,\ldots,s_{t-1},a_{t-1},s_t\right\} \in
\mathcal{H}_t:=\mathcal K^t \times \mathcal S$ and $\sum_{a \in
\mathcal A(s_t)}u_t(a|h_t)=1$. Further,  $u$ degenerates into a
\emph{Markov randomized policy} if $u_t$ depends on the current
state $s_t$ instead of history $h_t$, i.e.,
$u_t(\cdot|h_t)=u_t(\cdot|s_t)$, $\forall h_t \in \mathcal{H}_t$. In
addition, if $u_t$ is a deterministic decision rule, i.e., $u_t:
\mathcal{H}_t \rightarrow \mathcal{A}$ or $u_t: \mathcal{S}
\rightarrow \mathcal{A}$, we call $u$ a \emph{history-dependent
deterministic policy} or \emph{Markov deterministic policy},
respectively. For notational simplicity, we denote by $\mathcal
U^{\rm HR}$, $\mathcal U^{\rm MR}$, $\mathcal U^{\rm HD}$, and
$\mathcal U^{\rm MD}$ the sets of all history-dependent randomized
policies, Markov randomized policies, history-dependent
deterministic policies, and Markov deterministic policies,
respectively. Obviously, we have $\mathcal U^{\rm HR} \supset
\mathcal U^{\rm MR} \supset \mathcal U^{\rm MD}$ and $\mathcal
U^{\rm HR} \supset \mathcal U^{\rm HD} \supset \mathcal U^{\rm MD}$.
For each initial state $s_0 \in \mathcal S$ and policy $u \in
\mathcal U^{\rm HR}$, by Ionescu Tulcea's Theorem
\citep[P.178]{Hernandez-Lerma1996}, there exists a unique
probability measure $\mathbbm{P}_{s_0}^{u}$ on the measurable space
$(\mathcal K^T \times \mathcal S, \mathcal B(\mathcal K^T \times
\mathcal S))$ such that
\begin{equation*}
\mathbbm{P}_{s_0}^{u}(s_0,a_0,s_1,\ldots,s_{T-1},a_{T-1},s_T) =
u_0(a_0|s_0)P_0(s_1|s_0,a_0)\cdots
u_{T-1}(a_{T-1}|h_{T-1})P_{T-1}(s_{T}|s_{T-1},a_{T-1}).\nonumber
\end{equation*}
Here and in what follows, we denote by $\mathbbm{E}_{s_0}^{u}$ the
expectation operator corresponding to $\mathbbm{P}_{s_0}^{u}$.

%{\color{red}For each initial state $s_0 \in \mathcal S$ and policy $u \in \mathcal U^{\rm RH}$, by the Theorem of C. Ionescu-Tulcea
%\citep[P.178]{Hernandez-Lerma1996}, there exists a unique probability space $(\Omega,\mathcal F, \mathbbm{P}_{s_0}^{u})$, where the sample space
%$\Omega:=\mathcal K^T \times \mathcal S$, the corresponding Borel $\sigma$-algebra $\mathcal F:=\mathcal B(\Omega)$, and the probability measure
%\begin{equation}
%   \nonumber
%\mathbbm{P}_{s_0}^{u}(s_0,a_0,s_1,\ldots,s_{T-1},a_{T-1},s_T) :=
%u_0(a_0|s_0)P_0(s_1|s_0,a_0)\cdots
%u_{T-1}(a_{T-1}|h_{T-1})P_{T-1}(s_{T}|s_{T-1},a_{T-1}).
%\end{equation}
%We denote by $\mathbbm{E}_{s_0}^{u}$ the expectation operator with respect to $\mathbbm{P}_{s_0}^{u}$. We also define a state-action process $\big((S_t, A_t); t \in \mathcal T\big)$ on $(\Omega,\mathcal F, \mathbbm{P}_{s_0}^{u})$ such that
%\begin{equation}
%   \nonumber
%   S_t(\omega)=s_t, A_t(\omega)=a_t, \quad\forall \omega=(s_0,a_0,\ldots,s_T) \in \Omega, t \in \mathcal T.
%\end{equation}
%}
This paper aims to study the finite-horizon MV-MDPs where both the
mean and the variance of accumulated rewards are optimized.
Specifically, the horizon $T$ is supposed to be fixed and we denote
by random variable
\begin{equation}
    \nonumber
    R_{t:T} := \sum\limits_{\tau=t}^{T-1}  r_\tau(s_\tau, a_\tau) +
    r_T(s_T) = \sum\limits_{\tau=t}^{T-1}  r_\tau(s_\tau, a_\tau)
\end{equation}
the accumulated rewards from stage $t$ to the terminal stage $T$, where we assume
$r_T(s_T) \equiv 0$ without loss of generality. Given an initial
state $s_0 \in \mathcal{S}$ and a policy $u \in \mathcal U^{\rm HR}$,
the mean and the variance of $T$-horizon accumulated rewards are as
follows, respectively.
\begin{equation}\label{eq_muvar}
\begin{array}{ccl}
    \mu^u_0 (s_0) &:=& \mathbbm{E}_{s_0}^{u}[R_{0:T}], \\
    \sigma^u_0 (s_0) &:=& \mathbbm{E}_{s_0}^{u}\big[\big(R_{0:T}-\mu^u_0 (s_0)\big)^2\big].
\end{array}
\end{equation}
To derive the \emph{Pareto optima} of a multi-objective optimization
problem, we use the so-called global criterion method in which all
the multiple objective functions are linearly combined to form a
single objective function \citep{marler2004survey}. That is, we
introduce a risk aversion coefficient $\lambda \ge 0$ and define the
mean-variance value of the $T$-horizon accumulated rewards under
policy $u$ as
\begin{equation}\label{eq_Ju}
   J^u_0 (s_0) := \mu^u_0 (s_0)-\lambda \sigma^u_0 (s_0), \quad s_0 \in \mathcal{S},~ u \in \mathcal U^{\rm HR}.
\end{equation}
In what follows, $\lambda$ is fixed unless otherwise stated. We
denote by $\mathcal M$ the finite-horizon MV-MDP which aims to
maximize the combined metrics of mean and variance for each initial
state $s_0$, i.e.,
\begin{flalign}\label{equ:mv_mdp}
&\mathcal M: \hspace{4cm} J^*_0(s_0) := \sup\limits_{u \in \mathcal
U^{\rm HR}}   J^u_0 (s_0), \quad s_0 \in \mathcal{S},&
\end{flalign}
where $J^*_0(\cdot)$ is called the optimal value function of the
finite-horizon MV-MDP, and a policy $u^* \in \mathcal U^{\rm HR}$ is
called an \emph{optimal policy} for solving $\mathcal M$ if it
attains the optimal value, i.e., $J^{u^*}_0 (\cdot) = J^*_0(\cdot)$.

Note that, the finite-horizon MV-MDP $\mathcal M$ in
(\ref{equ:mv_mdp}) cannot be solved by directly using the method of
dynamic programming since the variance term in (\ref{eq_Ju}) is not
\emph{separable} into the summation of multiple \emph{Markovian} and
\emph{additive} terms, which also causes the \emph{time
inconsistency} \citep{Ruszczynski10,Shapiro09}. This fundamentally
challenging problem attracts a lot of research attention in
different disciplines, while it is not completely resolved in the
literature, as we introduced in Section~\ref{sec:intro}. In this
paper, we aim to propose a new optimization approach to accomplish
this challenge.
% Before stating the optimization approach, we derive the following lemma to show that the optimality of Markov randomized policies for the finite horizon  MV-MDP (\ref{equ:mv_mdp}).
% \begin{lem}\label{lem:mark}
%   The optimal policy can be attained by Markov randomized policies.
% \end{lem}
%\begin{proof}
%   This result holds due to the relationship between the history-dependent randomized policy and Markov randomized policy (Theorem 5.5.1 of \cite{Puterman94}).
%\end{proof}
%
%With Lemma~\ref{lem:mark}, we  can limit our policy searching
%space to $\Pi^M$.

\section{Optimization Approach}\label{sec:opt}
In this section, we propose a new optimization approach to study the
finite-horizon MV-MDP. First, we notice that the variance
of a random variable $X$ has the following property
\begin{equation}\label{eq_pseudovar}
\sigma(X) = \mathbbm E[(X - \mathbbm E[X])^2] = \min\limits_{y \in
\mathbb{R}} \mathbbm E[(X - y)^2] = \min\limits_{y \in \mathbb{R}}
\hat{\sigma}(X,y),
\end{equation}
where we call
$\hat{\sigma}(X,y):=\mathbbm E[(X - y)^2]$ the \emph{pseudo
variance} of $X$ with the \emph{pseudo mean} $y \in \mathbb R$, and
the minimum in (\ref{eq_pseudovar}) is attained at $y^*=\mathbbm
E[X]$. That is, the pseudo variance $\hat{\sigma}(X,y)$ equals the
real variance $\sigma(X)$ when the
pseudo mean $y$ equals the real mean $\mathbbm E[X]$ \citep{xia2016}.

Using the above property (\ref{eq_pseudovar}) and definition
(\ref{eq_muvar}), we can convert the finite-horizon MV-MDP
(\ref{equ:mv_mdp}) to a \emph{bilevel MDP} by introducing a pseudo
mean $y_0 \in \mathbb R$, i.e.,
\begin{align}\label{equ:bilevel}
J^{*}_0 (s_0)
&=\sup\limits_{u \in \mathcal U^{\rm HR}} \left\{\mu^u_0 (s_0)-\lambda \sigma^u_0 (s_0)\right\}    \nonumber\\
&=\sup\limits_{u \in \mathcal U^{\rm HR}}\max\limits_{y_0 \in \mathbb {R}} \left\{\mathbbm{E}_{s_0}^{u}[R_{0:T}]-\lambda \mathbbm{E}_{s_0}^{u}\big[\big(R_{0:T}-y_0\big)^2\big]\right\}   \nonumber \\
&=\max\limits_{y_0 \in \mathbb{R}}\sup\limits_{u \in \mathcal U^{\rm
HR}} \mathbbm{E}_{s_0}^{u}\big[R_{0:T}-\lambda\big(R_{0:T}-y_0\big)^2\big].
\end{align}
%The bilevel optimization is a standard conversion in MV-MDPs
%\citep{xia2016,xia2020}, which is also used in finite/infinite
%horizon CVaR MDPs \citep{bauerle2011,huang2016}, due to the
%properties of variance: $\var(R^0)=\min\limits_{y_0 \in \mathcal
%R}\mathbbm{E}_{s_0}^{u}\big[\big(R^0-y_0\big)^2\big]$ and of CVaR:
%$\CVaR_\alpha(R^0)=\min\limits_{y_0 \in \mathcal
%R}\left\{y_0+\frac{1}{1-\alpha}\mathbbm{E}_{s_0}^{u}[R^0-y_0]^+\right\}$.
The outer level of (\ref{equ:bilevel}) is a single parameter
optimization problem with variable $y_0$, and the inner level is a
policy optimization problem of maximizing the mean minus pseudo
variance. For notational simplicity, we denote the \emph{pseudo
mean-variance} of the $T$-horizon accumulated rewards under pseudo
mean $y_0 \in \mathbb R$ and policy $u \in \mathcal U^{\rm HR}$ by
\begin{equation}\label{eq_pseudoMV}
   \hat J^u_0(s_0,y_0):=\mathbbm{E}_{s_0}^{u}\big[R_{0:T}-\lambda\big(R_{0:T}-y_0\big)^2\big], \quad s_0 \in \mathcal {S}.
\end{equation}
We call the inner optimization problem a \emph{pseudo mean-variance
MDP} (pseudo MV-MDP) which is denoted by $\hat{\mathcal M}(y_0)$ and
aims to maximize the pseudo mean-variance under pseudo mean $y_0 \in
\mathbb R$ for each initial state $s_0 \in \mathcal S$, i.e.,
\begin{flalign}\label{equ:inner}
&\hat{\mathcal M}(y_0): \hspace{3cm}  \hat J^*_0(s_0,y_0) =
\sup\limits_{u \in \mathcal U^{\rm HR}} \hat J^u_0(s_0,y_0),
\quad  s_0 \in \mathcal {S}.&
\end{flalign}
We call $\hat J^*_0(\cdot,y_0) $ the optimal pseudo mean-variance
function. Further, a policy $\hat u^* \in \mathcal{U}^{\rm HR}$ is
called an optimal policy of the pseudo MV-MDP problem
(\ref{equ:inner}) if $\hat J^{\hat u^*}_0(\cdot,y_0)=\hat
J^*_0(\cdot,y_0)$.

It is worth noting that the inner problem $\hat{\mathcal M}(y_0)$ in
(\ref{equ:inner}) is not a standard MDP, because the square term
$(R_{0:T}-y_0)^2$ in the objective (\ref{eq_pseudoMV}) is not additive
and we cannot separate (\ref{eq_pseudoMV}) into a recursive form.
Below, we show that by defining an augmented MDP, we can treat
(\ref{equ:inner}) as a
standard finite-horizon MDP with an extended state space. %The main
%idea is to take the square term $(R^0-y_0)^2$ as a whole since it is
%not separable.

We define a new MDP by tuple $\widetilde{\mathcal{M}}=\langle {\mathcal
T},
\tilde{\mathcal{S}},\tilde{\mathcal{A}},(\tilde{\mathcal{A}}(\tilde{s}) \subset \tilde{\mathcal{A}},
\tilde{s} \in \tilde{\mathcal{S}}),(\tilde{P}_t,t \in {\mathcal
T}),(\tilde{r}_t,t \in {\mathcal T})\rangle$ with a 2-dimensional
state space $\tilde{\mathcal{S}} := \mathcal{S} \times \mathbb{R}$,
where the first dimension is the state of the original MDP and the
second dimension represents the anticipation of accumulated rewards
from the current stage to the terminal stage $T$. The action space
$\tilde{\mathcal{A}}:=\mathcal{A}$ and the admissible action set
$\tilde{\mathcal{A}}(s,y):=\mathcal{A}(s)$, for any augmented state
$(s,y) \in \tilde{\mathcal{S}}$. Suppose the state is $(s_t,y_t) \in
\tilde{\mathcal{S}}$ at time $t \in \mathcal T$ and an action $a_t
\in \mathcal{A}(s_t)$ is adopted, the system will receive an
instantaneous reward $\tilde{r}_t(s_t,y_t,a_t):=r_t(s_t,a_t)$, and
then move to a new state $(s_{t+1},y_{t+1}) \in \tilde{\mathcal{S}}$
which is determined by the transition kernel $P_t$ and the one-step
reward $r_t$ as follows.
\begin{align}
    \nonumber
    s_{t+1} &\sim P_t(\cdot|s_t,a_t), \\
    y_{t+1} &= y_t-r_t(s_t,a_t). \nonumber
\end{align}
That is,
%$\tilde{P}_t(s',
%y'|s,y,a)=P_t(s'|s,a)\delta_{y-r_t(s,a)}(y')$, where
%$\delta_{\cdot}(\cdot)$ denotes the Dirac measure.
$\tilde{P}_t(s', y'|s,y,a) :=
P_t(s'|s,a)\mathbb{I}_{\left\{y-r_t(s,a)\right\}}(y')$, where
$\mathbb I_{\left\{y-r_t(s,a)\right\}}(\cdot)$ denotes an indicator
function. The terminal reward of this MDP $\widetilde{\mathcal{M}}$ is
\begin{equation*}
\tilde{r}_{T}(s_T,y_T) := -\lambda y_T^2.
\end{equation*}
We denote by $\tilde{\mathcal U}^{\rm HR}$ the set of all
history-dependent randomized policies $\tilde{u}=(\tilde{u}_t; t \in
\mathcal T)$, where $\tilde{u}_t$ is a probability measure on
$\mathcal{A}$ given history $\tilde{h}_t =
\left\{s_0,y_0,a_0,\ldots,s_t,y_t\right\}$. Similarly, we denote by
$\tilde{\mathcal U}^{\rm MR}$ and $\tilde{\mathcal U}^{\rm MD}$ the
sets of all Markov randomized policies and Markov deterministic
policies of the MDP $\widetilde{\mathcal{M}}$, respectively. Given
initial state $(s_0,y_0) \in \tilde{\mathcal{S}}$ and policy
$\tilde{u} \in \tilde{\mathcal U}^{\rm HR}$, we denote by
$\mathbbm{P}_{(s_0,y_0)}^{\tilde{u}}$ the unique probability measure
on the space of trajectories of augmented states and actions and by
$\mathbbm{E}_{(s_0,y_0)}^{\tilde{u}}$ the expectation operator
corresponding to $\mathbbm{P}_{(s_0,y_0)}^{\tilde{u}}$.

With the definition of this new finite-horizon augmented MDP
$\widetilde{\mathcal{M}}$, we focus on the criterion of expected total
rewards. Given an initial state $(s_0,y_0) \in \tilde{\mathcal{S}}$
and a policy $\tilde{u} \in \tilde{\mathcal U}^{\rm HR}$, we define
the $T$-horizon expected rewards as below.
\begin{eqnarray*}
V_0^{\tilde{u}}(s_0,y_0) &:=& \mathbbm{E}_{(s_0,y_0)}^{\tilde{u}}
\big[\sum_{t=0}^{T-1}\tilde{r}_{t}(s_t,y_t,a_t) + \tilde{r}_{T}(s_T,y_T)] \\
&=& \mathbbm{E}_{(s_0,y_0)}^{\tilde{u}}
\big[\sum_{t=0}^{T-1} r_t(s_t,a_t) -\lambda y_T^2] \\
&=& \mathbbm{E}_{(s_0,y_0)}^{\tilde{u}}
\big[R_{0:T}-\lambda\big(R_{0:T}-y_0\big)^2\big],
\end{eqnarray*}
where the last equality recursively utilizes the fact
$y_{t+1}=y_t-r_t(s_t,a_t)$. It is interesting to find that the above
expected total rewards is exactly the same as the pseudo
mean-variance defined in (\ref{eq_pseudoMV}). The objective of MDP
$\widetilde{\mathcal M}$ is to maximize the above expected total
rewards for each initial state $(s_0,y_0) \in \tilde{\mathcal S}$,
i.e.,
\begin{flalign}\label{equ:newmdp}
&\widetilde{\mathcal M}: \hspace{4cm} V_0^*(s_0,y_0) =
\sup\limits_{\tilde{u} \in \tilde{\mathcal U}^{\rm HR}}
V_0^{\tilde{u}}(s_0,y_0), \quad (s_0,y_0) \in \tilde{\mathcal
S},&
\end{flalign}
where $V_0^*(\cdot,\cdot)$ is called the optimal value function and
we denote by $\tilde{u}^* \in \tilde{\mathcal U}^{\rm HR}$ an
optimal policy if it attains the above optimal value, i.e.,
$V_0^{\tilde{u}^*}(\cdot,\cdot) = V_0^*(\cdot,\cdot)$. It is worth
noting that $\widetilde {\mathcal M}$ in (\ref{equ:newmdp}) is a
standard finite-horizon MDP with the expectation criterion for total
rewards, which can be solved directly by dynamic programming. In
contrast, $\hat {\mathcal M}(y_0)$ in (\ref{equ:inner}) is an MDP
problem with the pseudo mean-variance criterion, to which the
classical dynamic programming principle is not applicable.

Next, we establish the relationship between the two MDP problems
$\widetilde{\mathcal M}$ and $\hat {\mathcal M}(y_0)$, as stated by
Theorem~\ref{thm:relation} below.

\begin{thm}\label{thm:relation}
For each $y_0 \in \mathbb{R}$ and $\tilde{u}=(\tilde{u}_t;t \in
\mathcal T) \in \tilde{\mathcal U}^{\rm HR}$, there exists a policy
$u = (u_t;t \in \mathcal T) \in \mathcal U^{\rm HR}$ such that
\begin{equation}\label{equ:relation}
    \hat J_0^{u}(s_0,y_0) =  V_0^{\tilde{u}}(s_0,y_0), \quad\forall s_0 \in \mathcal{S}.
\end{equation}
And further
\begin{equation}\label{equ:opt_relation}
    \hat J_0^*(s_0,y_0) =  V_0^*(s_0,y_0),\quad\forall (s_0,y_0) \in \tilde{\mathcal{S}}.
\end{equation}
\end{thm}

%The proof can be found in the Appendix.
Theorem~\ref{thm:relation} implies that the inner pseudo MV-MDP
$\hat {\mathcal M}(y_0)$ in (\ref{equ:inner}) can be converted to a
standard MDP $\widetilde {\mathcal M}$ in (\ref{equ:newmdp}) with the
criterion of expected total rewards, which can be solved by dynamic
programming. To this end, we denote by $\mathcal B(\tilde{\mathcal{S}})$ the
space of all bounded functions on $\tilde{\mathcal{S}}$ and define
%an operator $\mathbb L_t^\varphi: \mathcal B(\tilde{\mathcal{S}}) \rightarrow
%\mathcal B(\tilde{\mathcal{S}})$ for stochastic kernel $\varphi$ on
%$\mathcal{A}$ given $\tilde{\mathcal{S}}$ and $t \in \mathcal T$ by
%\begin{equation}\label{equ:opera}
%   \mathbb L_t^\varphi v(s,y) := \sum\limits_{a \in \mathcal{A}(s)}
%    \varphi(a|s,y)\big[r_t(s,a)+\sum\limits_{s' \in
%        \mathcal{S}}P_t (s'|s,a)v(s',y-r_t(s,a))\big],
%    \quad \forall v \in \mathcal B(\tilde{\mathcal{S}}), (s,y) \in \tilde{\mathcal{S}},
%    \end{equation}
%where we can view $\varphi$ as a decision rule at time $t$.
%Similarly, we also define
an operator $\mathbb L^*_t: \mathcal B(\tilde{\mathcal{S}})
\rightarrow \mathcal B(\tilde{\mathcal{S}})$ for $t \in \mathcal T$ by
\begin{equation}\label{equ:opt_opera}
\mathbb L^*_t v(s,y) := \max\limits_{a \in \mathcal{A}(s)}
\left\{r_t(s,a)+\sum\limits_{s' \in \mathcal{S}}P_t
(s'|s,a)v(s',y-r_t(s,a))\right\},\quad v \in
\mathcal B(\tilde{\mathcal{S}}), (s,y) \in \tilde{\mathcal{S}}.
\end{equation}
For notational simplicity, we further denote by
\begin{equation}
    \nonumber
     V_{t}^{\tilde{u}}(s_t,y_t) := \mathbbm{E}_{(s_0,y_0)}^{\tilde{u}}
    \big[R_{t:T}-\lambda\big(R_{t:T}-y_t\big)^2|s_t,y_t\big],
    \quad (s_t,y_t) \in \tilde{\mathcal S},~t \in \mathcal T
\end{equation}
and
\begin{equation}\label{equ:value_t}
    V_{t}^*(s_t,y_t) := \sup\limits_{\tilde{u} \in \tilde{\mathcal U}^{\rm MR}}V_{t}^{\tilde{u}}(s_t,y_t),
    \quad (s_t,y_t) \in \tilde{\mathcal S},~t \in \mathcal T
\end{equation}
the expected total rewards under Markov randomized policy $\tilde{u}
\in \tilde{\mathcal U}^{\rm MR}$ and the optimal value function from
stage $t$ to terminal stage $T$, respectively.

%For the standard MDP $\tilde {\mathcal M}$ in (\ref{equ:newmdp})
%with finite-horizon expected total reward criterion, it follows from
%Theorem~4.2.1 of \cite{Puterman94} that $V_0^{\tilde{u}}(s_0,y_0)$
%with randomized Markov policy $\tilde{u} \in \tilde{\mathcal U}^{\rm
%RM}$ can be evaluated by successively conducting a series of
%operators $\left\{L_t^{\tilde u_t}, t \in \mathcal T\right\}$,
%starting from the initial value $V_T^{\tilde{u}}(s_T,y_T):=-\lambda
%y_T^2$. We summarize this result in Lemma \ref{lem:dp} as follows.
%
%\begin{lem}\label{lem:dp}
%    Given  $\tilde u = (\tilde u_t,t \in \mathcal T) \in \tilde{\mathbb U}^{\rm RM}$, then we have
%    \begin{equation}\label{equ:dp}
%        V_t^{\tilde u}=L_t^{\tilde u_t}V_{t+1}^{\tilde u}, \quad\forall t \in \mathcal T.
%    \end{equation}
%\end{lem}
%
%The next theorem establishes the dynamic programming  to solve the
%standard MDP $\tilde {\mathcal M}$ in (\ref{equ:newmdp}) and derives
%the optimal policy of the inner pseudo MV-MDP $\hat {\mathcal
%M}(y_0)$ in (\ref{equ:inner}).

For the standard MDP $\widetilde {\mathcal M}$ in (\ref{equ:newmdp})
with finite-horizon expected total reward criterion, it is
straightforward that the optimal value function defined in
(\ref{equ:value_t}) can be solved by successively conducting a
series of operators $\left\{\mathbb L^*_t; t \in \mathcal
T\right\}$, starting from the initial value
$V^*_T(s_T,y_T):=-\lambda y_T^2$. We then establish the optimal
policy of the inner pseudo MV-MDP $\hat {\mathcal M}(y_0)$ in
(\ref{equ:inner}) by utilizing the optimal policy of the standard
MDP $\widetilde {\mathcal M}$. We summarize this result in
Theorem~\ref{thm:opt} as follows.
\begin{thm}\label{thm:opt}
The function sequence $\left\{V_t^*; t \in \mathcal T\right\}$ defined in (\ref{equ:value_t}) satisfies
\begin{equation}\label{equ:opt_dp}
    V_t^*=\mathbb L^*_t V_{t+1}^*,\quad\forall t \in \mathcal T \text{ with } V_T^*(s_T,y_T) := -\lambda y_T^2.
\end{equation}
In addition, there exists $a_t^*(s_t,y_t) \in \mathcal A(s_t)$ that attains the maximum
in $\mathbb L^*_t V_{t+1}^*(s_t,y_t)$, we have
\begin{itemize}
  \item [$(a)$] the Markov deterministic policy $\tilde{u}^*=(\tilde{u}_t^*;t \in \mathcal T) \in \tilde{\mathcal U}^{\rm {MD}}$ with $\tilde{u}_t^*(s_t,y_t) = a_t^*(s_t,y_t)$ is an optimal policy for the standard MDP $\widetilde {\mathcal M}$ in (\ref{equ:newmdp}).
  \item [$(b)$] given $y_0$, the history-dependent deterministic policy $\hat u^*=(\hat u^*_{t}; t \in \mathcal T) \in \mathcal U^{\rm {HD}}$ with $\hat u^*_{t}(s_0,a_0,\ldots,s_t) := \tilde{u}_t^*(s_t,{y_0}-\sum\limits_{\tau=0}^{t-1} r_\tau(s_\tau,a_\tau))$ is an optimal policy for the inner pseudo MV-MDP $\hat {\mathcal M}(y_0)$ in (\ref{equ:inner}).
\end{itemize}
\end{thm}

Therefore, with Theorems~\ref{thm:relation} and \ref{thm:opt}, the
inner pseudo MV-MDP $\hat{\mathcal M}(y_0)$ in (\ref{equ:inner}) can
be solved by executing dynamic programming (\ref{equ:opt_dp}) with
$\hat J_0^* = V_0^*$, and the optimal policy of $\hat{\mathcal
M}(y_0)$ can be determined by that of $\widetilde{\mathcal M}$ in
(\ref{equ:newmdp}), as stated by part~(b) above. Furthermore, after
$\hat J_0^*(s_0,y_0)$ is obtained, the original problem MV-MDP
$\mathcal M$ in (\ref{equ:mv_mdp}) can be solved by the following
single parameter optimization problem
\begin{equation}\label{equ:outer}
    J_0^*(s_0) = \max\limits_{y_0 \in \mathbb{R}}\hat J_0^*(s_0,y_0),\quad s_0 \in \mathcal S.
\end{equation}
We derive Theorem~\ref{thm:opt_policy} to establish the optimal
policy for the MV-MDP $\mathcal M$ in (\ref{equ:mv_mdp}).
\begin{thm}\label{thm:opt_policy}
Suppose $y^*_0$ attains the maximum of (\ref{equ:outer}) and
$\tilde{u}^*=(\tilde{u}^*_t;t \in \mathcal T) \in \tilde{\mathcal
U}^{\rm {MD}}$ is an optimal policy for the standard MDP
$\widetilde{\mathcal M}$ in (\ref{equ:newmdp}), then the
history-dependent deterministic policy ${u}^{*}=({u}^{*}_t;t \in \mathcal T)\in
\mathcal U^{\rm {HD}}$ with
$u^{*}_{t}(s_0,a_0,\ldots,s_t):=\tilde{u}_t^*(s_t,y^*_0-\sum\limits_{\tau=0}^{t-1}r_\tau(s_\tau,a_\tau))$
is optimal for the MV-MDP $\mathcal M$ in (\ref{equ:mv_mdp}).
\end{thm}

\begin{remark}\rm{
(i) Theorem~\ref{thm:opt_policy} implies that the optimum of the
finite-horizon MV-MDP $\mathcal M$ in (\ref{equ:mv_mdp}) can be
attained by a \textit{history-dependent deterministic} policy in
$\mathcal{U}^{\rm HD}$, which is not Markovian since
$y_t:=y^*_0-\sum\limits_{\tau=0}^{t-1}r_\tau(s_\tau,a_\tau)$ relies
on the history rewards up to time $t$. Therefore, we cannot limit
our policy space to $\mathcal{U}^{\rm MD}$, which is different from
the ordinary MDPs \citep{Puterman94} or the long-run MV-MDPs where
Markov deterministic policies are  able to attain optimum
\citep{xia2016,xia2020}. Moreover, $\sup$ in all the previous
contents can be replaced by $\max$.

(ii) In the above MV-MDPs, the state and action spaces are supposed
to be discrete and finite. %Section~\ref{sec:opt} establishes the
%equivalence between the finite-horizon MV-MDPs and the
%finite-horizon standard MDPs with the help of state augment.
Furthermore, all the results in Section~\ref{sec:opt} can be
parallel extended to continuous state and action spaces by replacing
transition probability function with \emph{transition density
function} and adding the so-called \emph{measurable selection
condition} (for example, the compactness assumption on action space
and the continuity assumption on transition function and reward
function) to ensure the existence of an optimal deterministic policy
as that in finite-horizon standard MDPs (see Chapter~3 of
\cite{Hernandez-Lerma1996} for instance).

(iii) In many applications, such as portfolio selection and
inventory management, system stochasticity is captured by a random
variable $\xi_t$ and the evolution of states is specified by a
difference equation $s_{t+1}=f_t(s_t,a_t,\xi_t)$, which is commonly
adopted in stochastic control. Such kind of models can be viewed as
a special case of our MDP model (see Chapter 2 of
\cite{Hernandez-Lerma1996} for instance), and our main results can
be extended to these stochastic control models, as discussed later
in Sections $\ref{sec:alg}$ and $\ref{sec:exp}$.}
\end{remark}

\section{Algorithm}\label{sec:alg}

With the main results in Section~\ref{sec:opt}, the original
finite-horizon MV-MDP in (\ref{equ:mv_mdp}) is converted to a
bilevel MDP as follows.
\begin{equation}\label{equ:outer2}
J_0^*(s_0) = \max\limits_{u \in {\mathcal U}^{\rm {HR}}} J_0^{u}(s_0)
= \max\limits_{y_0 \in \mathbb{R}}\max\limits_{u \in {\mathcal
U}^{\rm {HD}}}\hat {J}_0^{u}(s_0,y_0),\quad s_0 \in \mathcal
S.
\end{equation}
Although the inner problem is equivalent to a standard MDP with
augmented state, enumerating every possible $y_0 \in \mathbb R$ and
solving the associated inner problem is computationally intractable.
In this section, we aim to develop an efficient algorithm to solve
(\ref{equ:outer2}).

It is worth noting that the maximum of the outer level optimization
problem (\ref{equ:outer2}) is attained at $y^*_0=\mu_0^{u^*}(s_0)$,
i.e., if optimal policy is given as $u^* \in \mathcal U^{\rm {HD}}$,
$y^*_0$ in (\ref{equ:outer2}) equals the mean reward
$\mu_0^{u^*}(s_0)$ of this MDP with policy $u^*$. Thus, we can
restrict $y_0$ to a much smaller domain $\left\{\mu_0^u(s_0): u \in
\mathcal U^{\rm HD}\right\} \subset [T\underline r , T\bar r] =:
\mathcal{Y}$. Therefore, the bilevel MDP (\ref{equ:outer2}) can be
rewritten as
\begin{equation}\label{equ:bilev2}
        J_0^*(s_0) = \max_{y_0 \in \mathcal Y}\max_{u \in \mathcal U^{\rm HD}}\hat J_0^u(s_0,y_0), \quad s_0 \in \mathcal S.
\end{equation}

Although the domain of $y_0$ is reduced from $\mathbb R$ to a
bounded space $\mathcal Y$, the computation of solving
(\ref{equ:bilev2}) is inefficient yet. To resolve this challenge, we
need to further study the property of the bilevel MDP (\ref{equ:bilev2}). We find that $\mathcal Y$ can be divided
into finitely many intervals, where in each interval $\mathcal Y_i
\subset \mathcal Y$, the inner pseudo MV-MDPs $\left\{\hat{\mathcal
M}(y_0); y_0 \in \mathcal Y_i\right\}$ in (\ref{equ:inner}) can
retain the same optimal policy, as stated in
Theorem~\ref{thm:divid}.

%Since $\mathcal{Y}$ is bounded, we can use discretization to
%approximately solve (\ref{equ:bilev2}). Specifically, we first
%discretize  $\mathcal{Y}$ to a discrete space $\hat{\mathcal{Y}}$.
%Then, we solve a series of inner pseudo MV-MDPs $\hat{\mathcal
%M}(y_0)$ for each $y_0 \in \hat{\mathcal Y}$ and obtain the
%corresponding optimal policy $\hat u^*$ and optimal value function
%$\hat{J}_0^{\hat u^*}(s_0,y_0)$. Finally, we choose the one with the
%maximal pseudo mean-variance, i.e., $u^* \in \argmax\limits_{\hat
%u^*}\hat{J}_0^{\hat u^*}(s_0,y_0)$, which approximates the optimal
%policy of the original MV-MDP (\ref{equ:mv_mdp}).
%
%This discretization method is theoretically feasible but lacks
%computational efficiency, which is difficult to apply in practice
%especially for large scale state space $\mathcal{S}$ and long
%horizon $T$. It is desirable to develop efficient algorithms to
%solve the finite-horizon MV-MDP (\ref{equ:mv_mdp}).

\begin{thm}\label{thm:divid}
Given $s_0 \in \mathcal S$, there exist a sequence
$\left\{y^0,y^1,\ldots,y^n,y^{n+1}\right\}$ with $\underline
r=y^0<y^1<\ldots<y^n<y^{n+1}=\bar r$ and a sequence of deterministic
policies $\left\{\hat u^0_*,\hat u^1_*,\ldots,\hat u^n_*\right\}$
such that
\begin{equation}\nonumber
\hat J_0^{*}(s_0,y_0) = \hat J_0^{\hat u^k_*}(s_0,y_0), \quad\forall
y_0 \in [y^k,y^{k+1}],
\end{equation}
for a given $k \in \left\{0,1,\ldots,n\right\}$.
\end{thm}

%\begin{proof}
%The proof is presented in Appendix \ref{app:thmdivid}.
%\end{proof}

Based on Theorem~\ref{thm:divid}, we give the definition of
\emph{break points}, which play an important role in our
algorithm.

\begin{definition}\label{def:cri}
We call $y^c \in \mathcal Y$ a break point if there exist
$y_1,y_2$ with $y_1<y^c<y_2$ such that the pseudo MV-MDPs
$\left\{\hat{\mathcal M}(y); y \in [y_1,y^c]\right\}$ have the same
optimal policy, while this policy is not optimal for pseudo MV-MDPs
$\left\{\hat{\mathcal M}(y); y \in (y^c,y_2]\right\}$.
\end{definition}

Without loss of generality, we assume that
$\left\{y^1,\ldots,y^n\right\}$ is the set of all break points.
As a consequence of Theorem~\ref{thm:divid}, we prove that the
optimal value function $\hat J_0^*(s_0,y_0)$  of the pseudo MV-MDP
(\ref{equ:inner}) is divided into quadratic concave segments by
break points, as stated in Theorem~\ref{thm:piece_conca} and
illustrated in Figure~\ref{fig:pie_conca}.

\begin{thm}\label{thm:piece_conca}
Given $s_0 \in \mathcal S$, the optimal value function $\hat
J_0^*(s_0,y_0)$ is piecewise quadratic concave with respect to
$y_0$, and it is divided into quadratic concave segments by break
points.
\end{thm}

\begin{figure}[htbp]
\centering
\includegraphics[width=0.65\columnwidth]{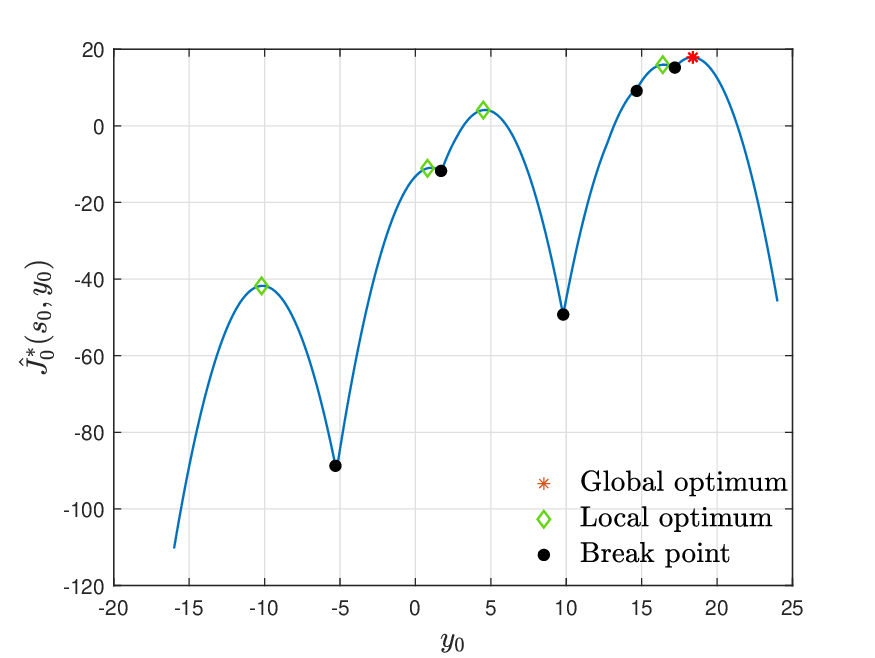}
\caption{The piecewise quadratic concave structure of optimal value
function $\hat J_0^{*}(s_0,y_0)$.} \label{fig:pie_conca}
\end{figure}

Theorem~\ref{thm:piece_conca} implies that the outer optimization
problem (\ref{equ:outer}) is not a convex optimization problem,
there may exist multiple local optima. In what follows, we develop a
policy iteration type algorithm to efficiently find a local optimum
of (\ref{equ:outer}), which also attains a locally optimal policy
for the original finite-horizon MV-MDP (\ref{equ:mv_mdp}). The basic
idea is that we solve the bilevel MDP
(\ref{equ:bilev2}), i.e., $J^*_0(s_0) = \max\limits_{y_0 \in
\mathcal Y}\max\limits_{u \in \mathcal U^{\rm HD}}\hat
J_0^u(s_0,y_0)$, by alternatingly maximizing between pseudo mean
$y_0$ and policy $u$. For a given policy $u$, we can attain the
maximum of the outer level problem by setting
$y_0=\mathbbm{E}_{s_0}^u[R_{0:T}]$. Then we fix this $y_0$ and optimize
the inner pseudo MV-MDP $\hat{\mathcal M}(y_0)$ to derive a new
policy $u'$, i.e.,
\begin{equation}
    u' \in \argmax_{u \in \mathcal U^{\rm HD}} \hat J_0^u(s_0,y_0).
\end{equation}
We can repeat this procedure by updating $y_0$ using this new policy
$u'$. We prove that such procedure can strictly improve the
mean-variance value function $J^{u}_0(s_0)$. We can observe that
this procedure usually converges fast and the performance
improvement of the first few iterations is significant, which is
similar to the policy iteration in classical MDPs.
When the iteration procedure stops at a pair $(u^*, y^*_0)$, it must
satisfy \textit{fixed point equations}
\begin{equation}\label{eq_fixedpoint}
\begin{array}{rcl}
   y^*_0 &=& \mathbbm{E}_{s_0}^{u^*}[R_{0:T}],\\
   u^* &\in& \argmax\limits_{u \in \mathcal {U}^{\rm HD}} \hat J_0^u(s_0,y^*_0).
\end{array}
\end{equation}
We can prove that such policy $u^*$ is locally optimal in a mixed
policy space specified. Moreover, we can further improve these
converged policies when the associated $y^*_0$ coincides with
break points, and these refined pseudo means are also locally
optimal in the space of $\mathcal Y$ as shown in
Figure~\ref{fig:pie_conca}.
 The detailed procedure is described in
Algorithm~\ref{alg:local} and the flowchat of the algorithm is
illustrated by Figure~\ref{fig:algo1}.

\begin{algorithm}
    \caption{An iterative algorithm to find local optima of finite-horizon MV-MDPs}\label{alg:local}
    \begin{algorithmic}[1]
        \Require {MDP parameters $\mathcal{M}
            = \langle \mathcal T, \mathcal S, \mathcal A, (\mathcal{A}(s) \subset \mathcal A, s \in \mathcal{S}),
            (P_t, t \in \mathcal T), (r_t, t \in \mathcal T)  \rangle$}

        \Ensure {A locally optimal policy $u^*$}

        \State \textit{Initialization}:
        Arbitrarily choose a policy $u^{(0)} \in \mathcal U^{\rm HD}$, $k \gets 0$.

        \While{$u^{(k)} \neq u^{(k-1)}$}

        \State \textit{Policy Evaluation}: For the initial state $s_0 \in \mathcal S$, compute the pseudo mean
        \begin{equation}
            \nonumber
            y^{(k)}_0 =  \mathbbm{E}^{u^{(k)}}_{s_0}[R_{0:T}].
        \end{equation}

        \State \textit{Policy Improvement}: Solve the pseudo MV-MDP $\hat{\mathcal M}(y^{(k)}_0)$ in (\ref{equ:inner}), or equivalently the augmented MDP $\widetilde{\mathcal M}$ in (\ref{equ:newmdp}) with initial state $(s_0,y^{(k)}_0)$ by using dynamic programming (\ref{equ:opt_dp}), and obtain the inner optimal policy
        $\tilde u^*=(\tilde u^*_t;t \in \mathcal T) \in \tilde{\mathcal U}^{\rm MD}$. Generate a new policy $u'=(u'_t; t \in \mathcal T) \in \mathcal U^{\rm HD}$ by Theorem~\ref{thm:opt}
        \begin{equation}\label{equ:gene}
            u'_t(s_0,a_0,\ldots,s_t)=\tilde{u}_t^*(s_t, y^{(k)}_0-\sum\limits_{\tau=0}^{t-1}r_\tau(s_\tau,a_\tau)).
        \end{equation}
        \indent Keep $u'=u^{(k)}$ if possible, to avoid policy oscillations.

        \State \textit{Parameters Update}: $u^{(k+1)} \gets u'$, $k \gets
        k+1$.

        \EndWhile

        \If{$y^{(k)}_0$ is a break point}
        \State Go to line~4 (Policy Improvement).
        Choose a new inner optimal policy $\tilde u^{*'} \neq \tilde u^*$ and generate a policy $u{''}$ with $\tilde u^{*'}$ in lieu of $\tilde u^*$ in (\ref{equ:gene}) such that $\hat J^{u{''}}_0(s_0,\mu^{u{''}}_0(s_0))>\hat J^{u'}_0(s_0,\mu^{u'}_0(s_0))$.
        \State $u^{(k+1)} \gets u{''}$, $k \gets k+1$, and go to line~3 (Policy Evaluation).
        \EndIf

        \State \textbf{return}  $u^{(k)}$ %$u, y=(y(s_0), s_0 \in \mathcal S)$
        %   \EndProcedure
    \end{algorithmic}
\end{algorithm}
From Figure~\ref{fig:algo1}, we can see that the pseudo mean
$y^{(k)}_0$ and the policy $u^{(k)}$ are updated alternatingly.
Next, we will show that the sequence of $\left\{(u^{(k)}, y^{(k)}_0);k\ge 0\right\}$ will converge to a fixed point solution $(u^*,y^*_0)$
to (\ref{eq_fixedpoint}), and the associated sequence of
mean-variance value $\left\{J^{u^{(k)}}_0(s_0); k\ge 0\right\}$ is
monotonically increasing.
To further characterize the local
optimality of the converged pseudo mean $y_0^*$ and policy $u^*$, we
will show that the associated $\hat{J}_0^*(s_0,y_0)$ attains the
local optimum at $y_0^*$ in the pseudo mean space $\mathcal Y$ and
the associated $J^{u}_0(s_0)$ attains the local optimum at $u^*$ in
the sense of a well-specified policy space.
%The converged policy
%$u^*$ is locally optimal, i.e., the associated $J^{u}_0(s_0)$
%attains local optima at each solution $u^*$ in the sense of a
%well-specified solution space.
\begin{figure}[htbp]
    \centering
    \includegraphics[width=1\columnwidth]{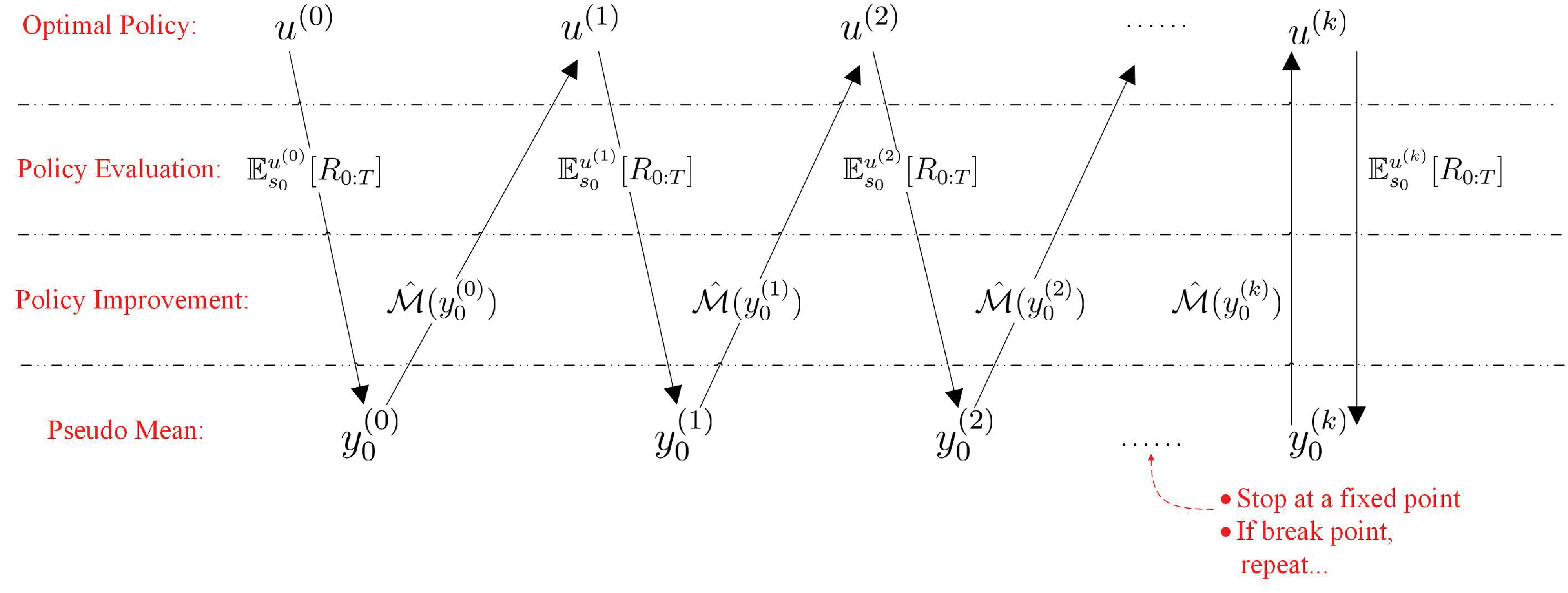}
    \caption{Flowchat illustration of Algorithm~\ref{alg:local}.}
    \label{fig:algo1}
\end{figure}

First, we introduce the concept of mixed policy. For any two
deterministic policies $u, u' \in \mathcal U^{\rm HD}$ and a constant
$\delta \in [0,1]$, we define $\delta_u^{u'}:=(1-\delta) u + \delta
u'$ as a mixed policy between $u$ and $u'$, which adopts policy $u$
with probability $1-\delta$ and adopts policy $u'$ with probability
$\delta$. We denote by $\mathcal{U}^{\rm MIX}$ the space of all mixed
policies. Then, we give the definition of the so-called \emph{valid
pruned deterministic policy space} as follows.

\begin{definition}\label{def:valid_policy}
We call a policy space $\mathcal{U}_{\rm valid}^{\rm HD} \subseteq
\mathcal{U}^{\rm HD}$ a valid pruned deterministic policy space, if
the optimal policy of the finite-horizon MV-MDP (\ref{equ:mv_mdp})
can be obtained in $\mathcal{U}_{\rm valid}^{\rm HD}$.
\end{definition}

Next, we give the definition of \emph{local optimum} in a mixed
policy space as below.

\begin{definition}\label{def:local_opt}
Suppose $\mathcal{U}_{\rm valid}^{\rm HD}$ is a valid pruned
deterministic policy space, we call a deterministic policy $u \in
\mathcal{U}_{\rm valid}^{\rm HD}$ locally optimal in the mixed policy
space generated by $\mathcal{U}_{\rm valid}^{\rm HD}$, if there
exists a constant $\epsilon>0$ such that
\begin{equation*}
J_0^{u}(s_0) \ge J_0^{\delta_{u}^{u'}}(s_0), \quad\forall \delta \in
(0,\epsilon), u' \in \mathcal{U}_{\rm valid}^{\rm HD}, s_0 \in
\mathcal S.
\end{equation*}
Further, if the inequality is strict, $u$ is called a strictly local
optimum in the mixed policy space generated by $\mathcal{U}_{\rm
valid}^{\rm HD}$.
\end{definition}

With the above definition of local optimum, the convergence of
Algorithm~\ref{alg:local} is guaranteed by the following theorem.

\begin{thm}\label{thm:converge}
Algorithm~\ref{alg:local} converges to a fixed point solution
$(u^*,y^*_0)$ to (\ref{eq_fixedpoint}). Furthermore,
\begin{itemize}
\item[(i)] The policy space defined by
\begin{equation}\label{eq_barU}
\mathcal{U}_{\rm valid}^{\rm HD}(u^*) := \left\{u \in \mathcal{U}^{\rm
HD}: \hat {J}^{u}_0(s_0,y^*_0) \neq \hat {J}^{u^*}_0(s_0,y^*_0), ~
\exists s_0 \in \mathcal S\right\} \cup \left\{u^*\right\}
\end{equation}
is a valid pruned deterministic policy space.
\item[(ii)] Algorithm \ref{alg:local} converges to a strictly local optimum $u^*$ in the mixed policy
space generated by $\mathcal{U}_{\rm valid}^{\rm HD}(u^*)$ for the
finite-horizon MV-MDP (\ref{equ:mv_mdp}) with value function
$J^{u}_0(s_0)$, $\forall u \in~\left\{(1-\delta)u'+\delta
u'':~u',u'' \in \mathcal{U}_{\rm valid}^{\rm HD}(u^*),\delta \in
[0,1]\right\}$.
\item[(iii)] Algorithm~\ref{alg:local} converges to a local optimum $y^*_0$ in the real space for the pseudo MV-MDP (\ref{equ:inner}) with optimal value function $\hat J^*_0(s_0,y_0)$, $\forall y_0 \in \mathcal Y$.
\end{itemize}
\end{thm}

\begin{remark}{\rm
        (i) With the output policy $u^*$ by Algorithm~\ref{alg:local}, we
        can divide the deterministic policy space $\mathcal{U}^{\rm HD}$ into two parts: $\mathcal{U}_{\rm valid}^{\rm HD}(u^*)$ and $\bar{\mathcal{U}}_{\rm valid}^{\rm HD}(u^*):=\mathcal{U}^{\rm HD} \backslash \mathcal{U}_{\rm valid}^{\rm HD}(u^*)$.
        From the proof of Theorem~\ref{thm:converge} in Appendix, we can see
        that the mean and variance under each policy $u \in
        \bar{\mathcal{U}}_{\rm valid}^{\rm HD}(u^*)$ remain the same as those under policy
        $u^*$. Thus, we have
        \begin{equation*}
            J_0^u(s_0)= J_0^{u^*}(s_0), \quad\forall u \in \bar{\mathcal{U}}_{\rm valid}^{\rm HD}(u^*), s_0 \in \mathcal S.
        \end{equation*}
    We also have
        \begin{equation*}
            \frac{\partial J_0^{\delta_{u^*}^{u}}(s_0)}{\partial
                \delta}\Big|_{\delta=0} \le 0, \quad\forall u \in \mathcal U^{\rm
                HD}, s_0 \in \mathcal S,
        \end{equation*}
        which implies that $u^*$ is a \emph{stationary point} of the value function
        $J^{u}_0(s_0)$ in the mixed policy space $\mathcal{U}^{\rm MIX}$. Furthermore, by
        dividing $\mathcal U^{\rm HD}$ into $\mathcal{U}_{\rm valid}^{\rm HD}(u^*)$ and $\bar{\mathcal {U}}_{\rm valid}^{\rm HD}(u^*)$, we can verify that
        \begin{equation*}
            \begin{array}{rcl}
                \frac{\partial J_0^{\delta_{u^*}^{u}}(s_0)}{\partial
                    \delta}\Big|_{\delta=0} < 0, \quad\forall u \in \mathcal{U}_{\rm valid}^{\rm HD}(u^*), ~ \exists s_0 \in \mathcal S.
            \end{array}
        \end{equation*}
        Therefore, we can conclude that the output policy $u^*$ by Algorithm~\ref{alg:local}
        is \emph{globally optimal} in the deterministic policy space
        $\bar{\mathcal {U}}_{\rm valid}^{\rm HD}(u^*)$ and \emph{strictly locally optimal} in
        the mixed policy space generated by $\mathcal {U}_{\rm valid}^{\rm HD}(u^*)$. The
        relation of these policy spaces is illustrated by
        Figure~\ref{fig:local_global}.
        \begin{figure}[htbp]
            \centering
            \includegraphics[width=0.85\columnwidth]{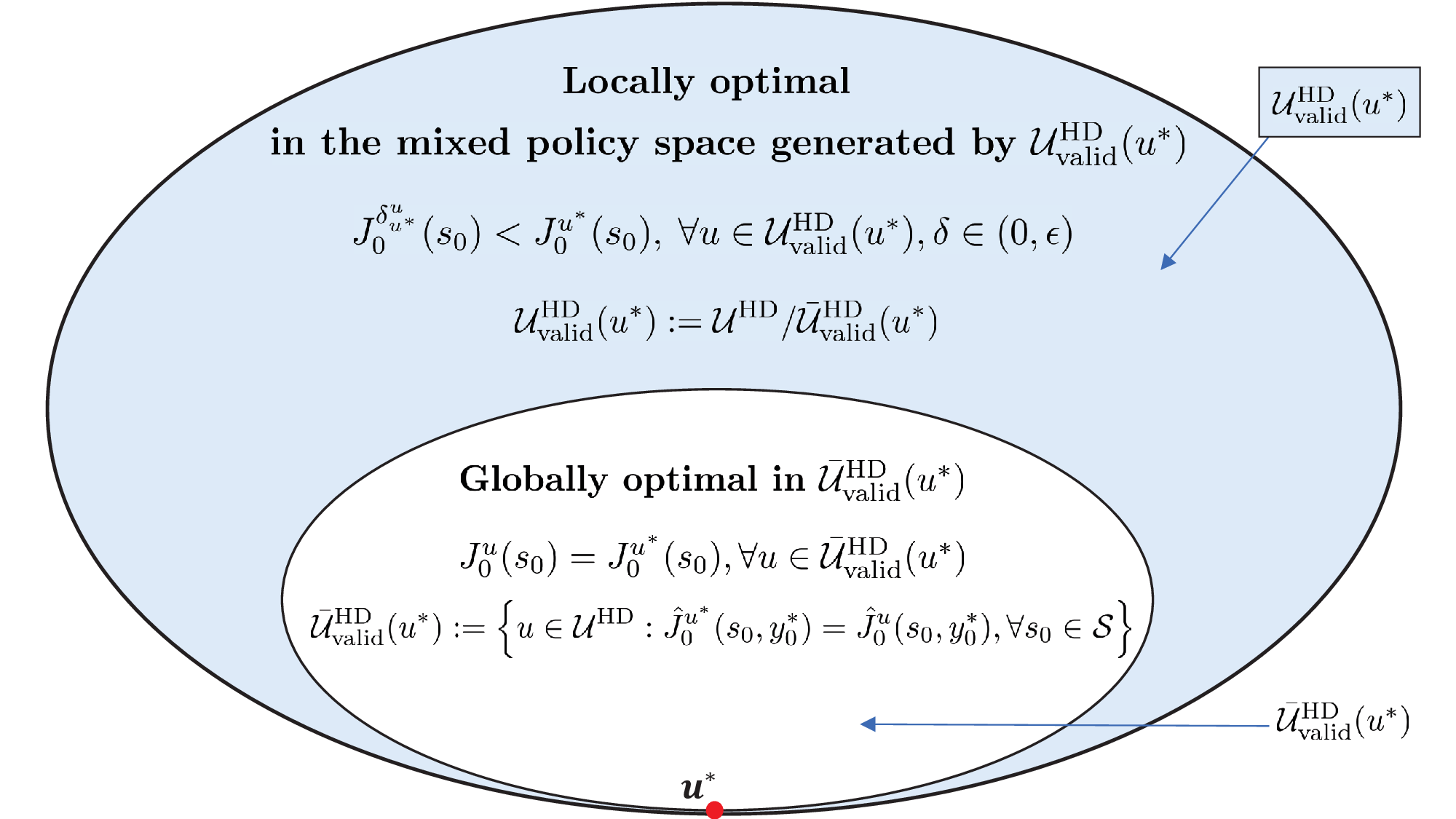}
            \caption{Relation of the optimal policy $u^*$ by
                Algorithm~\ref{alg:local} in different policy spaces.}
            \label{fig:local_global}
        \end{figure}

        (ii) In fact,
        $\bar{\mathcal{U}}_{\rm valid}^{\rm HD}(u^*)$ usually has very few elements, since it
        requires  $J_0^u(s_0)= J_0^{u^*}(s_0)$ and
        $\mu_0^u(s_0)=\mu_0^{u^*}(s_0)$, $\forall u \in \bar{\mathcal{U}}_{\rm valid}^{\rm HD}(u^*), s_0 \in \mathcal S$. We further observe that
        $\bar{\mathcal{U}}_{\rm valid}^{\rm HD}(u^*)$ is empty in most of numerical examples.
        Therefore, we may expect that $\mathcal{U}_{\rm valid}^{\rm HD}(u^*) =
        \mathcal{U}^{\rm HD}$ in most cases and
        Algorithm~\ref{alg:local}
        converges to a strictly local optimum in the mixed policy space $\mathcal{U}^{\rm MIX}$.

%        {\color{red}
%           (iii) Since $u^*$ attains the maximization of the pseudo MV-MDP $\hat{\mathcal M}(y_0^*)$, it satisfies the following Bellman equation
%           \begin{equation}\label{eq:local_bellman}
%               \hat{J}_0^{u^*}(s_0,y_0^*)=\max\limits_{a \in \mathcal A(s_0)}\left\{r_0(s_0,a)+\sum\limits_{s_1 \in \mathcal S}P(s_1|s_0,a)\hat{J}_1^{u^*}(s_1,y_0^*-r_0(s_0,a))\right\},
%           \end{equation}
%           where $\hat{J}_1^{u^*}(s_1,y_0^*-r_0(s_0,a))\cdots$.
%           It implies that a locally optimal policy $u^*$ satisfies the \emph{Bellman locally optimal equation} \eqref{eq:local_bellman}.
%           If $u^*$ is further a globally optimal policy,
%        }
    }
\end{remark}

It is known that policy iteration usually has a fast convergence in
classical MDPs, although its complexity analysis is still an open
question \citep{Littman95}. Since Algorithm~\ref{alg:local} is of a
form of policy iteration, it is expected that
Algorithm~\ref{alg:local} also has a fast convergence in practice,
which is demonstrated by examples in Section~\ref{sec:exp}. As
illustrated in Figure~\ref{fig:pie_conca}, the optimal pseudo
mean-variance $\hat{J}^*_0(s_0,y_0)$ is piecewise quadratic concave
with $y_0$, which leads to a local convergence guaranteed by
Theorem~\ref{thm:converge}. If we can find a condition under which
the function $\hat{J}^*_0(s_0,y_0)$ is concave with $y_0$, the
global convergence of Algorithm~\ref{alg:local} can be guaranteed,
which is stated by Theorem~\ref{thm:global} below.

\begin{thm}\label{thm:global}
If the following two conditions hold at each $t \in \mathcal T$,
    \begin{itemize}
        \item[(i)]({\bf convexity}) Both $\mathcal S$ and $\mathcal A$ are convex real number spaces, $\mathcal A$ is compact; the feasible space of state-action pairs $\mathcal K$ is a convex set;
        \item[(ii)]({\bf linearity}) Both the state transition function $s_{t+1}=f_t(s_t,a_t,\xi_t)$ and the reward function $r_t(s_t,a_t,\xi_t)$ are linear to state $s_t \in \mathcal S$ and action $a_t \in \mathcal A$, that is,
        \begin{eqnarray*}
            f_t(s_t,a_t,\xi_t)&=&f_{t,1}(\xi_t) s_t +
            f_{t,2}(\xi_t)a_t+f_{t,3}(\xi_t),\\
            r_t(s_t,a_t,\xi_t)&=&r_{t,1}(\xi_t) s_t +
            r_{t,2}(\xi_t)a_t+r_{t,3}(\xi_t),
        \end{eqnarray*}
where $\xi_t$ is a random variable capturing all the stochasticity
of the system and is defined with support on $\mathcal X$ and
distribution $q_t$, $\left\{(f_{t,i},r_{t,i}),i=1,2,3\right\}$ are
functions on $\mathcal X$.
    \end{itemize}
Then Algorithm \ref{alg:local} converges to the global optimum.
\end{thm}

\begin{remark}{\rm
(i) In Condition~(ii), we treat the next state $f_t(s_t,a_t,\xi_t)$
and the reward $r_t(s_t,a_t,\xi_t)$ as random variables, which can
be unified with the MDP models used in Section~\ref{sec:model},
where the transition probability and the one-step reward can be
determined by $f_t$ and $r_t$, respectively. In this sense, the
Bellman operator defined in (\ref{equ:opt_opera}) takes a slightly
different form
\begin{equation}\nonumber
\mathbb L^*_t v(s,y)=\max\limits_{a \in \mathcal{A}(s)}\int_{\mathcal X}
\left\{r_t(s,a,x)+v(f_t(s,a,x),y-r_t(s,a,x))\right\}q_t(dx),\quad
v \in \mathcal B(\tilde{\mathcal{S}}), (s,y) \in \tilde{\mathcal{S}}.
\end{equation}
All the results in Sections~\ref{sec:model} $\sim$ \ref{sec:alg}
still hold.

(ii) In the proof of Theorem~\ref{thm:global}, we need the concavity
of $\int_{\mathcal
X}V^*_{t+1}(f_{t+1}(s_t,a_t,x),y_t-r_t(s_t,a_t,x))q_t(dx)$ with
respect to $(s,a,y) \in \mathcal K \times \mathcal Y$. Therefore,
the state and action spaces are supposed to be continuous. For
continuous state and action spaces, Algorithm~\ref{alg:local} still
converges to a local optimum by using the \emph{monotone convergence
theorem}, which is consistent to the case of discrete and finite
spaces.}
\end{remark}

For these MV-MDPs with linear transition and linear reward, we
further find some structural properties that can speed up
Algorithm~\ref{alg:local}, as stated in
Theorem~\ref{thm:structure_mv} and Remark~\ref{rem:structure}.

\begin{thm}\label{thm:structure_mv}
Suppose the convexity and linearity conditions in
Theorem~\ref{thm:global} hold, then the optimal pseudo mean $y^*_0$
is linear to $s_0$, that is, $y^*_0 = k_1 s_0 + k_0$ for some real
numbers $k_0, k_1$ independent of $s_0$.
%   \begin{itemize}
%       \item[(i)] The optimal pseudo mean $y^*_0$ is linear with $s_0$, that is, $y^*_0=k_1 s_0 + k_0$ for some real numbers $k_0, k_1$ independent of $s_0$;
%       \item[(ii)] The optimal mean-variance $J_0^*(s_0)$ is linear with $s_0$, that is, $J_0^*(s_0)=k_1 s_0 + k_0'$, where $k_1$ is the same real number as in (i).
%   \end{itemize}
    \end{thm}

\begin{remark}\label{rem:structure}\rm{
To obtain the global convergence of Algorithm~\ref{alg:local},
Theorem~\ref{thm:global} requires that the state space is continuous
real number space, which is infinite. It is inefficient to traverse
each initial state in Algorithm~\ref{alg:local}. However,
Theorem~\ref{thm:structure_mv} implies that it is sufficient to
implement Algorithm~\ref{alg:local} for only two initial states.
Since $y^*_0$ is linear to $s_0$, the optimal pseudo mean for other
initial states can be directly computed by using $y^*_0 = k_0 + k_1
s_0$. Therefore, we only need to implement Algorithm~\ref{alg:local}
for two initial states $s_0^1,s_0^2 \in \mathcal S$ and further
directly solve the standard MDP $\widetilde{\mathcal M}$ with initial
state $(s_0,y^*_0)$ for other $s_0 \in \mathcal S \backslash
\left\{s_0^1,s_0^2\right\}$.}
\end{remark}

In practice, many linear control models satisfy the two conditions
in Theorem~\ref{thm:global}. For example, $s_{t+1} = A s_t + B a_t +
O \nu_t$ and $r_{t} = C s_t + D a_t + O' \nu_t$, where $s_t$ and
$a_t$ are physical state variable and control variable,
respectively, which are usually bounded real vectors, $A,B,C,D,O,O'$
are matrices with proper dimensions, and $\nu_t$ is a noise process.
The mean-variance optimization of accumulated rewards
$\sum_{t=0}^{T-1}r_t$ of such linear system satisfies
Conditions~(i)\&(ii) and our Algorithm~\ref{alg:local} can find the
globally optimal control law. In the next section, we will discuss
some application examples that exactly satisfy such conditions.

\section{Application Examples}\label{sec:exp}

In this section, we apply the theoretical results and the algorithm
in Sections~\ref{sec:opt} and \ref{sec:alg} to some practical
examples, including multi-period mean-variance optimization for
portfolio selection, queueing control, and inventory management
problems.

\subsection{Multi-Period Mean-Variance Portfolio Selection}\label{sec:pf}
Multi-period mean-variance portfolio selection is a well-known
challenging problem in finance engineering, which is described as
follows. An investor has an initial wealth $s_0$. There are a
riskless security ($0$) and $n$ risky securities ($1,\ldots,n$) in
the market. Each security $i$ takes a random return rate $e_t^i$ at
period $t$, and the expectation of $e_t^i$ and the covariance of
$e_t^i$ and $e_t^j$ are known, $\forall i,j=1,2,\dots,n$,
$t=0,1,\dots,T-1$. The objective is to find the best allocation of
wealth among these securities such that the mean and variance of
terminal wealth at period $T$ is optimized. The single-period
portfolio selection was initially proposed by the Nobel laureate
\cite{markowitz1952}, while the multi-period case is challenging
because of the time inconsistency. \cite{li2000optimal} proposed a
so-called embedding method to solve this problem in an analytical
form, via a formulation of stochastic control model, which initiates
intensive research attention following this pioneering work. In this
subsection, we use the MDP model to formulate this problem and apply
our approach to solve it. We find that our MDP approach can obtain
the same result as that of \cite{li2000optimal} and further show
that our Algorithm~\ref{alg:local} can find the global optimum of
this problem.

We formulate the multi-period mean-variance portfolio selection
problem as a finite-horizon MV-MDP $\mathcal{M}_p = \langle \mathcal
T, \mathcal S, \mathcal A, (\bm Q_t, t \in \mathcal T), (r_t, t \in
\mathcal T) \rangle$. For each period $t \in \mathcal
T:=\left\{0,1,\ldots,T-1\right\}$, the state $s_t \in \mathcal S
:=(0,+\infty)$ represents the current wealth, action $\bm
a_t=(a_t^1,\ldots,a_t^n)' \in \mathcal A:=\mathbb R^n$ denotes the
allocation of wealth $s_t$ among $n$ risky securities, where
$a_t^i<0$ represents short sale and the superscript $^\prime$
indicates the transpose of vectors. All the left wealth
$s_t-\sum\limits_{i=1}^n a_t^i$ is allocated to the riskless
security 0 with a constant return rate $e_t^0$. The state transition
is determined by $s_{t+1}=e_t^0 s_t+\bm Q_t' \bm a_t$, where $\bm
Q_t'=[e_t^1-e_t^0,\ldots,e_t^n-e_t^0]$ is the excess return vector.
The one-step instantaneous reward is set as the wealth changed,
i.e., $r_t(s_t,\bm a_t, \bm e_t)=e_t^0 s_t+\bm Q_t' \bm a_t-s_t$,
where $\bm e_t=(e_t^1,\ldots,e_t^n)'$ is the random variable vector
of return rates which captures the stochasticity of the whole
system. The terminal wealth $s_T=s_0+\sum\limits_{t=0}^{T-1}r_t(s_t,
a_t, \bm e_t)=s_0+R_{0:T}$. The objective is to maximize the
combined mean-variance metric of the terminal wealth, i.e.,
\begin{align}\label{equ:mv_pf}
    J_0^*(s_0)
    &=\max\limits_{u \in \mathcal U^{\rm HR}} J_0^u(s_0) = \max\limits_{u \in \mathcal U^{\rm HR}} \left\{\mathbbm E_{s_0}^u[s_T]-\lambda\sigma_{s_0}^u(s_T) \right\} \nonumber\\
    &=\max\limits_{u \in \mathcal U^{\rm HR}} \left\{\mathbbm E_{s_0}^u\big[s_0+R_{0:T}\big]-\lambda\mathbbm E_{s_0}^u\big[(s_0+R_{0:T}-\mathbbm E_{s_0}^u(s_0+R_{0:T}))^2\big]\right\}.
\end{align}
It is easy to verify that this problem setting satisfies
Conditions~(i)\&(ii) of Theorem~\ref{thm:global}, since $s_t$ and
$\bm a_t$ belong to real spaces, and $s_{t+1}$ and $r_{t}$ have
linear forms. Following the optimization approach in
Section~\ref{sec:opt}, we convert the above maximization problem to
a bilevel MDP,
\begin{align}\label{equ:bilevel_pf}
    J_0^*(s_0)
    =\max\limits_{y_0 \in \mathcal{Y}}\max\limits_{u \in \mathcal U^{\rm HD}} \mathbbm{E}_{s_0}^{u}\big[s_0+R_{0:T}-\lambda\big(s_0+R_{0:T}-y_0\big)^2\big].
\end{align}
Given $y_0$, the inner level is a pseudo MV-MDP to
maximize the pseudo mean-variance of the terminal wealth,
\begin{equation}\label{equ:inner_pf}
    \hat J_0^*(s_0,y_0) = \max\limits_{u \in \mathcal U^{\rm HD}} \hat J_0^u(s_0,y_0) = \max\limits_{u \in \mathcal U^{\rm HD}}
\mathbbm{E}_{s_0}^{u}\big[s_0+R_{0:T}-\lambda\big(s_0+R_{0:T}-y_0\big)^2\big].
\end{equation}
In contrast to the general pseudo MV-MDP (\ref{equ:inner}), dynamic
programming can be directly applied to solve (\ref{equ:inner_pf})
without augmented state space because this problem has a special
form of reward function $r_t=s_{t+1}-s_t$ and the total wealth $s_T
\equiv s_t+R_{t:T}, \forall t \in \mathcal
T$. We summarize this result as Theorem~\ref{thm:dp_pf} below.

\begin{thm}\label{thm:dp_pf}
Given $y_0 \in \mathcal Y$, define an operator $ \mathbb{\hat L}^*_t:
\mathcal B(\tilde{\mathcal{S}}) \rightarrow \mathcal B(\tilde{\mathcal{S}})$ for $t
\in \mathcal T$ by
\begin{equation}\label{equ:opt_opera_pf}
    \mathbb{\hat L}^*_t  v(s_t,y_0)=\max\limits_{\bm a \in \mathcal{A}(s_t)}\mathbbm E[v(e_t^0 s_t+\bm Q_t' \bm a,y_0)],\quad
     v \in \mathcal B(\tilde{\mathcal{S}}).
\end{equation}
And we define a function sequence $\left\{V^*_t \in \mathcal B(\tilde{\mathcal{S}}); t \in \mathcal
T\right\}$ by
\begin{equation}\label{equ:dp_pf}
  V^*_t= \mathbb{\hat L}^*_t V^*_{t+1},\quad\forall t \in \mathcal T \text{ and } V^*_T(s_T,y_0):=s_T-\lambda (s_T-y_0)^2,
\end{equation}
then we have $\hat J_0^*=V^*_0$. Further, if $\bm a_t^*\in \mathcal A(s_t)$
attains the maximum in the operation $\mathbb{\hat{L}}^*_t
V^*_{t+1}(s_t,y_0)$, then the policy $\hat{u}^*=(\hat{u}^*_t; t \in
\mathcal T) \in {\mathcal U}^{\rm {MD}}$ with $\hat{u}_{t}^*(s_t)=\bm
a_t^*(s_t,y_0)$ is an optimal policy for the inner pseudo MV-MDP
(\ref{equ:inner_pf}), which is a Markov policy depending only on the
current state $s_t$.
\end{thm}

%\begin{proof}
%This theorem is similar to Theorem~\ref{thm:opt}, but has some
%differences due to the special reward function. We present the proof
%in Appendix~\ref{app:dp_pf}.
%\end{proof}

From (\ref{equ:dp_pf}), we find that $y_0$ does not change during
the procedure of dynamic programming, which is different from
part~(b) of Theorem~\ref{thm:opt}. Thus, in this specific model of
portfolio selection, we need not to treat $y_0$ as an auxiliary
state, which is different from the augmented state $(s_t, y_t) \in
\tilde {\mathcal S}$ defined in Section~\ref{sec:opt}. The inner
pseudo MV-MDP (\ref{equ:inner_pf}) can be simplified as a standard
finite-horizon MDP, where $y_0$ can be viewed as a predetermined
parameter of this MDP. The optimal policy $\hat u^*$ can be
deterministic Markovian, not depending on history anymore. For
notational simplicity, in what follows, we rewrite $y_0$ as $y$ to
avoid misunderstandings.

Therefore, we can solve the inner pseudo MV-MDP (\ref{equ:inner_pf})
by using dynamic programming (\ref{equ:dp_pf}). We obtain analytical
solutions of an optimal policy $\hat u^*=(\hat u_t^*; t \in \mathcal
T) \in \mathcal U^{\rm MD}$ and the optimal value function $\hat
J_0^*$ as follows, where we utilize the quadratic form of $\mathbbm E[V_t^*(e_t^0 s_t+\bm Q_t' \bm a,y_0)]$
with respect to $\bm a$ and the detailed analysis process is ignored
for space limit.
\begin{align}\label{equ:policy_pf}
    \hat u^*_{t}(s_t)=\left[-e_t^0 s_t+({y}+\frac{1}{2\lambda}){\prod_{\tau=t+1}^{T-1}(e_\tau^0)^{-1}}\right]\bm\Sigma_t^{-1}\bm\mu_t, \quad s_t \in \mathcal S, t \in \mathcal T.
\end{align}
%\begin{align}
%    \hat J_0^*(s_0,y)
%    &=-\lambda \left(1-\sum\limits_{\tau=0}^{T-1}C_\tau\right){y}^2 + \left(\sum\limits_{\tau=0}^{T-1}C_\tau+2\lambda \prod_{\tau=0}^{T-1}e_\tau^0\left[1-\bm\mu_{\tau}'\bm\Sigma_\tau^{-1}\bm\mu_{\tau}\right] s_0 \right)y \nonumber\\
%    &+ \frac{1}{4\lambda}\sum\limits_{\tau=0}^{T-1}C_\tau+
%    \prod_{\tau=0}^{T-1}e_\tau^0\left[1-\bm\mu_{\tau}'\bm\Sigma_\tau^{-1}\bm\mu_{\tau}\right]s_0
%    -\lambda \prod_{\tau=0}^{T-1}(e_\tau^0)^2\left[1-\bm\mu_{\tau}'\bm\Sigma_\tau^{-1}\bm\mu_{\tau}\right]s_0^2, \quad\forall s_0 \in \mathcal S,\nonumber
%\end{align}
\begin{align}
    \hat J_0^*(s_0,y)
    &=-\lambda \left(\prod_{\tau=0}^{T-1}C_\tau\right) {y}^2 + \left[1-\prod_{\tau=0}^{T-1}C_\tau+2\lambda \prod_{\tau=0}^{T-1}\left(e_\tau^0 C_\tau\right) s_0 \right]y \nonumber\\
    &+ \frac{1}{4\lambda}\left(1-\prod_{\tau=0}^{T-1}C_\tau\right)+
    \prod_{\tau=0}^{T-1}\left(e_\tau^0 C_\tau\right) s_0
    -\lambda \prod_{\tau=0}^{T-1}\left((e_\tau^0)^2 C_\tau\right) s_0^2, \quad s_0 \in \mathcal S,\nonumber
\end{align}
where $\bm\mu_t := \mathbbm{E} [\bm Q_t]$, $\bm\Sigma_t :=
\mathbbm{E}[\bm Q_t \bm Q_t']$, and $C_t :=
1-\bm\mu_t'\bm\Sigma_t^{-1}\bm\mu_t$.
%\begin{equation}
%    \nonumber
%    C_t := {\prod_{\tau=t+1}^{T-1}\left[1-\bm\mu_\tau'\bm\Sigma_\tau^{-1}\bm\mu_\tau\right]}\bm\mu_t'\bm\Sigma_t^{-1}\bm\mu_t.
%\end{equation}
%\begin{equation}
%    \nonumber
%    C_t := 1-\bm\mu_t'\bm\Sigma_t^{-1}\bm\mu_t.
%\end{equation}

Note that $\hat J_0^*(s_0,y)$ is quadratically concave with respect
to $y$. Therefore, the outer level of the bilevel MDP (\ref{equ:bilevel_pf}) is a quadratic convex optimization
problem and can be solved analytically with solution
%\begin{equation}\label{eq_y*s0}
%    y^*_0=\frac{2\lambda \prod_{\tau=0}^{T-1}e_\tau^0\left[1-\bm\mu_{\tau}'\bm\Sigma_\tau^{-1}\bm\mu_{\tau}\right] s_0 + \sum\limits_{\tau=0}^{T-1}C_\tau}{2\lambda \left(1-\sum\limits_{\tau=0}^{T-1}C_\tau\right)}, \quad\forall s_0 \in \mathcal S.
%\end{equation}
\begin{equation}\label{eq_y*s0}
    y^*=\left(\prod_{\tau=0}^{T-1} e_\tau^0\right) s_0 + \frac{1-\prod_{\tau=0}^{T-1} C_\tau}{2\lambda \prod_{\tau=0}^{T-1} C_\tau}, \quad s_0 \in \mathcal S.
\end{equation}
and the corresponding mean-variance is
\begin{equation}\label{eq_J*s0}
 J_0^*(s_0)=\hat J_0^*(s_0,y^*)
    =\left(\prod_{\tau=0}^{T-1} e_\tau^0\right) s_0 + \frac{1-\prod_{\tau=0}^{T-1} C_\tau}{4\lambda \prod_{\tau=0}^{T-1} C_\tau}, \quad s_0 \in \mathcal S.
\end{equation}
(\ref{eq_y*s0}) implies that $y^*$ is linear to $s_0$, which is
consistent with Theorem~\ref{thm:structure_mv}.
Therefore, we can obtain the optimal policy $u^*=(u^*_t;t \in
\mathcal T) \in \mathcal U^{\rm HD}$ for the original multi-period
mean-variance portfolio selection problem (\ref{equ:mv_pf}) by
substituting (\ref{eq_y*s0}) into (\ref{equ:policy_pf}), i.e.,
\begin{equation}\label{equ:opt_policy}
    u_t^*(s_t)=-\bm\Sigma_t^{-1}\bm\mu_t e_t^0 s_t+\left(\prod_{\tau=0}^{T-1}e_\tau^0 s_0 + \frac{1}{2\lambda \prod_{\tau=0}^{T-1}C_\tau}\right)
    {\prod_{\tau=t+1}^{T-1}(e_\tau^0)^{-1}}\bm\Sigma_t^{-1}\bm\mu_t, \quad s_t \in \mathcal S, t \in \mathcal T.
\end{equation}
We can see that the above control law has a linear form, i.e., the
action $u^*_t$ is linear to the current state $s_t$. This solution
is exactly the same as the result by \cite{li2000optimal}.

\begin{remark}\rm{
It is observed from (\ref{equ:opt_policy}) that $u_t^*$ depends only
on the initial state $s_0$ and the current state $s_t$ rather than
the history sequence $h_t=\left\{s_0,a_0,\ldots,s_t\right\}$, such
policy $u^*=(u_t^*; t \in \mathcal T)$ may be called a
\textit{semi-Markov} policy \citep{fainberg1982non}. Moreover, if
$\bm e_t$ is time-homogeneous and $e_t^0=1$, then $u_t^*$ is
independent of $t$ and $u^*$ is a \textit{stationary} semi-Markov
policy. }
% This special structural policy relies on the special reward function $r_t=s_{t+1}-s_t$, which results the total reward $R^0=S_T$ depends only on the terminal state $s_T$.
\end{remark}

In addition, although this problem has the analytical form solution
(\ref{equ:opt_policy}), we also implement our
Algorithm~\ref{alg:local} to iteratively solve this problem such
that the convergence capability of Algorithm~\ref{alg:local} can be
validated. In what follows, we use the same experiment setting as
that in Example~2 of \cite{li2000optimal} to verify the above
theoretical and algorithmic results.

% as Theorem \ref{thm:glob} shows.
% \begin{thm}\label{thm:glob}
% In our multi-period mean-variance portfolio selection problem, Algorithm \ref{alg:local} converges to the optimal policy of MV-MDP (\ref{equ:mv_pf}).
% \end{thm}
% \begin{proof}
%    Suppose Algorithm \ref{alg:local} converges to $\hat {\bm y}^*=(\hat y^*_0, s_0 \in \mathcal S )$ and $\hat u^*$, that is, $(\hat {\bm y}^*, \hat u^*)$ is a fixed point, i.e.,
%    \begin{equation}
%        \nonumber
%        \hat J_0^{\hat u^*}(s_0,\hat { y}^*(s_0))=\hat J_0^*(s_0,\hat { y}^*(s_0)), \quad\forall s_0 \in \mathcal S,
%    \end{equation}
%    \begin{equation}
%        \nonumber
%        \hat J_0^{\hat u^*}(s_0,\hat { y}^*(s_0))=    J_0^{\hat u^*}(s_0), \quad\forall s_0 \in \mathcal S.
%    \end{equation}
%    On the other hand, $\hat J_0^*(s_0,y)$ is  quadratic concave with respect to $y$, it must be
%    \begin{equation}
%        \nonumber
%        \hat J_0^*(s_0,\hat y^*_0) =   \max\limits_{y \in \mathcal R}\hat J_0^*(s_0,y), \quad\forall s_0 \in \mathcal S.
%    \end{equation}
%The above three equations imply that
%    \begin{equation}
%        \nonumber
%        J_0^{\hat u^*}(s_0) = \max\limits_{y \in \mathcal R}\hat J_0^*(s_0,y)=J_0^{*}(s_0), \quad\forall s_0 \in \mathcal S.
%    \end{equation}
%    Therefore, $\hat u^*$ is the optimal policy of MV-MDP (\ref{equ:mv_pf}).
% \end{proof}

\begin{example}\label{exam:pf}
An investor has wealth $s_0>0$ at the beginning of the planning
horizon $\mathcal{T}=\left\{0,1,2,3\right\}$. The investor is trying
to find the best allocation of his wealth among three risky
securities (1,2,3) and one riskless security (0). The riskless
security has a constant return rate $e_t^0 \equiv 1.04$ and the
expected return rates of risky securities are
$\mathbbm{E}[e_t^1]=1.162, \mathbbm{E}[e_t^2]=1.246,
\mathbbm{E}[e_t^3]=1.228$. The covariance of $\bm
e_t=(e_t^1,e_t^2,e_t^3)'$ is
 \[
 {\rm Cov}(\bm e_t)=
 \begin{bmatrix}
    0.0146& 0.0187& 0.0145 \\
    0.0187& 0.0854& 0.0104 \\
    0.0145& 0.0104& 0.0289 \\
 \end{bmatrix}
 ,\quad \forall t \in \mathcal{T}.
 \]
The risk aversion coefficient is $\lambda=2$. The investor aims to
find an efficient portfolio policy to maximize the expected return
and minimize the variance of terminal wealth at $T=4$, i.e.,
\begin{equation}
    \nonumber
    \max\limits_{u \in \mathcal U^{\rm HR}} \left\{ \mathbbm{E}_{s_0}^u[s_4] - 2\sigma_{s_0}^u[s_4]\right\}, \quad {\rm given~} s_0.
\end{equation}
\end{example}

We formulate this problem as a finite-horizon MV-MDP and solve it
analytically and numerically, respectively. First, according to the
expectation and covariance of $\bm e_t$, we have
$\bm\mu_t=\mathbbm{E}[\bm Q_t]=[0.122,0.206,0.188]'$ and
  \[
\bm\Sigma_t=\mathbbm{E}[\bm Q_t \bm Q_t'] =
\begin{bmatrix}
    0.0295& 0.0438& 0.0374 \\
    0.0438& 0.1278& 0.0491 \\
    0.0374& 0.0491& 0.0642 \\
\end{bmatrix}
 , \quad \forall t \in \mathcal{T}.
 \]
Based on (\ref{eq_y*s0}) and (\ref{equ:opt_policy}), we obtain
\begin{eqnarray}
& y^*  = 1.1697 s_0 + 8.9751, \nonumber \\
& J^*_0(s_0) = 1.1697 s_0 + 4.4876, \nonumber \\
&u_t^*(s_t) = -\left[
    \begin{array}{l}
        0.4004\\0.6496\\2.3133
    \end{array}
    \right] s_t+1.04^{t-3}\times(1.1699s_0+9.2193)\times
    \left[
    \begin{array}{l}
    0.3887 \\ 0.6240 \\2.2247
    \end{array}
    \right], \quad s_t \in \mathcal S, t \in \mathcal T.
    \nonumber
\end{eqnarray}
This analytical result is exactly the same as that of
\cite{li2000optimal} by taking the initial wealth $s_0=1$.

\begin{figure} [htbp]
    \centering
    %\begin{minipage}[b]{0.65\textwidth}
        \includegraphics[width=0.55\textwidth]{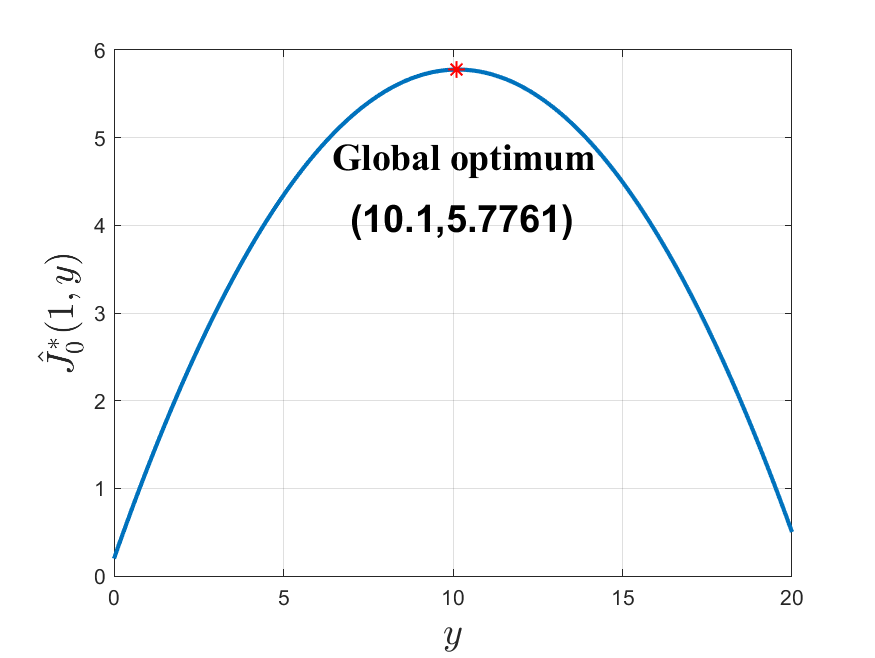}
    %\end{minipage}
    \caption{Illustration of the optimal value of the pseudo mean-variance $\hat J_0^*(1,y)$.
    }
    \label{fig:pmv_pf}
\end{figure}

Next, we suppose that the initial wealth $s_0=1$ and give an
illustration curve of $\hat J_0^*(1,y)$ in Figure~\ref{fig:pmv_pf}
based on the above analytical solution. The maximum is attained at
$y^*=10.1$ with optimal mean-variance value $J_0^*(1)=5.7761$. As a
comparison, we use Algorithm~\ref{alg:local} to iteratively compute
the solution of Example~\ref{exam:pf}. Since this portfolio
selection problem clearly satisfies the conditions in
Theorem~\ref{thm:global}, we expect that Algorithm~\ref{alg:local}
can find the global optimum. To verify the global convergence, we
choose different initial pseudo mean $y^{(0)}$ with values
$2,5,10,12,20$. The convergence results of pseudo mean $y$ and
pseudo mean-variance $\hat J^*_0(1,y)$ are presented in
Figure~\ref{fig:conver_pf}(a)\&(b), respectively. We can see that
pseudo mean $y$ always converges to $10.1$ and pseudo mean-variance
$\hat J^*_0(1,y)$ always converges to $5.7761$ in
Figure~\ref{fig:conver_pf}, which is the same as the analytical
result. Thus, the global convergence of Algorithm~\ref{alg:local} is
demonstrated in this example.
\begin{figure*}[htbp]
    \subfigure[Convergence of pseudo mean.]{
        \begin{minipage}[b]{0.5\textwidth}
            \includegraphics[width=1\textwidth]{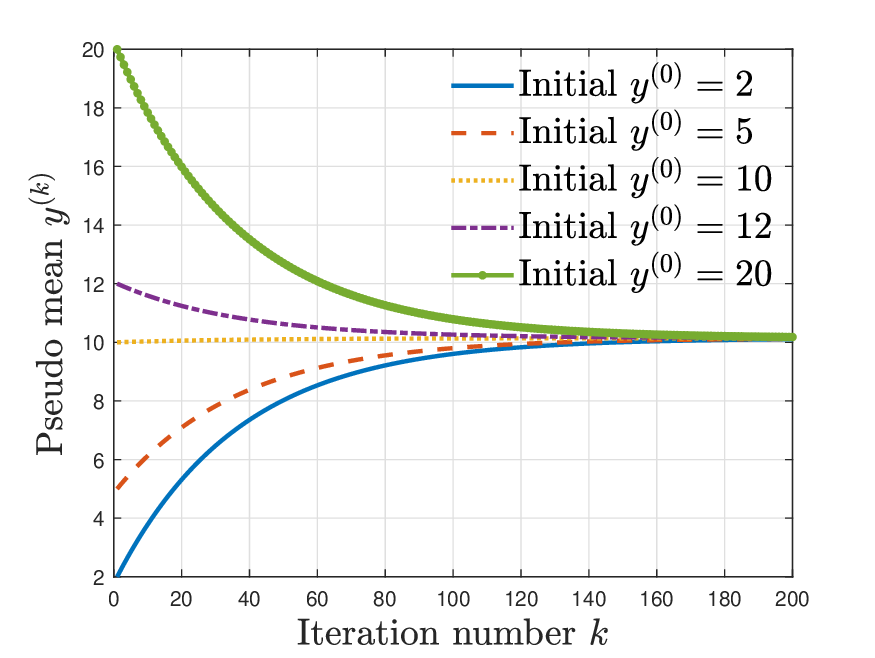}

        \end{minipage}
    }
    \subfigure[Convergence of pseudo mean-variance.]{
        \begin{minipage}[b]{0.5\textwidth}
            \includegraphics[width=1\textwidth]{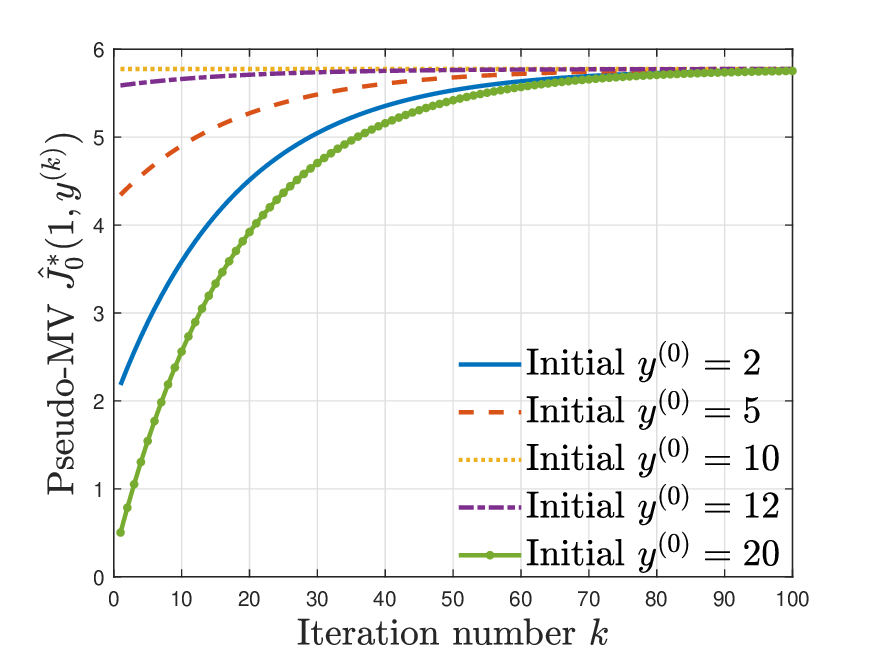}
        \end{minipage}
    }
    \caption{Convergence results of Algorithm~\ref{alg:local} for solving Example~\ref{exam:pf}. } \label{fig:conver_pf}
\end{figure*}

\begin{remark}\label{rem:pf}\rm{
In the multi-period mean-variance portfolio selection problem, due
to the special form of reward function and the linearity of state
transition function, the results in Sections~\ref{sec:opt} and
\ref{sec:alg} are further specified, including the existence of
optimal \emph{semi-Markov deterministic} policies and the \emph{global}
convergence of Algorithm~\ref{alg:local}. Although the method of
\cite{li2000optimal} elegantly solve the multi-period portfolio
selection, it heavily relies on the specific model and is hardly
extended to other problems. In contrast, our approach works for a
general MDP model which has much wider application scenarios since
most of stochastic dynamic systems can be formulated as Markov
models. In the following subsections, we give a preliminary
investigation of applying our approach to study the mean-variance
optimization for queueing control and inventory management, which
demonstrates the applicability of our approach.}

%Different from stochastic control methods that rely on specific
%problem setting, MDP is a fairly generic model to optimize various
%multi-period mean-variance portfolio selection problems, including
%correlated return, uncertain horizon, continuous time, etc.
%Moreover, our approach is generally applicable to any problem
%modeled by MDPs, such as the multi-period mean-variance inventory
%management problem, which is studied in the next subsection.
\end{remark}

\subsection{Mean-Variance Queueing Control}\label{sec:que}

Queueing models are widely used in operations research and
management. In this subsection, we study the mean-variance
optimization of the random costs incurred in queueing systems, which
may reflect the performance and fairness of systems.

We consider a discrete-time $Geo/D/1$ queue in which the arrival is
a geometric process with probability $0<q<1$ and the service is a
deterministic process. In this example, we focus on the workload
process of queueing models
\citep{borovkov2003integral,he2005age,perry2001m}, where the system
workload is the sum of all customers' service requirements. When a
customer arrives with probability $q$, the service requirement
(workload) of that customer is a random variable uniformly
distributed in $[0,X]$. The system state is the total remaining
workload, and the system has a workload capacity $S>0$. At each time
epoch $t$, the remaining system workload $s_t \in [0,S]$ is observed
and the decision maker needs to determine the service rate $a_t \in
[0,A]$. The system has two types of costs, operating cost and
holding cost, which are proportional to service rate and remaining
workload with unit price $c_o$ and $c_h$, respectively. Our
objective is to minimize both the mean and variance of the total
costs over a finite period
$\mathcal{T}=\left\{0,1,\ldots,T-1\right\}$.

We formulate this mean-variance queueing control problem as a
finite-horizon MV-MDP $\mathcal M_{q}=\left\{\mathcal T, \mathcal S,
\mathcal A, \mathcal X, (q_t,t \in \mathcal T),  (r_t,t \in \mathcal
T)\right\}$. %For each epoch $t \in \mathcal T$, state $s_t \in
%\mathcal S := [0,S]$ represents the current remaining workload,
%action $a_t \in \mathcal A:=[0,A]$ denotes the current service rate.
At each time $t \in \mathcal T$, an arriving workload $\xi_t \in
\mathcal X:=[0,X]$ will be generated, with probability density
$q_t(\xi_t=x_t)=\frac{q}{X}$ for $x_t \in (0,X]$ and with
probability $q_t(\xi_t=0)=1-q$. The transition function of system
state (remaining workload) is given by
$s_{t+1}=\min\left\{[s_t-a_t]^++\xi_t,S\right\}$ and the cost
function $c_t(s_t,a_t,\xi_t) =  c_o \cdot a_t + c_h \cdot
\min\left\{[s_t-a_t]^++\xi_t,S\right\} $, where
$[\cdot]^+:=\max\left\{\cdot,0\right\}$. We let
$r_t(s_t,a_t,\xi_t):=-c_t(s_t,a_t,\xi_t)$ as the reward function for
convenience. Our goal is to maximize the combined mean-variance
metric of the total rewards
$R_{0:T}=\sum\limits_{t=0}^{T-1}r_t(s_t,a_t,\xi_t)$, i.e.,
\begin{align*}
J_0^*(s_0) &= \max\limits_{u \in \mathcal U^{\rm HR}} \left\{ \mu^u_0
(s_0)-\lambda \sigma^u_0 (s_0) \right\} \\
&= \max\limits_{u \in \mathcal
U^{\rm HR}} \left\{\mathbbm{E}_{s_0}^{u}[R_{0:T}]-\lambda
\mathbbm{E}_{s_0}^{u}\big[\big(R_{0:T}-\mathbbm{E}_{s_0}^{u}[R_{0:T}]\big)^2\big]\right\}.
\end{align*}
Following the optimization approach in Section~\ref{sec:opt}, we
convert the MV-MDP problem to a bilevel MDP
\begin{align*}
J_0^*(s_0) =\max\limits_{y_0 \in \mathcal{Y}}\max\limits_{u \in
\mathcal U^{\rm HD}}
\mathbbm{E}_{s_0}^{u}\big[R_{0:T}-\lambda\big(R_{0:T}-y_0\big)^2\big] =
\max\limits_{y_0 \in \mathcal{Y}}  \hat J_0^*(s_0,y_0).
\end{align*}
The experiment parameters are set as $T=4, S=10, A=X=1, q=1/2,
c_o=2, c_h=1, \lambda=2$. We aim to solve this problem with the
initial state $s_0 \in [4,6]$. Under this parameter setting, both
the transition function and the reward function are linear to
$s_t,a_t$, i.e.,
\begin{equation}
        \nonumber
        s_{t+1}=s_t-a_t+\xi_t,\quad\forall s_0 \in [4,6],
\end{equation}
\begin{equation}
        \nonumber
        r_t(s_t,a_t,\xi_t) = - c_o \cdot a_t - c_h
        (s_t-a_t+\xi_t),\quad\forall s_0 \in [4,6].
\end{equation}
The convexity of $\mathcal S$ and $\mathcal A$ is obviously satisfied. Therefore,
Algorithm \ref{alg:local} converges to the global optimum by Theorem
\ref{thm:global}. In what follows, we apply Algorithm
\ref{alg:local} numerically to verify the global convergence. Since
the state and action spaces are continuous, we use discretization
technique on these continuous spaces. The discretized fineness is
set as $0.01$. %The algorithms are implemented on a personal
%computer with i5-13500HX Intel 2.5GHz CPU, 16GB RAM, MATLAB R2022b
%computing platform, and Win10 OS.

\begin{figure*}[htbp]
    \subfigure[$s_0=4$]{
        \begin{minipage}[b]{0.31\textwidth}
            \includegraphics[width=1\textwidth]{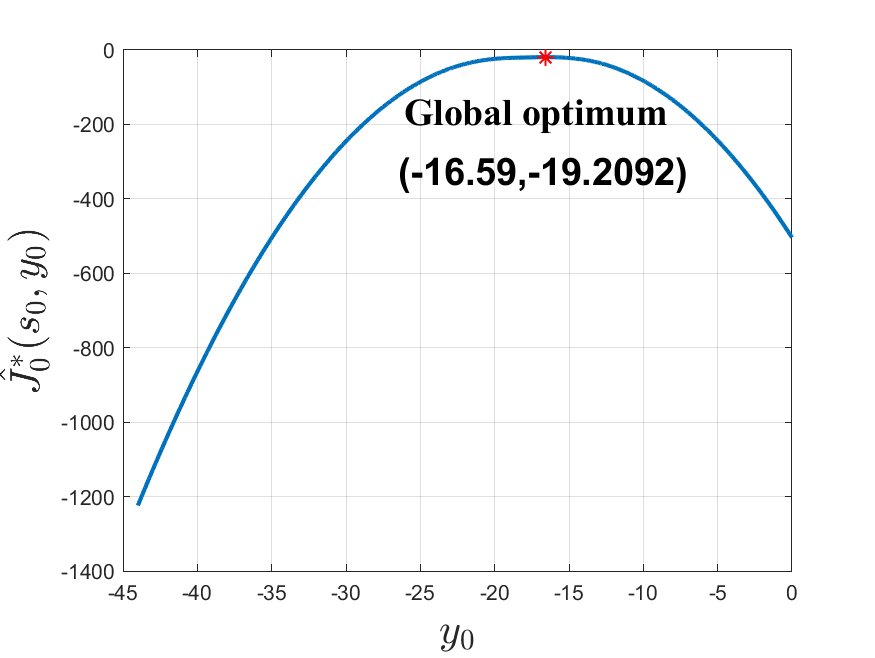}

        \end{minipage}
    }
    \subfigure[$s_0=5$]{
        \begin{minipage}[b]{0.31\textwidth}
            \includegraphics[width=1\textwidth]{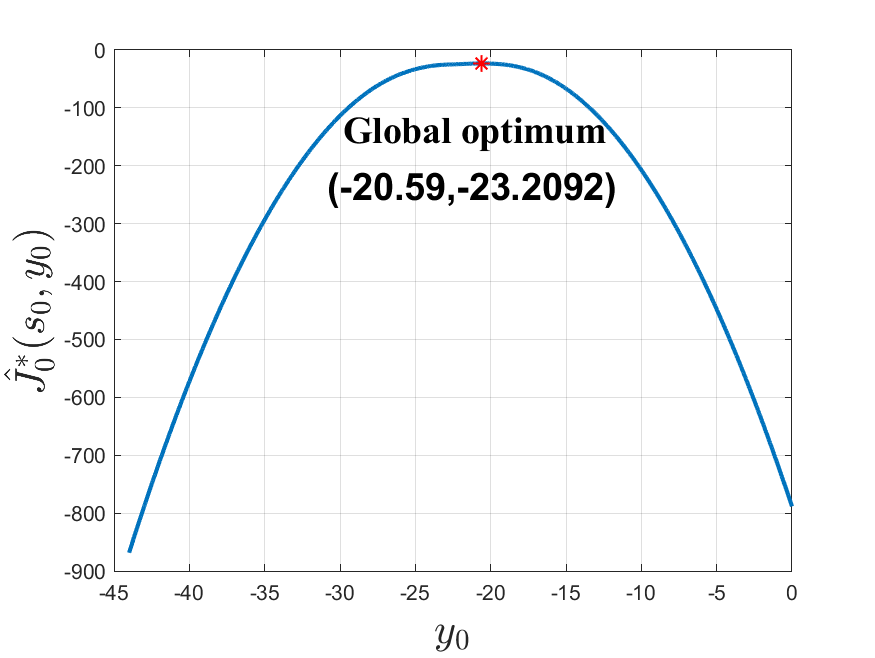}
        \end{minipage}
    }
    \subfigure[$s_0=6$]{
        \begin{minipage}[b]{0.31\textwidth}
            \includegraphics[width=1\textwidth]{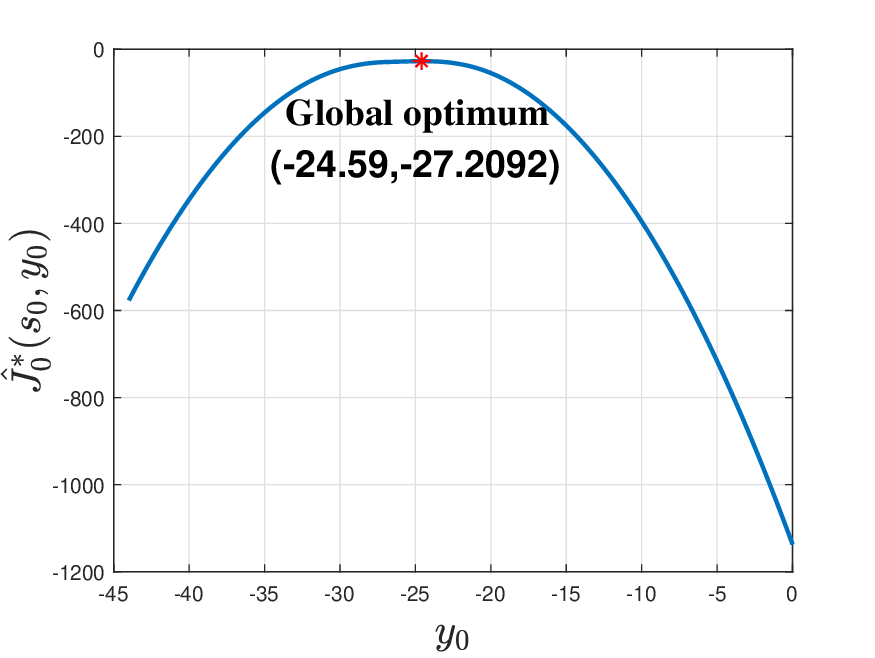}
        \end{minipage}
    }
    \caption{Curves of the optimal pseudo mean-variance $\hat J_0^*(s_0,y_0)$ with respect to $y_0$, computed by the grid search method.} \label{fig:pmv_que}
\end{figure*}
First, we use the grid search method to enumeratively solve the
augmented MDP $\widetilde{\mathcal M}$. It is easy to verify that
$r_t(s,a,\xi) \in [-11,0]$, necessarily $\mathcal Y = [-44,0]$. We
discretize the continuous space $\mathcal Y$ to a discrete space
$\hat{\mathcal Y}$ with the same fineness $0.01$. Thus, we compute
$\hat J_0^*(s_0,y_0)$ by dynamic programming~(\ref{equ:opt_dp}) at
each $y_0 \in \hat{\mathcal Y}$, and choose the  maximum as the
approximate value of $y^*_0$ and $J_0^*(s_0)$.
%The running time is around 2900 seconds.
In Figure~\ref{fig:pmv_que}, we give illustration curves of $\hat
J_0^*(s_0,y_0)$ with respect to $y_0$ at different initial states
$s_0=4,5,6$. We can observe that these curves truly have a single
local optimum that is also globally optimal.

Next, we apply Algorithm~\ref{alg:local} to iteratively solve this
problem. We choose different initial pseudo mean $y_0^{(0)}$ to
verify the global convergence of
Algorithm~\ref{alg:local}. %the stopping rule is set as
%$|y_0^{(k)}-y_0^{(k-1)}|<0.01$ where $y_0^{(k)}$ denotes the $k$-th
%iteration value. The average running time for each iteration is
%approximately 1100 seconds, which is faster than the grid search method.
The convergence processes of Algorithm~\ref{alg:local} under
different initial state $s_0$ and initial pseudo mean $y^{(0)}_0$
are illustrated in Figure~\ref{fig:conver_que}. We observe that
Algorithm~\ref{alg:local} always converges to the global optimum
under different initial values, which verifies
Theorem~\ref{thm:global}. We also observe that
Algorithm~\ref{alg:local} usually converges fast after very few
iterations. Moreover, Figure~\ref{fig:conver_que} indicates that the
optimal pseudo means for initial states $s_0=4,5,6$ are
$y_0^*=-16.59,-20.59,-24.59$, respectively, presenting a linearity
with respect to $s_0$, which also verifies
Theorem~\ref{thm:structure_mv}.

%    \renewcommand{\arraystretch}{1.1}
%    \begin{table}[htbp]
%        \newcommand{\tabincell}[2]{\begin{tabular}{@{}#1@{}}#2\end{tabular}}
%        \centering
%        \fontsize{4}{7}\selectfont
%        \caption{Convergence results of $y^*_0$ in Algorithm~\ref{alg:local} under different initial $y^{(0)}_0$ and~$s_0$.}
%        \label{tab:conver_que}
%        \small
%        \begin{tabular}{|p{20mm}<{\centering}|p{18mm}<{\centering}|p{18mm}<{\centering}|p{18mm}<{\centering}|}
%            \hline
%            %\diagbox
%            {$y_0^{(0)}$}{$y_0^*$}{$s_0$} & \tabincell{c}{$4$} & \tabincell{c}{$5$}
%            & \tabincell{c}{$6$} \\
%            \hline
%            $0$     &$-16.59$     &$-20.59$    &$-24.59$      \\ \hline
%            $-5$    &$-16.59$     &$-20.59$    &$-24.59$    \\ \hline
%            $-10$   &$-16.59$     &$-20.59$    &$-24.59$      \\ \hline
%            $-15$   &$-16.59$     &$-20.59$    &$-24.59$      \\ \hline
%            $-20$   &$-16.59$     &$-20.59$    &$-24.59$      \\ \hline
%            $-25$   &$-16.59$     &$-20.59$    &$-24.59$      \\ \hline
%            $-30$   &$-16.59$     &$-20.59$    &$-24.59$      \\ \hline
%            $-35$   &$-16.59$     &$-20.59$    &$-24.59$      \\ \hline
%            $-40$   &$-16.59$     &$-20.59$    &$-24.59$      \\ \hline
%            $-45$   &$-16.59$     &$-20.59$    &$-24.59$      \\ \hline
%        \end{tabular}
%    \end{table}

    \begin{figure*}[htbp]
        \subfigure[$s_0=4$]{
            \begin{minipage}[b]{0.31\textwidth}
                \includegraphics[width=1\textwidth]{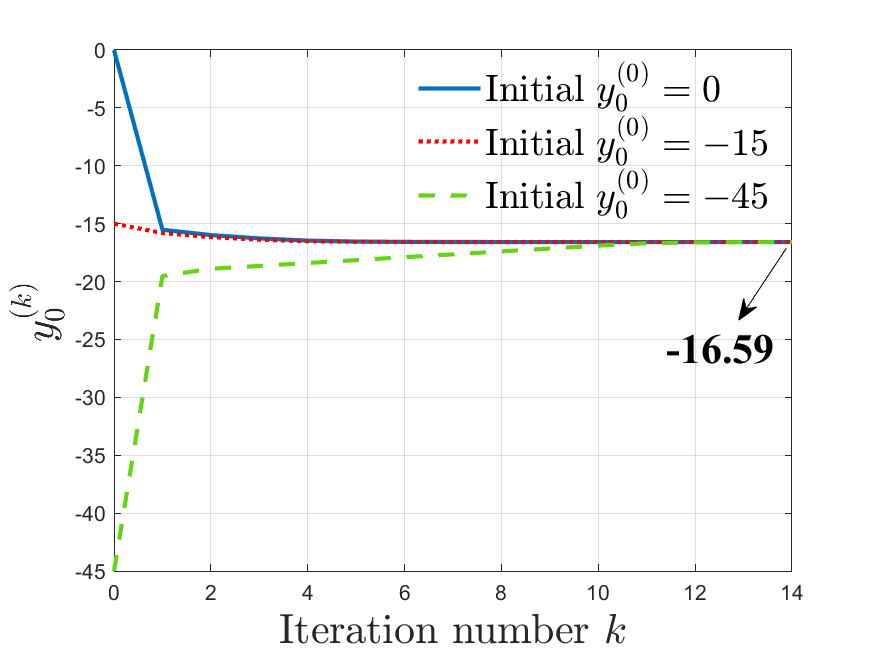}

            \end{minipage}
        }
        \subfigure[$s_0=5$]{
            \begin{minipage}[b]{0.31\textwidth}
                \includegraphics[width=1\textwidth]{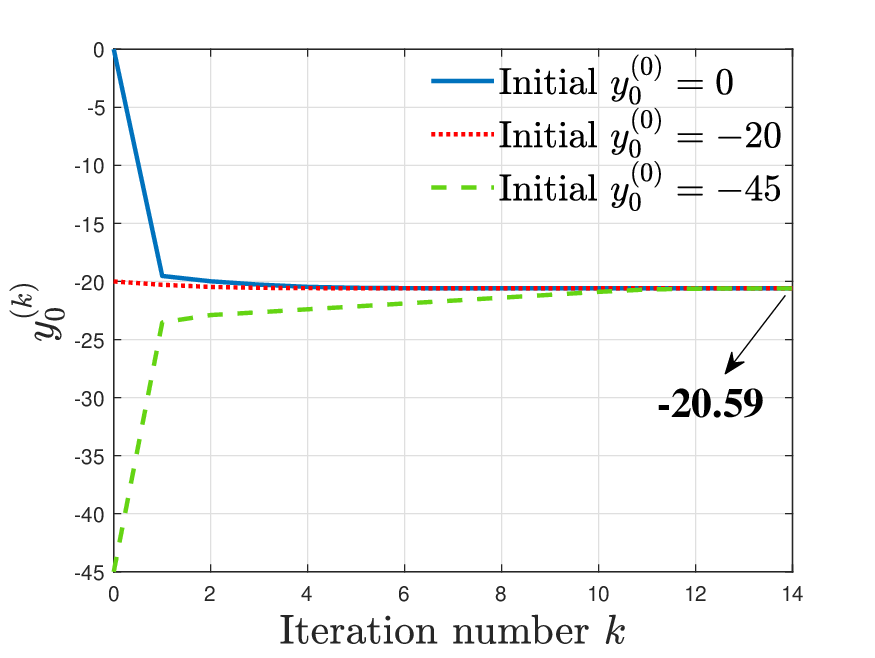}
            \end{minipage}
        }
        \subfigure[$s_0=6$]{
            \begin{minipage}[b]{0.31\textwidth}
                \includegraphics[width=1\textwidth]{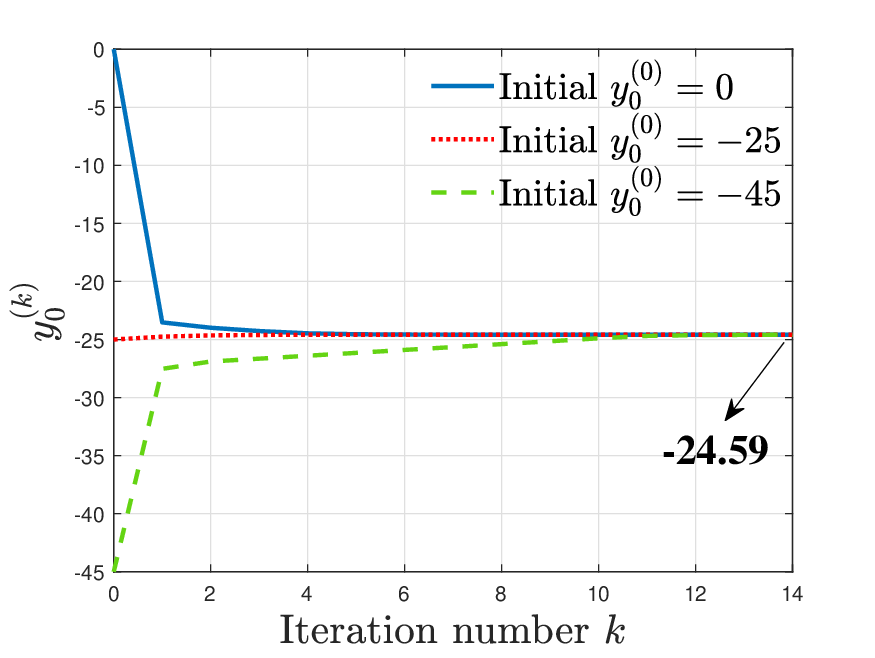}
            \end{minipage}
        }
        \caption{Convergence processes of pseudo mean $y^*_0$ in Algorithm~\ref{alg:local} under different initial values of $y^{(0)}_0$ and $s_0$.} \label{fig:conver_que}
    \end{figure*}

\subsection{Multi-Period Mean-Variance Inventory Management}\label{sec:inventory}

Risk management in dynamic inventory control is a challenging
research topic in the literature \citep{Chen07,Chio2016}. In this
subsection, we demonstrate that our approach can give a promising
avenue to study this problem.

We consider a simple inventory control problem with planned
shortages, non-negative bounded stock, and a maximum capacity $S$.
At each time epoch $t$, the stock level $s_{t} \in
\left\{0,1,\ldots,S\right\}$ is reviewed, an order amount $a_{t} \in
\left\{0,1,\ldots,S-s_t\right\}$ is then restocked, and a stochastic
demand $\xi_{t}$ is realized. Let $p_r$ be the revenue for unit
demand, $c_o$ be the unit order cost, $c_h$ be the unit holding cost
for excess inventory, and $c_s$ be the unit shortage cost for
unfilled demand. These unit parameters are all positive integers
with $c_s>c_o$ and $p_r>c_o$. For convenience, we assume that the
demand variables $\left\{\xi_{t}\right\}$ are independent discrete
random variables uniformly distributed in
$\left\{0,1,\ldots,S\right\}$.
%integer-valued and i.i.d. random variables uniformly distributed in
%$\left\{0,1,\ldots,S\right\}$.
The inventory manager aims to maximize the mean and minimize the
variance of the total return over a finite period $\mathcal
T=\left\{0,1,\ldots,T-1\right\}$.

We formulate this multi-period mean-variance inventory control as a
finite-horizon MV-MDP $\mathcal M_{i}=\left\{\mathcal T, \mathcal S,
\mathcal A, (\mathcal A(s)\subset \mathcal A, s \in \mathcal S), \mathcal X, (q_t,t
\in \mathcal T),  (r_t,t \in \mathcal T)\right\}$. For each time $t
\in \mathcal T$, state $s_t \in \mathcal S :=
\left\{0,1,\ldots,S\right\}$ represents the current stock level, and
action $a_t \in \mathcal A(s_t):=\left\{0,1,\ldots,S-s_t\right\}$
denotes the current order amount. Then a demand $\xi_t \in \mathcal
X:=\left\{0,1,\ldots,S\right\}$ with probability
$q_t(\xi_t=x_t)=\frac{1}{S+1}$ for $x_t \in \mathcal X$ is realized.
The transition function of system state is
$s_{t+1}=[s_t+a_t-\xi_t]^+$ and the reward function is
$r_t(s_t,a_t,\xi_t) = p_r \cdot \xi_t - c_o \cdot a_t - c_h \cdot
[s_t+a_t-\xi_t]^+ - c_s \cdot [\xi_t - s_t-a_t]^+$. The goal is to
maximize the combined mean-variance metric of total rewards
$R_{0:T}=\sum\limits_{t=0}^{T-1}r_t(s_t,a_t,\xi_t)$.
Using the optimization approach in Section~\ref{sec:opt}, we convert
this MV-MDP problem to a bilevel MDP
\begin{equation}
J_0^*(s_0)=\max\limits_{y_0 \in \mathcal{Y}}\max\limits_{u \in
\mathcal U^{\rm HD}}
\mathbbm{E}_{s_0}^{u}\big[R_{0:T}-\lambda\big(R_{0:T}-y_0\big)^2\big]=\max\limits_{y_0
\in \mathcal{Y}}  \hat J_0^*(s_0,y_0).
\end{equation}
The experiment parameters are set as $T=10, S=10, p_r=4, c_o=2,
c_h=1, c_s=3, \lambda=2$.  Under this parameter setting, it is easy
to verify that $r_t(s,a,\xi) \in [-30,40]$, necessarily $\mathcal Y
= [-300,400]$. Thus, the maximum of $\hat J_0^*(s_0,y_0)$ must be
attained with $y_0 \in [-300,400]$. We apply both the grid search
method and Algorithm~\ref{alg:local} to solve this problem.
%The algorithms are implemented on a personal computer with i5-13500HX Intel 2.50GHz CPU,
% 16.0 GB RAM, MATLAB R2022b computing platform, and Win10 OS.

First, we use the grid search method to enumeratively solve the
inner pseudo MV-MDPs at every possible $y_0 \in \mathcal Y$. For
easy computation, we discretize the continuous space $\mathcal Y$ to
a discrete space $\hat{\mathcal Y}$ with fineness $0.1$. Thus, we
compute $\hat J_0^*(s_0,y_0)$ by dynamic
programming~(\ref{equ:opt_dp}) at each $y_0 \in \hat{\mathcal Y}$,
and choose the maximum as the approximate value of $y^*_0$ and
$J_0^*(s_0)$. %The running time is approximately 27.52 seconds.
As a consequence, we give illustration curves of $\hat
J_0^*(s_0,y_0)$ with respect to $y_0$ at different initial states
$s_0$, which are shown in Figure~\ref{fig:pmv_inv}.
\begin{figure*}[t]
    \subfigure[$s_0=0$]{
        \begin{minipage}[b]{0.22\textwidth}
            \includegraphics[width=1\textwidth]{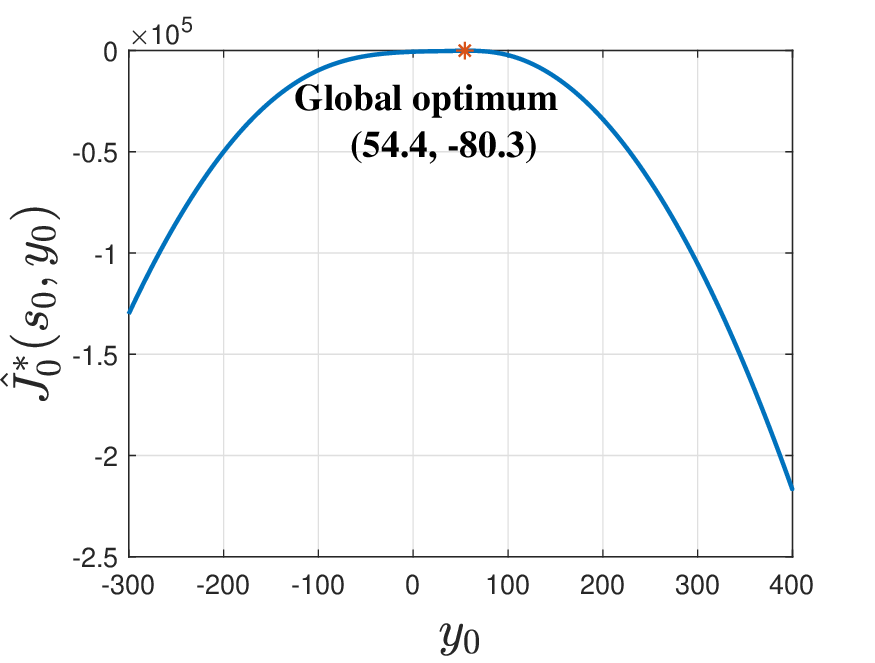}

        \end{minipage}
    }
    \subfigure[$s_0=1$]{
        \begin{minipage}[b]{0.22\textwidth}
            \includegraphics[width=1\textwidth]{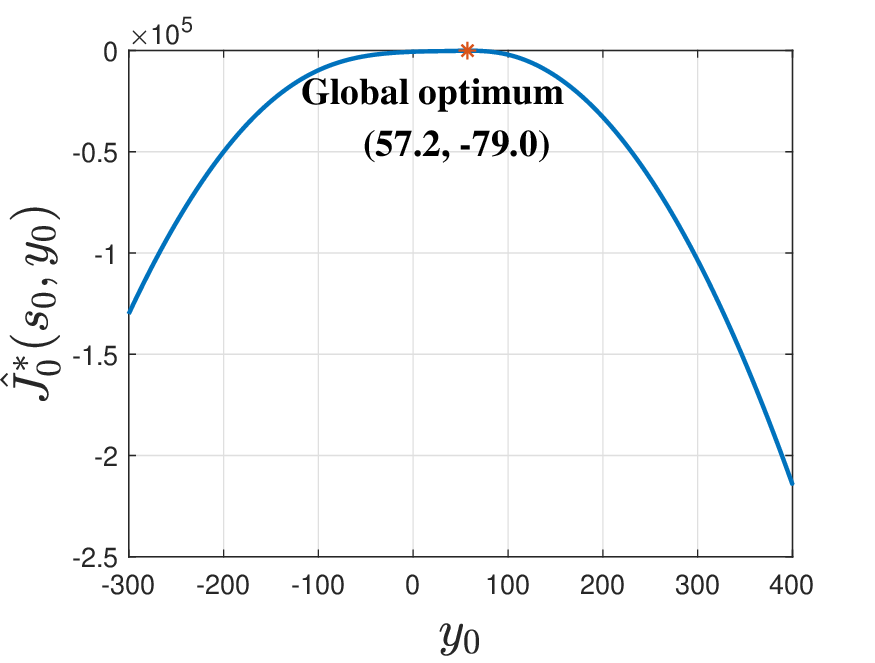}
        \end{minipage}
    }
    \subfigure[$s_0=2$]{
        \begin{minipage}[b]{0.22\textwidth}
            \includegraphics[width=1\textwidth]{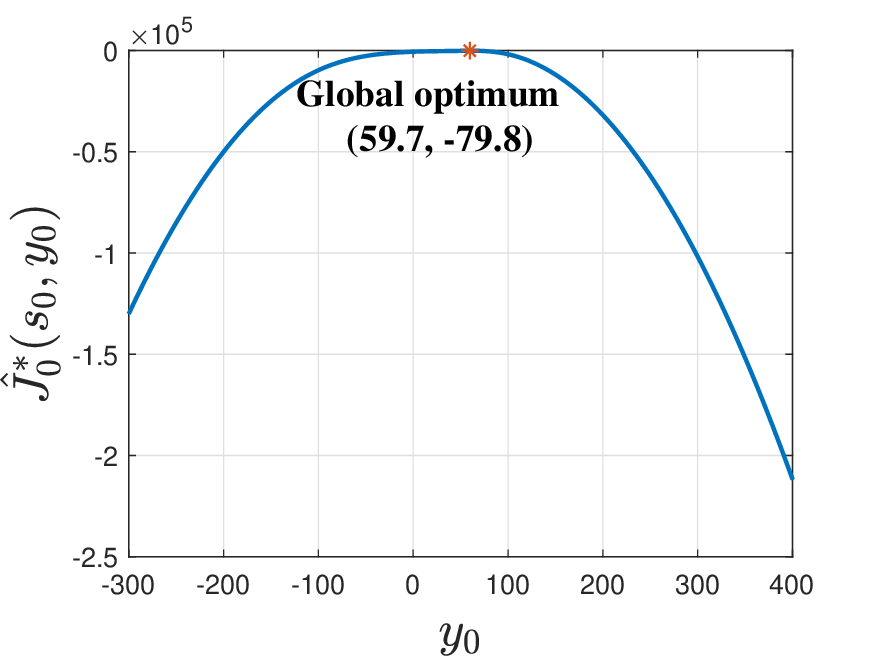}
        \end{minipage}
    }
    \subfigure[$s_0=3$]{
        \begin{minipage}[b]{0.22\textwidth}
            \includegraphics[width=1\textwidth]{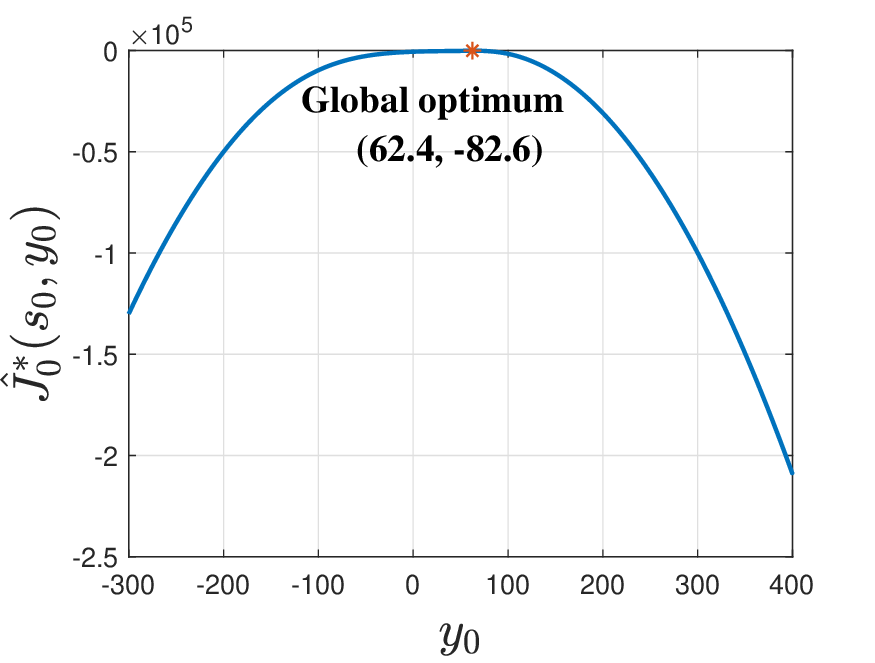}
        \end{minipage}
    }

    \subfigure[$s_0=4$]{
        \begin{minipage}[b]{0.22\textwidth}
            \includegraphics[width=1\textwidth]{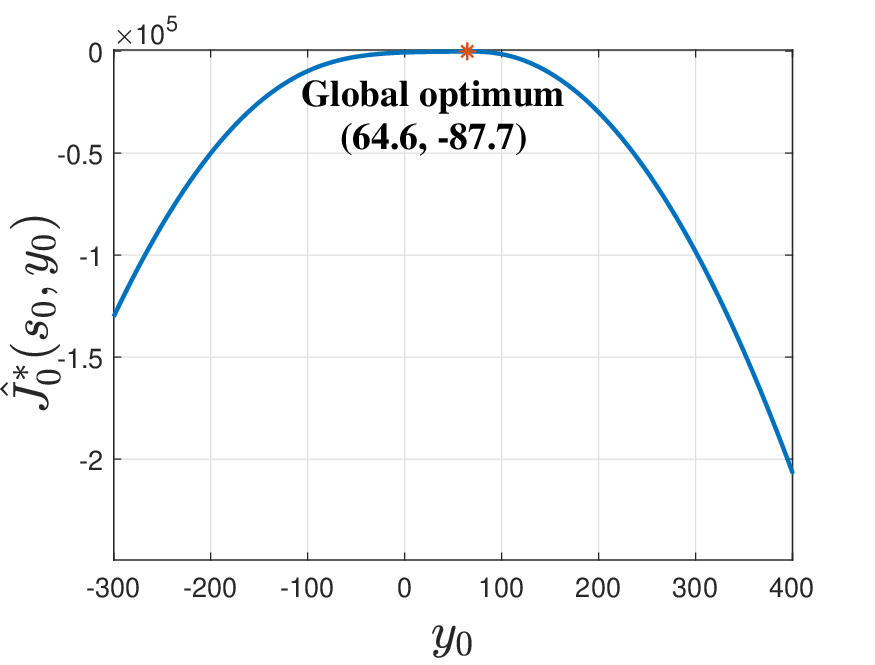}
        \end{minipage}
    }
    \subfigure[$s_0=5$]{
        \begin{minipage}[b]{0.22\textwidth}
            \includegraphics[width=1\textwidth]{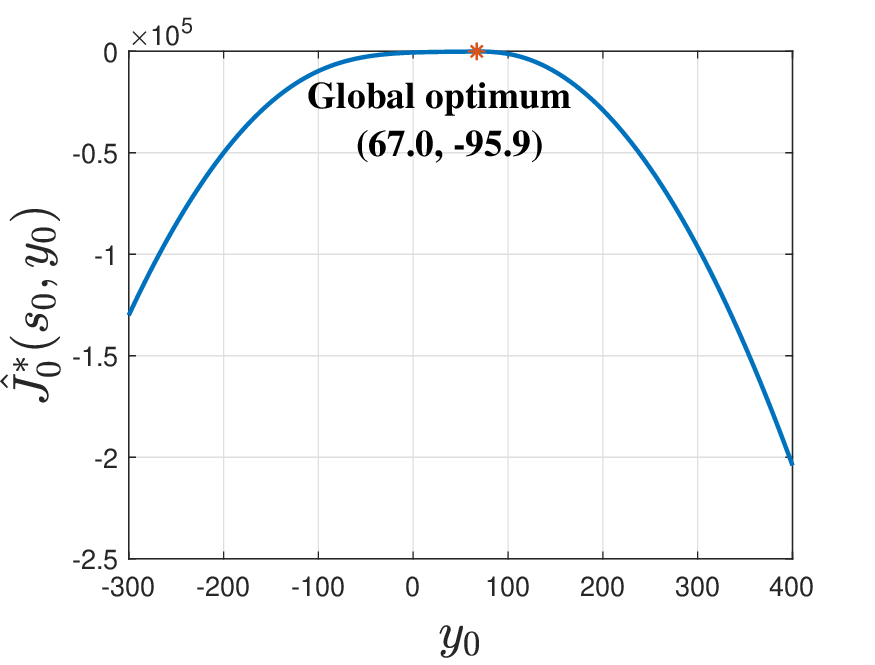}
        \end{minipage}
    }
    \subfigure[$s_0=6$]{
        \begin{minipage}[b]{0.22\textwidth}
            \includegraphics[width=1\textwidth]{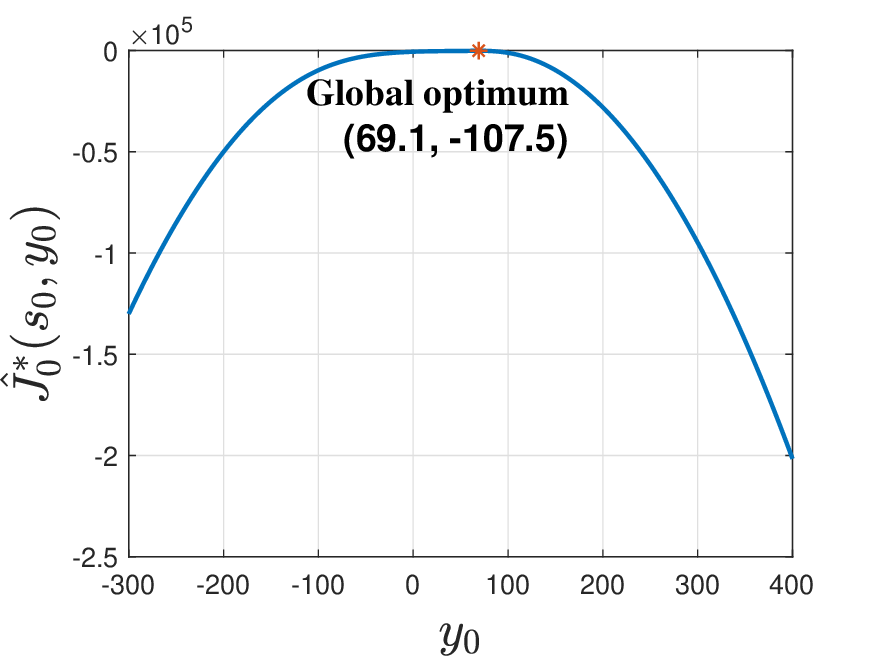}
        \end{minipage}
    }
    \subfigure[$s_0=7$]{
        \begin{minipage}[b]{0.22\textwidth}
            \includegraphics[width=1\textwidth]{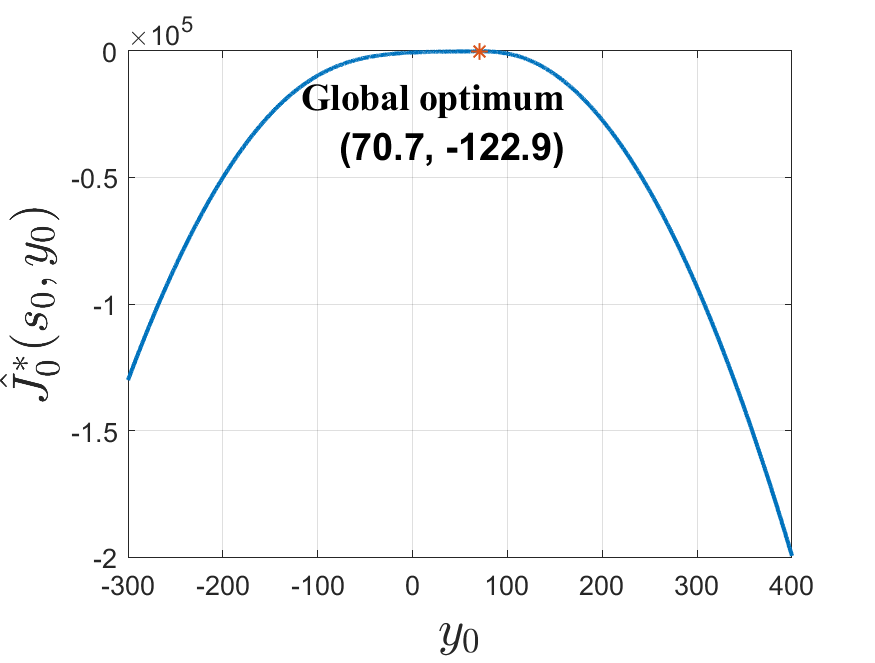}
        \end{minipage}
    }
    \subfigure[$s_0=8$]{
        \begin{minipage}[b]{0.22\textwidth}
            \includegraphics[width=1\textwidth]{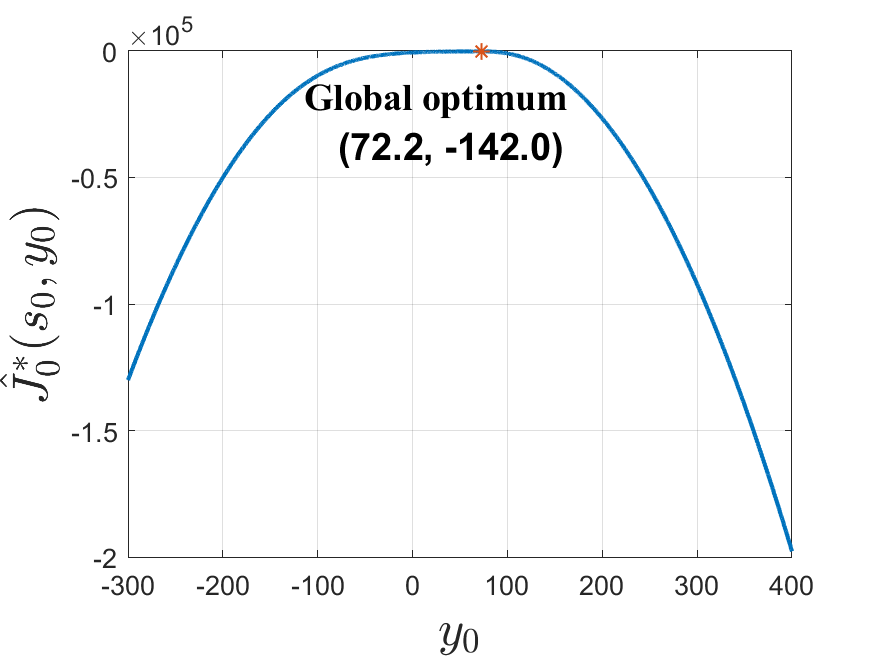}
        \end{minipage}
    }
    \subfigure[$s_0=9$]{
        \begin{minipage}[b]{0.22\textwidth}
            \includegraphics[width=1\textwidth]{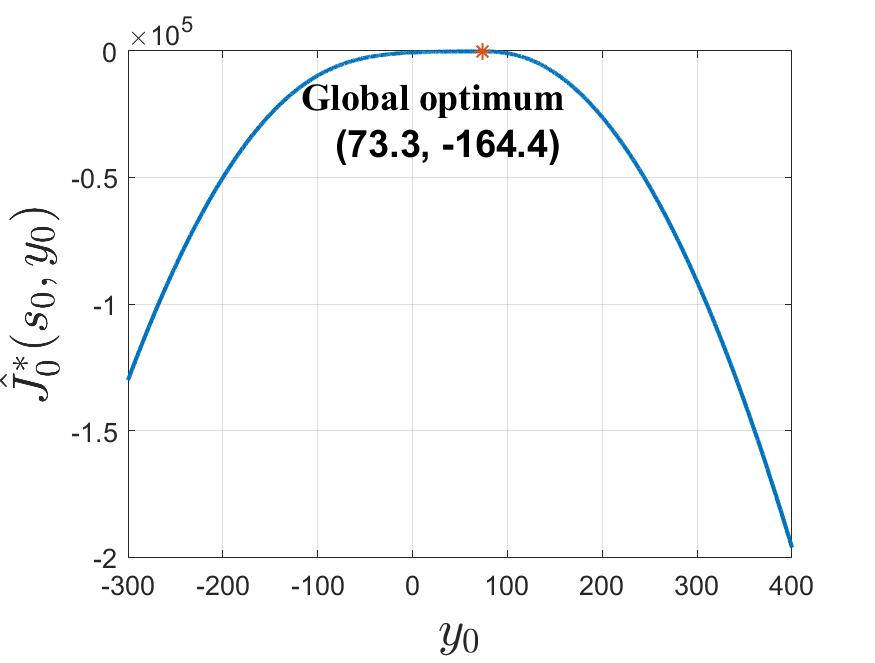}
        \end{minipage}
    }
    \subfigure[$s_0=10$]{
        \begin{minipage}[b]{0.22\textwidth}
            \includegraphics[width=1\textwidth]{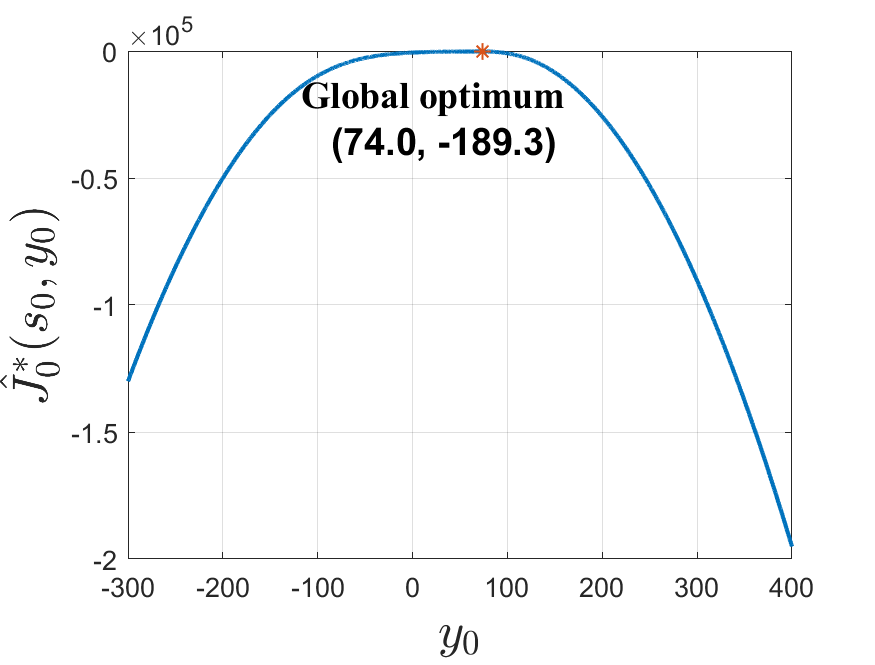}
        \end{minipage}
    }
    \caption{Curves of the optimal pseudo mean-variance $\hat J_0^*(s_0,y_0)$ with respect to $y_0 \in
        [-300,400]$, computed by the grid search method.} \label{fig:pmv_inv}
\end{figure*}

Note that there actually exist multiple local optima on the curves
of Figure~\ref{fig:pmv_inv}, but their values are quite close
(please refer to the refined illustration in
Figure~\ref{fig:refine_inv}), which also hints that local optimum is
usually good enough in practice. The optimal value function $J_0^*$
and the corresponding mean $y_0^*$ and variance $\sigma_0^*$ at each
initial state $s_0$ are presented in Table~\ref{tab:glo_opt}, where
we observe that the optimal mean and variance are both increasing in
the initial stock~$s_0$.
\renewcommand{\arraystretch}{1.0}
 \begin{table}[htbp]
    \newcommand{\tabincell}[2]{\begin{tabular}{@{}#1@{}}#2\end{tabular}}
    \centering
    \fontsize{4}{7}\selectfont
  % \small
    \caption{Global optima of the mean-variance inventory management problem by grid search.}
    \label{tab:glo_opt}
    \small
\begin{tabular}{|p{10mm}<{\centering}|p{9mm}<{\centering}|p{9mm}<{\centering}|p{9mm}<{\centering}|p{9mm}<{\centering}|p{9mm}<{\centering}|p{9mm}<{\centering}|p{10mm}<{\centering}|p{10mm}<{\centering}|p{10mm}<{\centering}|p{10mm}<{\centering}|p{10mm}<{\centering}|}
        \hline
        $s_0$ & \tabincell{c}{$0$}& \tabincell{c}{$1$} & \tabincell{c}{$2$}
        & \tabincell{c}{$3$}& \tabincell{c}{$4$} & \tabincell{c}{$5$} & \tabincell{c}{$6$} & \tabincell{c}{$7$}
        & \tabincell{c}{$8$}& \tabincell{c}{$9$} & \tabincell{c}{$10$} \\
        \hline
        $y_0^*$&54.4    &57.2    &59.7   &62.4    &64.6    &67.0  &69.1 &70.7 & 72.2 &73.3 &74.0    \\ \hline
        $\sigma_0^*(s_0)$&$67.35$    &$68.1$   &$69.75$  &$72.5$    &$76.15$    &$81.45$  &$88.3$ &$96.8$ & $107.1$ &$118.85$ &$131.65$    \\ \hline
        $J_0^*(s_0)$&$-80.3$    &$-79.0$   &$-79.8$  &$-82.6$    &$-87.7$    &$-95.9$  &$-107.5$ &$-122.9$ & $-142.0$ &$-164.4$ &$-189.3$    \\ \hline
\end{tabular}
\end{table}

Next, we apply Algorithm~\ref{alg:local} to iteratively solve this
problem. We choose different initial pseudo mean $y_0^{(0)}$ with
values $-500,-50,0,60,500$ to study the convergence of
Algorithm~\ref{alg:local}, which is illustrated by
Figure~\ref{fig:conver_inv}.
We observe that Algorithm~\ref{alg:local} always converges, but may
converge to different optima in some cases. Specifically, for
initial states $s_0=0,1,2,3,4$, Algorithm~\ref{alg:local} always
converges to the global optimum, while for initial states
$s_0=5,6,7,8,9,10$, it may not under some initial pseudo mean
$y_0^{(0)}$. In order to further verify whether the convergence
points are local optima, we choose three initial states $s_0=5,8,10$
and refine the illustration of pseudo mean-variance $\hat
J_0^*(s_0,y_0)$ in the neighborhood of the convergence points, as
illustrated in Figure~\ref{fig:refine_inv}. The curves (in numerical
values) show that all the convergence points are truly local optima.
It is also observed from Figure~\ref{fig:conver_inv} that when we
choose $y_0^{(0)}=500$, Algorithm~\ref{alg:local} always converges
to the global optimum from every initial state.

\begin{figure*}[htbp]
    \subfigure[$s_0=0$]{
        \begin{minipage}[b]{0.22\textwidth}
            \includegraphics[width=1\textwidth]{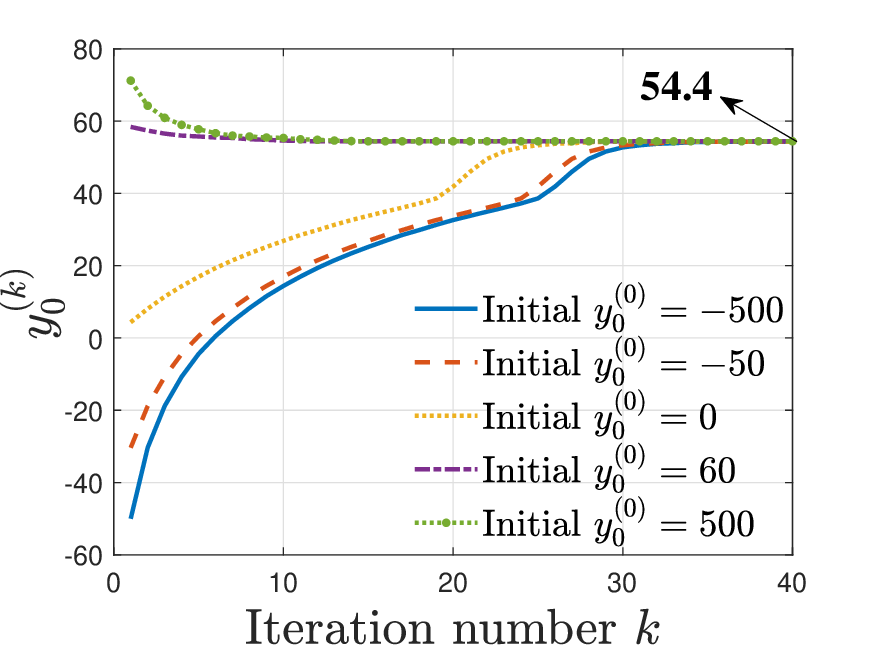}

        \end{minipage}
    }
    \subfigure[$s_0=1$]{
        \begin{minipage}[b]{0.22\textwidth}
            \includegraphics[width=1\textwidth]{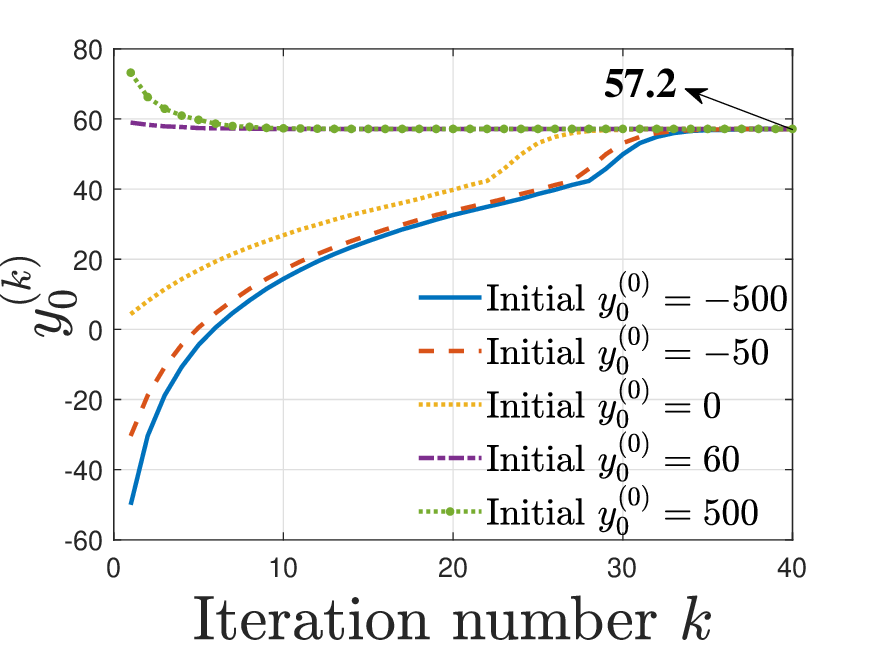}
        \end{minipage}
    }
    \subfigure[$s_0=2$]{
        \begin{minipage}[b]{0.22\textwidth}
            \includegraphics[width=1\textwidth]{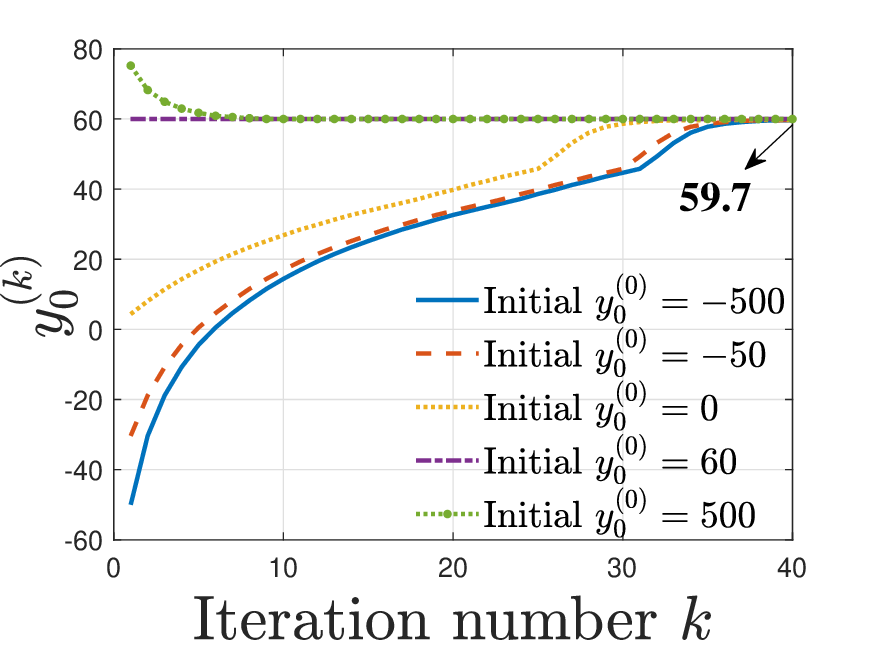}
        \end{minipage}
    }
     \subfigure[$s_0=3$]{
        \begin{minipage}[b]{0.22\textwidth}
            \includegraphics[width=1\textwidth]{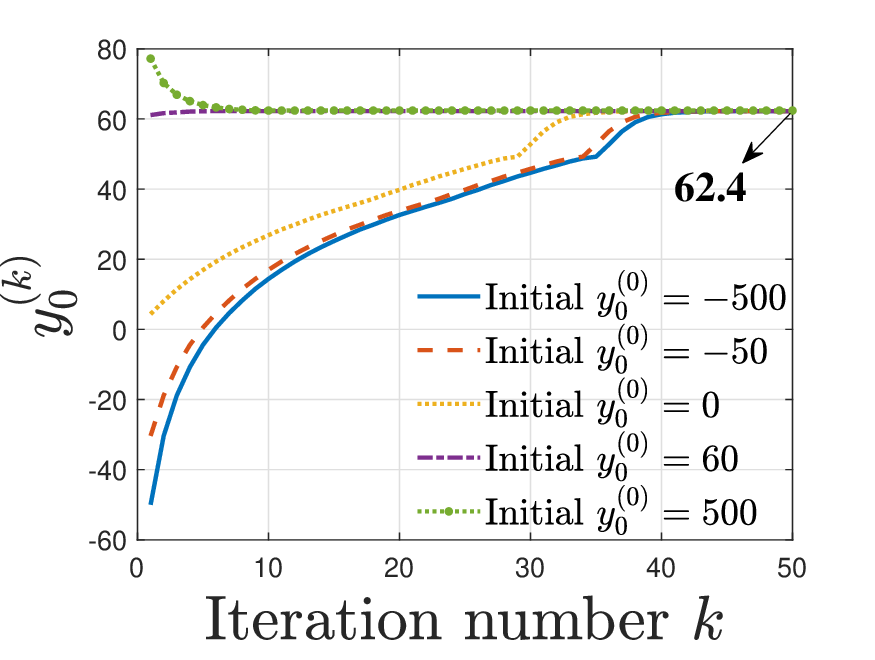}
        \end{minipage}
    }
     \subfigure[$s_0=4$]{
        \begin{minipage}[b]{0.22\textwidth}
            \includegraphics[width=1\textwidth]{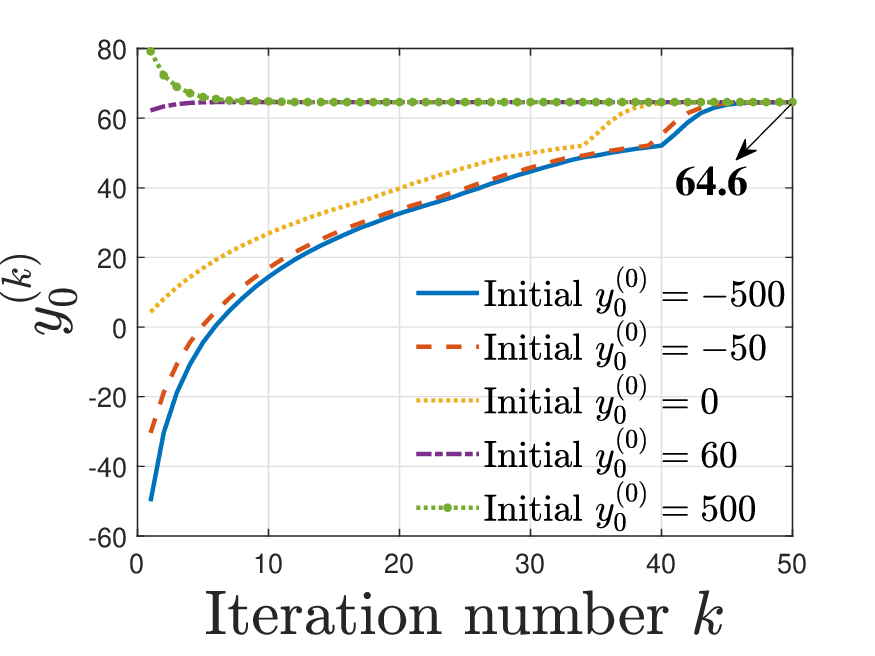}
        \end{minipage}
    }
     \subfigure[$s_0=5$]{
        \begin{minipage}[b]{0.22\textwidth}
            \includegraphics[width=1\textwidth]{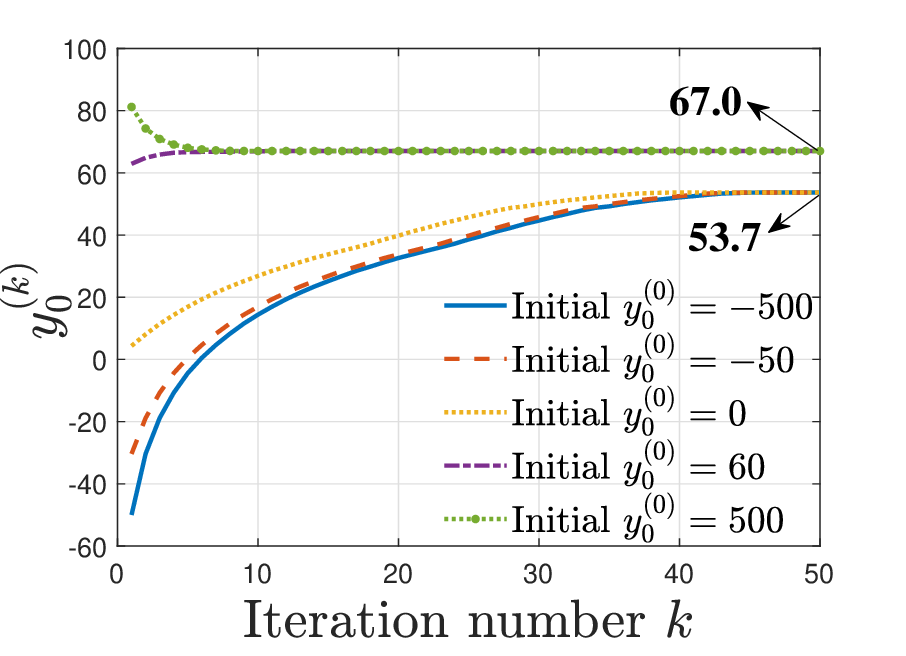}
        \end{minipage}
    }
     \subfigure[$s_0=6$]{
        \begin{minipage}[b]{0.22\textwidth}
            \includegraphics[width=1\textwidth]{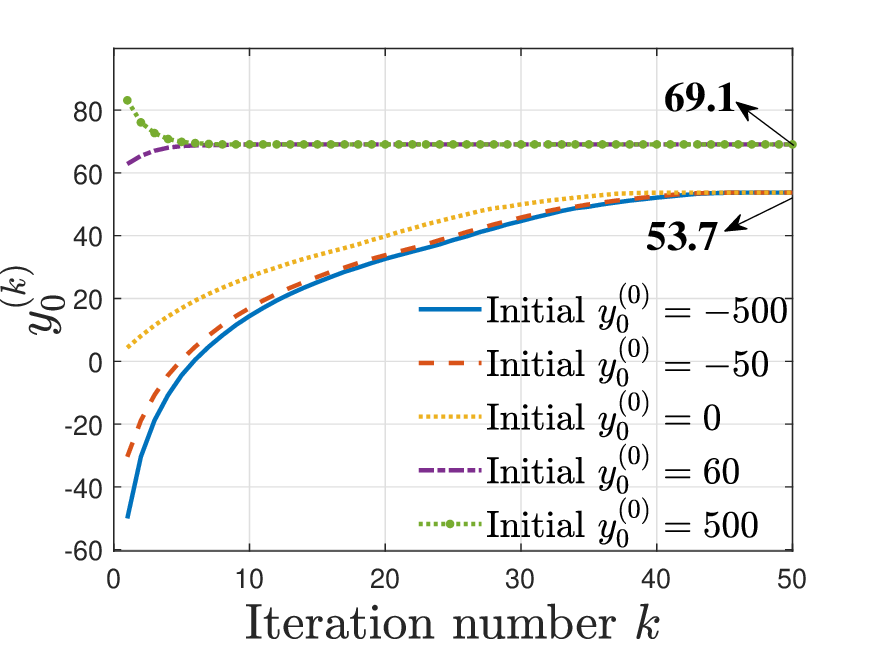}
        \end{minipage}
    }
     \subfigure[$s_0=7$]{
        \begin{minipage}[b]{0.22\textwidth}
            \includegraphics[width=1\textwidth]{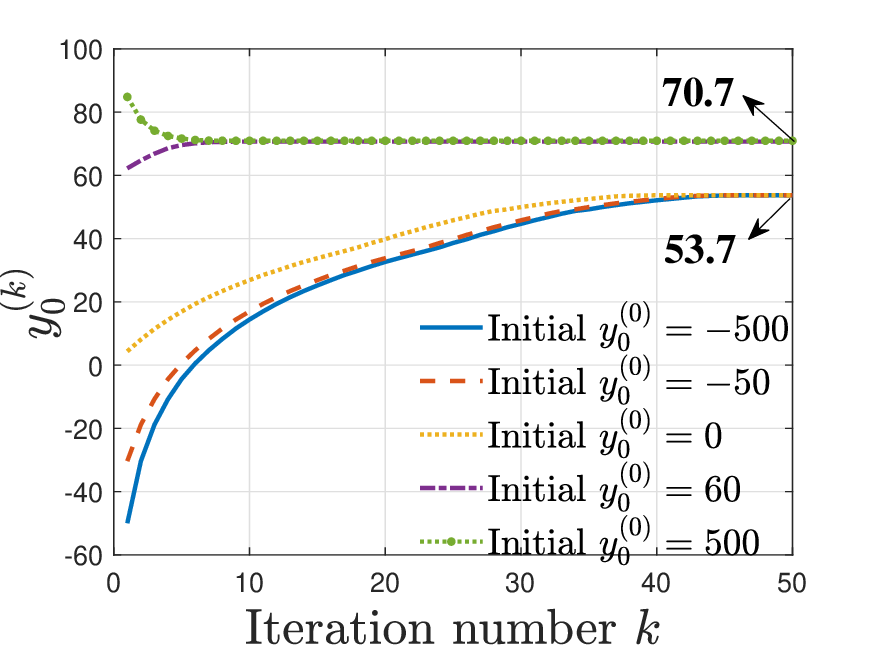}
        \end{minipage}
    }
     \subfigure[$s_0=8$]{
        \begin{minipage}[b]{0.22\textwidth}
            \includegraphics[width=1\textwidth]{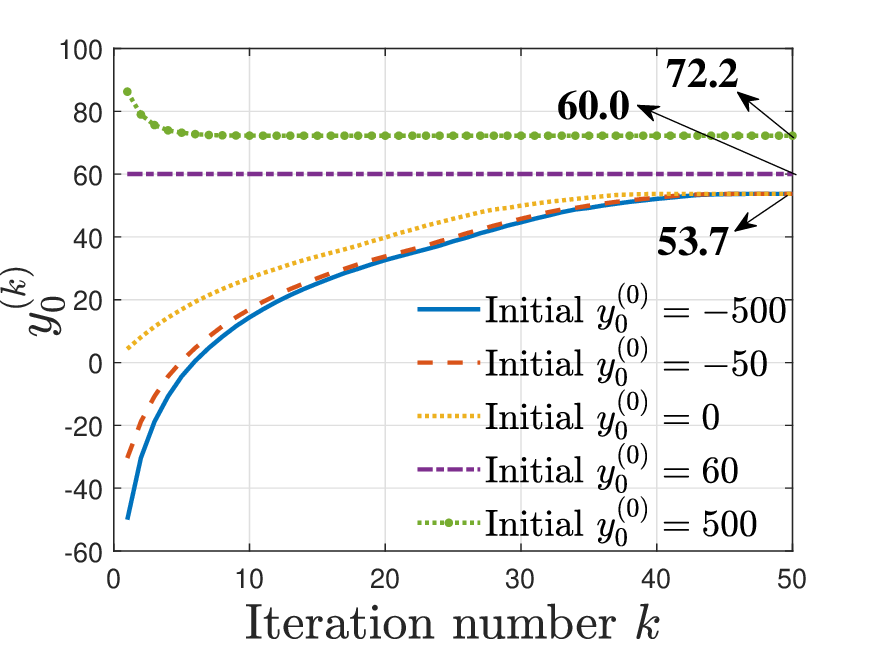}
        \end{minipage}
    }
     \subfigure[$s_0=9$]{
        \begin{minipage}[b]{0.22\textwidth}
            \includegraphics[width=1\textwidth]{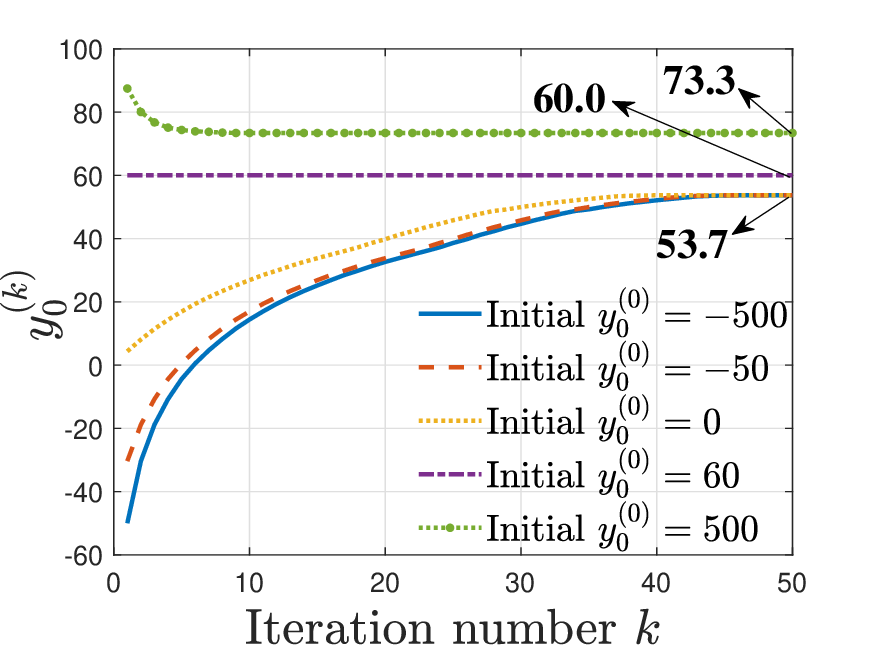}
        \end{minipage}
    }
     \subfigure[$s_0=10$]{
        \begin{minipage}[b]{0.22\textwidth}
            \includegraphics[width=1\textwidth]{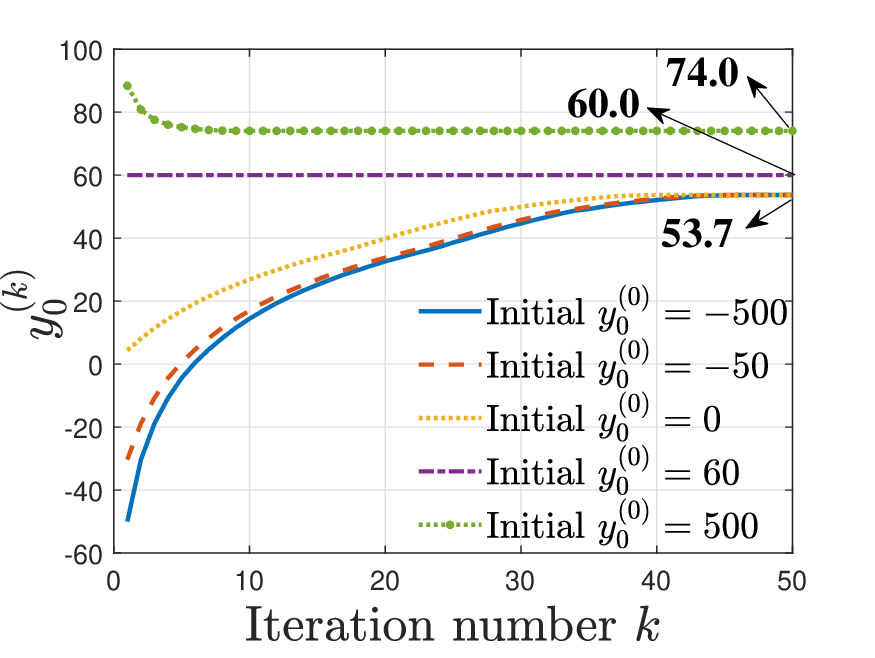}
        \end{minipage}
    }
    \caption{Convergence processes of pseudo mean $y^*_0$ in Algorithm~\ref{alg:local} under different initial values of $y^{(0)}_0$ and $s_0$. } \label{fig:conver_inv}
\end{figure*}

\begin{figure*}[htbp]
    \subfigure[$s_0=5$]{
        \begin{minipage}[b]{0.31\textwidth}
            \includegraphics[width=1\textwidth]{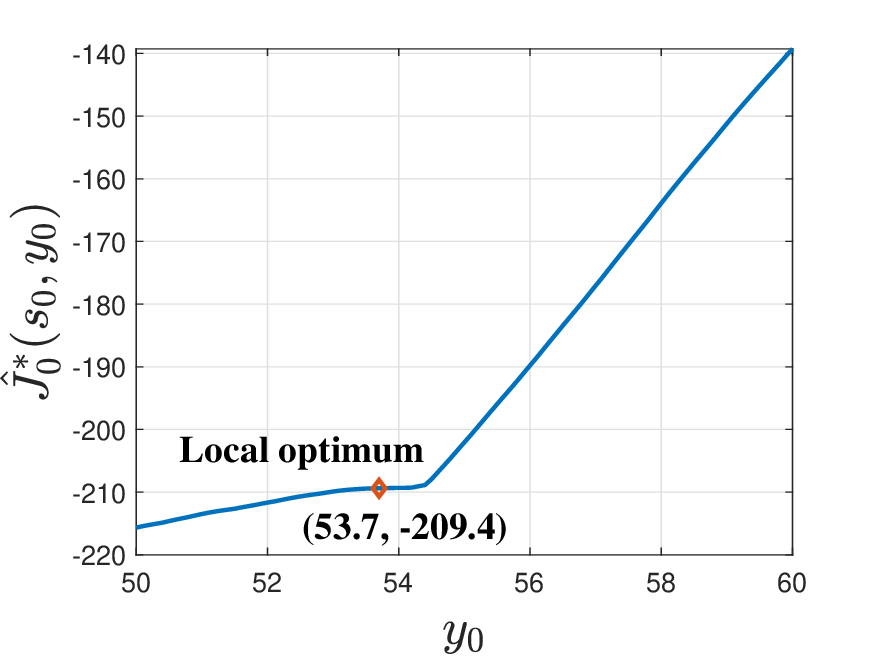}
        \end{minipage}
    }
    \subfigure[$s_0=8$]{
        \begin{minipage}[b]{0.31\textwidth}
            \includegraphics[width=1\textwidth]{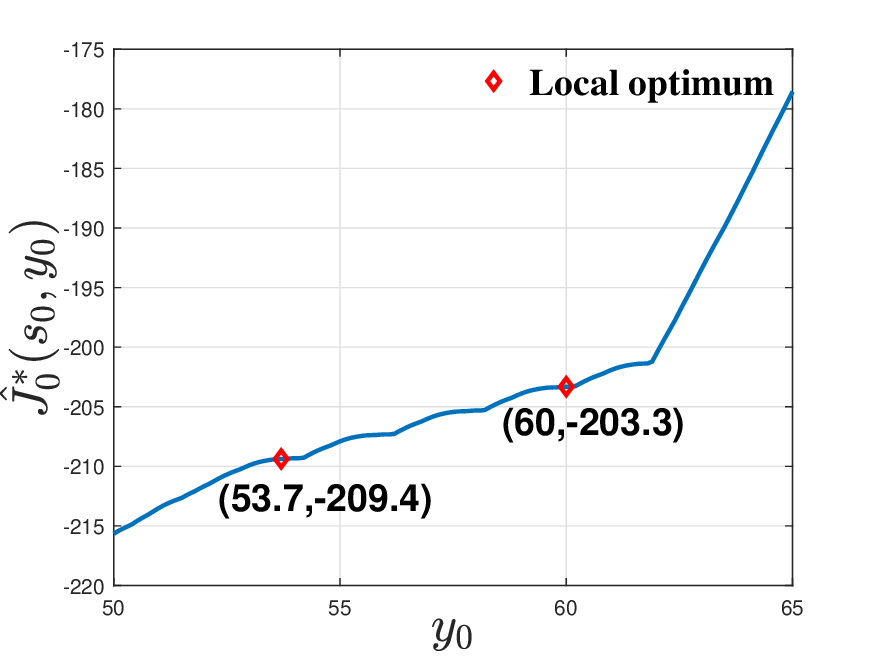}
        \end{minipage}
    }
    \subfigure[$s_0=10$]{
        \begin{minipage}[b]{0.31\textwidth}
            \includegraphics[width=1\textwidth]{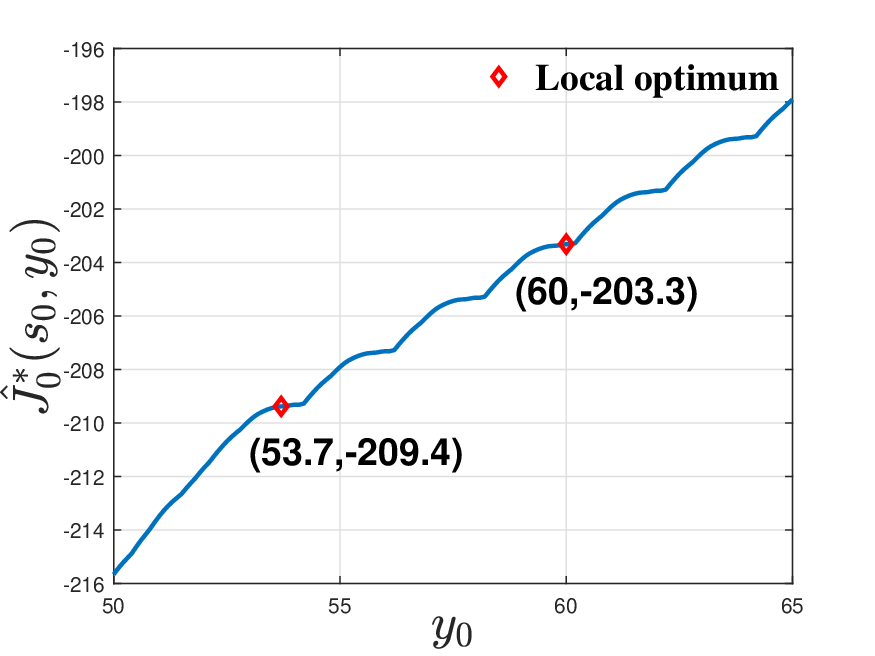}
        \end{minipage}
    }
    \caption{Refined illustration of pseudo mean-variance $\hat J_0^*(s_0,y_0)$ with respect to
        $y_0$ under initial states $s_0=5,8,10$.} \label{fig:refine_inv}
\end{figure*}

The numerical results in this example demonstrate that
Algorithm~\ref{alg:local} can find locally optimal policies of
multi-period mean-variance inventory management problem. We may
further find the globally optimal policy by taking proper initial
values or using some perturbation techniques widely adopted in
evolutionary algorithms. Moreover, this example also shows that
these local optima have quite close values, which hints that even
the local convergence of Algorithm~\ref{alg:local} may be good
enough in practical applications.

\section{Conclusion}\label{sec:con}
In this paper, we study the optimization and algorithm for
finite-horizon discrete-time MDPs with a mean-variance optimality
criterion. The objective is to maximize the combined mean-variance
metric of accumulated rewards among history-dependent randomized
policy space. By introducing concepts called pseudo mean and pseudo variance, we convert the MV-MDP to a bilevel MDP, where the inner pseudo MV-MDP is equivalent to a standard
finite-horizon MDP with an augmented state space and the outer level
is a single parameter optimization problem with respect to the
pseudo mean. The properties of this MV-MDP, including the optimality
of history-dependent deterministic policies and the piecewise
quadratic concavity of the optimal values of inner MDPs with respect
to the pseudo mean, are derived. Based on these properties, we
develop a policy iteration type algorithm to effectively solve this
finite-horizon MV-MDP, which alternatingly optimizes the inner
policy and the outer pseudo mean. The convergence and the local
optimality of the algorithm are proved. We further derive a
sufficient condition under which our algorithm can converge to the
global optimum. Finally, we apply this approach to study the
mean-variance optimization of multi-period portfolio selection,
queueing control, and inventory management, which demonstrate that
our approach can find the optimum effectively.

One of the future research topics is to extend the global
convergence condition with the help of sensitivity analysis on
pseudo mean. On the other hand, it is of significance to further
study infinite-horizon MV-MDPs, including discounted MV-MDPs and
limiting average MV-MDPs, i.e., the mean-variance optimization of
discounted accumulated rewards and the limiting average
mean-variance optimization of total accumulated rewards. Moreover,
the combination of our approach with the technique of reinforcement
learning is also a promising research topic, which can contribute to
develop a framework of data-driven risk-sensitive decision making.

{}

\newpage

\begin{appendices}

\section{Proof of Theorems}

\subsection{Proof of Theorem~\ref{thm:relation}}
\begin{proof}
    Given $y_0 \in \mathbb{R}$ and
    $\tilde{u}=(\tilde{u}_t;t \in \mathcal T) \in \tilde{\mathcal U}^{\rm HR}$,
    we define a policy $u = (u_t;t \in \mathcal T) \in \mathcal U^{\rm HR}$
    as follows.
    \begin{eqnarray*}
        u_0(\cdot|s_0) &:=& \tilde{u}_0(\cdot|s_0,y_0),\\
        u_1(\cdot|s_0,a_0,s_1) &:=& \tilde{u}_1(\cdot|s_0,y_0,a_0,s_1,y_0-r_0(s,a_0)),\\
        &\cdots& \\
        u_t(\cdot|s_0,a_0,\ldots,s_t) &:=& \tilde{u}_t(\cdot|s_0,y_0,a_0,s_1,y_0-r_0(s,a_0),a_1,\ldots,s_t,y_0-\sum\limits_{\tau=0}^{t-1}r_\tau(s_\tau,a_\tau)).
        %&\cdots
    \end{eqnarray*}
    In this sense, the two policies $u$ and $\tilde{u}$ share the same
    decision rule, which implies (\ref{equ:relation}).

    Based on (\ref{equ:relation}), it holds for each $(s_0,y_0) \in
    \tilde{\mathcal{S}}$ and $\tilde{u} \in \tilde{\mathcal U}^{\rm HR}$ that
    \begin{equation*}
        V_0^{\tilde{u}}(s_0,y_0) = \hat J_0^{u}(s_0,y_0) \le
        \sup\limits_{u \in \mathcal U^{\rm HR}} \hat J_0^{u}(s_0,y_0) = \hat
        J_0^{*}(s_0,y_0).
    \end{equation*}
    On the other hand, since the policy space $\tilde{\mathcal U}^{\rm HR}$
    contains $\mathcal U^{\rm HR}$, necessarily we have
    \begin{equation*}
        V_0^{*}(s_0,y_0) = \sup\limits_{\tilde{u} \in \tilde{\mathcal
                U}^{\rm HR}}  V_0^{\tilde{u}}(s_0,y_0) \ge \sup\limits_{u \in \mathcal
            U^{\rm HR}}  V_0^{u}(s_0,y_0) = \sup\limits_{u \in \mathcal U^{\rm HR}} \hat
        J_0^{u}(s_0,y_0) = \hat J_0^{*}(s_0,y_0).
    \end{equation*}
    The above two inequalities lead to (\ref{equ:opt_relation}).
\end{proof}

\subsection{Proof of Theorem~\ref{thm:opt}}
\begin{proof}
    The results of \eqref{equ:opt_dp} and part~(a) are derived directly
    by following Theorems~4.3.2 and 4.3.3 of \cite{Puterman94}. And
    part~(b) is a direct corollary of part~(a) and
    Theorem~\ref{thm:relation}.
\end{proof}

\subsection{Proof of Theorem~\ref{thm:opt_policy}}
\begin{proof}
    To prove the theorem, we only need to prove $J_0^{{u}^{*}}(s_0) \ge
    J_0^*(s_0)$ for all $s_0 \in \mathcal{S}$. Given an $s_0 \in
    \mathcal{S}$, we have
         \begin{eqnarray*}
            J_0^{{u}^{*}}(s_0)
            &=& \max\limits_{y_0 \in \mathbb{R}}\hat J_0^{{u}^{*}}(s_0,y_0) \\
            &\ge& \hat J_0^{{u}^{*}}(s_0,y^*_0)\\
            &=&  V_0^{\tilde{u}^{*}}(s_0,y^*_0)\\
            &=&  V_0^{*}(s_0,y^*_0)\\
            &=& \hat J_0^{*}(s_0,y^*_0)\\
            &=& J_0^{*}(s_0),
        \end{eqnarray*}
    where the first equality follows from the variance property (\ref{eq_pseudovar}), the second and fourth
    equalities are ensured by Theorem~\ref{thm:relation}, the third and
    last equalities use the fact that $\tilde{u}^*$ and $y^*_0$
    attain the maximum of (\ref{equ:newmdp}) and (\ref{equ:outer}),
    respectively.
\end{proof}

\subsection{Proof of Theorem~\ref{thm:divid}}
\begin{proof}
We prove the theorem by contradiction. Suppose that there exists an
interval $(y_1^0,y_2^0) \in \mathcal Y$ such that for any
sub-interval $ \mathcal Y_{\rm sub} \subset (y_1^0,y_2^0)$, the
pseudo MV-MDPs $\left\{\hat{\mathcal M}(y): y \in \mathcal Y_{\rm
sub} \right\}$  does not have the same optimal policy.

Specifically, for the interval $(y_1^0,y_2^0) \in \mathcal Y$, there
exist $y_1^1, y_2^1$ with $y_1^0<y_1^1<y_2^1<y_2^0$ such that pseudo
MV-MDPs $\hat{\mathcal M}(y_1^1) ~\text{and}~ \hat{\mathcal
M}(y_2^1)$  have no common optimal policy, we assume that
$\hat u_{*i}^{1} \in \mathcal U^{\rm HD}$ is the optimal policy for
the pseudo MV-MDP $\hat{\mathcal M}(y_i^1)$. Using the same
argument, we obtain an increasing sequence $\left\{y_1^n; n \ge
0\right\}$ and a decreasing sequence $\left\{y_2^n; n \ge 0\right\}$
with $y_1^n < y_2^n$, where we denote $\hat  u_{*i}^{n} \in \mathcal
U^{\rm HD}$ as an optimal policy for pseudo MV-MDP $\hat{\mathcal
M}(y_i^n)$. Since $\mathcal U^{\rm HD}$ is finite, there exist two
sub-sequences $\left\{y_1^{k_n}; n \ge 0\right\} \subset
\left\{y_1^{n}; n \ge 0\right\}, \left\{y_2^{k_n}; n \ge 0\right\}
\subset \left\{y_2^{n}; n \ge 0\right\}$ and two policies $\hat
u_*^{1},\hat u_*^{2} \in \mathcal U^{\rm HD}$ such that $\hat
u_*^{i}$ is optimal for pseudo MV-MDPs $\left\{\hat{\mathcal
M}(y_i^{k_n}); n \ge 0\right\}$, i.e.,
\begin{equation}\label{equ:seque_opt}
        \hat J_0^{\hat u_*^{i}}(s_0,y_i^{k_n}) = \hat J_0^*(s_0,y_i^{k_n}),\quad\forall i=1,2, n \ge 0.
\end{equation}
We denote by $\hat y_i=\lim\limits_{n \rightarrow \infty}y_i^{k_n},
i=1,2$,    without loss of generality, we assume $\hat y_1=\hat
y_2:=y_1$ (otherwise, replace $(y_1^0,y_2^0)$ with $(\hat y_1,\hat
y_2)$). Taking $n \rightarrow \infty$ in the LHS and RHS of
(\ref{equ:seque_opt}), we obtain
\begin{equation}\label{equ:lim_opt}
        \hat J_0^{\hat u_*^{i}}(s_0,y_1) = \hat J_0^*(s_0,y_1),\quad\forall i=1,2
\end{equation}
based on the continuity of $\hat J_0^{\hat u_*^{i}}(s_0,\cdot)$ and
$\hat J_0^{*}(s_0,\cdot)$. (\ref{equ:lim_opt}) implies that $\hat
u_*^{1} ~\text{and}~ \hat u_*^{2}$ are both optimal policies of
pseudo MV-MDP $\hat{\mathcal M}(y_1)$.

Now, we turn to interval $(y_1^0,\hat y_1) \subset (y_1^0,y_2^0)$
and use the same argument, there exists $y_2 \in (y_1^0,\hat y_1)$
such that the pseudo MV-MDP $\hat{\mathcal M}(y_2)$ has at least two
different optimal policies. Repeat this process, there exists an
infinite sequence $\left\{y_n,n \ge 1\right\}$ such that each
pseudo MV-MDP $\hat{\mathcal M}(y_n)$ has at least two different
optimal policies, which contradicts the finite deterministic policy
space $\mathcal U^{\rm HD}$. Therefore, Theorem~\ref{thm:divid}
holds.
\end{proof}

\subsection{Proof of Theorem~\ref{thm:piece_conca}}
\begin{proof}
Given $s_0 \in \mathcal S$ and $u \in \mathcal U^{\rm HD}$, the
pseudo mean-variance of (\ref{eq_pseudoMV}) can be rewritten as
\begin{equation}\nonumber
        \hat J_0^u(s_0,y_0) = -\lambda \cdot y_0^2 + 2\lambda \mathbbm{E}_{s_0}^u
        [R_{0:T}] \cdot y_0 + \mathbbm{E}_{s_0}^u [R_{0:T}-\lambda R_{0:T}^2],
\end{equation}
which is obviously a quadratic concave function of $y_0$.

According to Theorem~\ref{thm:divid}, if $y_0 \in (y^k,y^{k+1}]$ for
some $k$, we have $\hat J_0^{*}(s_0,y_0)=\hat J_0^{\hat
u^k_*}(s_0,y_0)$, where the optimal policy $\hat u^k_*$ remains
unvaried for any $y_0 \in (y^k,y^{k+1}]$ and the assoicated $\hat
J_0^{*}(s_0,y_0)$ is quadratic concave with respect to $y_0$.
Therefore, $\hat J_0^{*}(s_0,y_0)$ is piecewise quadratic concave
and is divided into concave segments by the break points
$\left\{y^1,\ldots,y^n\right\}$.
\end{proof}

\subsection{Proof of Theorem~\ref{thm:converge}}\label{app:thm_conver}
We prove this theorem from the perspective of sensitivity-based
optimization theory, which has been used for optimizing long-run
(mean-)variance MDPs \citep{xia2016,xia2020} and discounted variance
MDPs \citep{xia18}.
    We first introduce the \emph{performance difference
        formula} and the \emph{performance derivative formula} for
    finite-horizon standard MDPs\citep{jia2011solving,zhao2010optimization}.
    Consider a finite-horizon standard MDP, where the objective is to maximize the
    expectation of $T$-horizon accumulative rewards among
    Markov randomized policies,
    \begin{equation}
        \nonumber
        \mu_0^*(s_0):=\max\limits_{u \in \mathcal U^{\rm MR}} \mu_0^u(s_0),\quad s_0 \in \mathcal S.
    \end{equation}

    For any Markov randomized policy $u=(u_t; t \in \mathcal T) \in
    \mathcal{U}^{\rm MR}$, we denote
    \begin{equation}
        \nonumber
        r_{u_t}(s)=\sum\limits_{a \in \mathcal A(s)}r_t(s,a)u_t(a|s), \quad t \in \mathcal T, s \in \mathcal S,
    \end{equation}
    \begin{equation}
        \nonumber
        P_{u_t}(s'|s)=\sum\limits_{a \in \mathcal A(s)}P_t(s'|s,a)u_t(a|s), \quad t \in \mathcal T, s,s' \in \mathcal S,
    \end{equation}
    and let $\bm r_{u_t}, \bm P_{u_t}$ be the corresponding vector form
    and matrix form, respectively. For notational simplicity, we omit
    $u$ and use $\bm r_{t}, \bm P_{t}$ and $\mu_0(s_0)$ to represent
    $\bm r_{u_t}, \bm P_{u_t}$ and the mean $\mu_0^u(s_0)$,
    respectively. We also use the superscript ``$'$" to indicate the
    parameters under Markov randomized policy $u'=(u'_t; t \in \mathcal
    T) \in \mathcal{U}^{\rm MR}$, and use ``$\delta$" to indicate the
    parameters under mixed policy $\delta_u^{u'}=(1-\delta)u+\delta u'$.

    In what follows, we give the {performance difference
        formula}, the {performance derivative formula} and  the \emph{optimality condition}  for
    finite-horizon standard MDPs, as stated in Lemmas \ref{lem:pdf}-\ref{lem:opt_con}.
    \begin{lem}\label{lem:pdf}
        (\textbf{Performance Difference Formula})

        For any two Markov policies $u=(u_t; t \in \mathcal T), u'=(u'_t; t
        \in \mathcal T)$, and initial state $s_0 \in \mathcal S$, we derive
        the performance difference between $\mu_0(s_0)$ and $\mu_0'(s_0)$ as
        follows,
        \begin{equation}\label{equ:pdf}
            \mu_0'(s_0)-\mu_0(s_0) = \bm{e}_{s_0}\sum_{t=0}^{T-1}
            \prod_{\tau=0}^{t-1} \bm P'_\tau\left[\bm r'_t-\bm r_t+\left(\bm
            P'_t-\bm P_t\right)\bm g_{t+1}\right],
        \end{equation}
        where $\bm
        g_{t+1}=\sum\limits_{k=t+1}^{T-1}\prod\limits_{\tau=t+1}^{k-1}\bm
        P_\tau \bm r_k$, $\bm P'_{T-1}=\bm P_{T-1}=\bm{I}$, $\bm{I}$
        denotes the identity matrix, and $\bm{e}_{s_0}$ denotes the unit row
        vector with $\bm{e}_{s_0}(s_0)=1$.
    \end{lem}
    \begin{lem}\label{lem:deriv}
        (\textbf{Performance Derivative Formula})

        Given two Markov policies $u=(u_t; t \in \mathcal T), u'=(u'_t; t
        \in \mathcal T)$, the initial state $s_0 \in \mathcal S$ and a
        constant $\delta \in [0,1]$, the performance derivative of the mean
        $\mu_0^\delta(s_0)$ at policy $u$ along direction $u'$ takes the
        following form
        \begin{equation}\label{equ:deriva}
            \frac{\partial \mu_0^\delta(s_0)}{\partial
                \delta}\Big|_{\delta=0}=\bm{e}_{s_0}\sum_{t=0}^{T-1}
            \prod_{\tau=0}^{t-1} \bm P_\tau\left[\bm r'_t-\bm r_t+\left(\bm
            P'_t-\bm P_t\right)\bm g_{t+1}\right].
        \end{equation}
    \end{lem}

    \begin{lem}\label{lem:opt_con}
    (\textbf{Optimality Condition})

    A Markov deterministic policy $u=(u_t; t \in \mathcal T) \in
    \mathcal{U}^{\rm MD}$ is an optimal policy if and only if it holds
    that for any $t \in \mathcal T, (s,a) \in \mathcal K$,
    \begin{align}\label{equ:opt_condition}
        \nonumber
            r_t(s,u_t(s))+\sum\limits_{s' \in \mathcal S}P_t(s'|s,u_t(s))g_{t+1}^{u}(s')\geq
            r_t(s,a)+\sum\limits_{s' \in \mathcal S}P_t(s'|s,a)g_{t+1}^{u}(s').
    \end{align}
    \end{lem}

With the aid of the performance difference
formula (\ref{equ:pdf}) and the performance derivative formula (\ref{equ:deriva}) for finite-horizon standard MDPs, we can also derive the performance difference
formula and the performance derivative formula for finite-horizon MV-MDPs.

According to Theorem~\ref{thm:relation}, for any
history-dependent randomized policy $u=(u_t; t \in \mathcal T) \in
\mathcal{U^{\rm HR}}$, there exists a history-dependent randomized
policy $\tilde{u}'=(\tilde u_t'; t \in \mathcal T)
\in \tilde{\mathcal{U}}^{\rm HR}$ such that
\begin{equation*}
V_0^{\tilde u'}(s_0,\mu_0^u(s_0))=  J_0^u(s_0),\quad\forall s_0 \in
\mathcal S.
\end{equation*}
By utilizing Theorem~5.5.1 of \cite{Puterman94}, we can further find
a Markov randomized policy $\tilde{u}=(\tilde u_t; t \in
\mathcal T) \in \tilde{\mathcal{U}}^{\rm MR}$ such that
\begin{equation*}
V_0^{\tilde u}(s_0,\mu_0^u(s_0))=V_0^{\tilde
u'}(s_0,\mu_0^u(s_0))=J_0^u(s_0),\quad\forall s_0 \in \mathcal S.
\end{equation*}
Therefore, each history-dependent randomized policy $u=(u_t; t \in \mathcal T) \in \mathcal{U^{\rm HR}}$ for finite-horizon
MV-MDP (\ref{equ:mv_mdp}) corresponds to a Markov randomized policy
$\tilde{u}=(\tilde u_t; t \in \mathcal T) \in
\tilde{\mathcal{U}}^{\rm MR}$ for finite-horizon standard MDP
(\ref{equ:newmdp}).

We follow the notations in the finite-horizon standard MDPs and use `$\sim$' to denote the parameters under policy $\tilde{u}$.
%and derive the performance difference
%formula and the performance derivative formula for finite-horizon MV-MDPs.

\begin{lem}
(\textbf{Performance Difference Formula for MV-MDPs})

For any two history-dependent randomized policies $u=(u_t; t \in \mathcal T),u'=(u'_t; t \in \mathcal T)$ and initial state $s_0
\in \mathcal S$, we derive the difference between the mean-variance
metrics $J_0(s_0)$ and $J_0'(s_0)$ as follows.
\begin{equation}\label{equ:pdf_mv}
J_0'(s_0)-J_0(s_0) = \bm{e}_{(s_0,\mu_0(s_0))}\sum_{t=0}^{T}
\prod_{\tau=0}^{t-1} \tilde{\bm P'}_\tau\left[\tilde{\bm
r}'_t-\tilde{\bm r}_t+\left(\tilde{\bm P'}_t-\tilde{\bm
P}_t\right)\tilde{\bm g}_{t+1}\right] + \lambda
(\mu_0'(s_0)-\mu_0(s_0))^2,
\end{equation}
where $\bm \tilde
g_{t+1}=\sum\limits_{k=t+1}^{T}\prod\limits_{\tau=t+1}^{k-1}\tilde{\bm
P}_\tau \tilde{\bm r}_k$, $\bm P'_{T}=\bm P_{T}=\bm{I}$, and
$\bm{e}_{(s_0,\mu_0(s_0))}$ denotes the unit row vector with
$\bm{e}_{(s_0,\mu_0(s_0))}(s_0,\mu_0(s_0))=1$.
\end{lem}
\begin{proof}
    According to the property of variance, the mean-variance $J_0(s_0)$ and the pseudo mean-variance $\hat J_0(s_0,y_0)$ have the following relation
    \begin{equation}\label{equ:pro_var}
        J_0(s_0)=\hat J_0(s_0,y_0)+\lambda(\mu_0(s_0)-y_0)^2.
    \end{equation}
    Therefore, we have
    \begin{eqnarray*}
        J_0'(s_0)-J_0(s_0)
        &=&\hat J_0'(s_0,\mu_0(s_0))-\hat J_0(s_0,\mu_0(s_0))+\lambda(\mu_0'(s_0)-\mu_0(s_0))^2\\
        &=&V_0'(s_0,\mu_0(s_0))-V_0(s_0,\mu_0(s_0))+\lambda(\mu_0'(s_0)-\mu_0(s_0))^2\\
        &=&\bm{e}_{(s_0, \mu_0(s_0))}\sum_{t=0}^{T} \prod_{\tau=0}^{t-1} \tilde{\bm P'}_\tau\left[\tilde{\bm r}'_t-\tilde{\bm r}_t+\left(\tilde{\bm P'}_t-\tilde{\bm P}_t\right)\tilde{\bm g}_{t+1}\right] + \lambda (\mu_0'(s_0)-\mu_0(s_0))^2,
    \end{eqnarray*}
    where the first equality follows from (\ref{eq_pseudoMV}) and (\ref{equ:pro_var}), the second equality is guaranteed by Theorem~\ref{thm:relation}, and the last equality follows directly from the performance difference formula (\ref{equ:pdf}).
\end{proof}

Similar to the proof of Lemma \ref{lem:deriv} and noting that
\begin{equation*}
\frac{\partial (\mu_0^\delta(s_0)-\mu_0(s_0))^2}{\partial
\delta}\Big|_{\delta=0} = 2 (\mu_0^\delta(s_0)-\mu_0(s_0))
\frac{\partial \mu_0^\delta(s_0)}{\partial \delta}\Big|_{\delta=0} =
0,
\end{equation*}
we derive the performance derivative formula for finite-horizon
MV-MDPs as below.

\begin{lem}\label{lem:deriv_mv}
(\textbf{Performance Derivative Formula for MV-MDPs})

Given two history-dependent randomized policies $u=(u_t; t \in \mathcal T),u'=(u'_t; t \in \mathcal T)$, the initial state $s_0
\in \mathcal S$ and a constant $\delta \in [0,1]$, the derivative of
the mean-variance $J_0^\delta(s_0)$ at policy $u$ along direction
$u'$ takes the following form
\begin{equation}\label{equ:deriva_mv}
\frac{\partial J_0^\delta(s_0)}{\partial
\delta}\Big|_{\delta=0}=\bm{e}_{(s_0,\mu_0(s_0))}\sum_{t=0}^{T}
\prod_{\tau=0}^{t-1} \tilde{\bm P}_\tau\left[\tilde{\bm
r}'_t-\tilde{\bm r}_t+\left(\tilde{\bm P'}_t-\tilde{\bm
P}_t\right)\tilde{\bm g}_{t+1}\right].
\end{equation}
\end{lem}

Now, we give the proof of Theorem \ref{thm:converge}.

\begin{proof}
We first prove that Algorithm~\ref{alg:local} converges to a fixed
point solution to (\ref{eq_fixedpoint}). For each $s_0 \in
\mathcal{S}$ and $k \ge 0$, we have
\begin{equation}\label{equ: mono}
  J_0^{u^{(k)}}(s_0) \le J_0^{u^{(k+1)}}(s_0).
\end{equation}
The above inequality is ensured by the policy improvement step in
Algorithm~\ref{alg:local}, i.e.,
\begin{equation*}
J_0^{u^{(k)}}(s_0)=\hat J_0^{u^{(k)}}(s_0,y_0^{(k)}) \le \hat
J_0^{u^{(k+1)}}(s_0,y_0^{(k)}) \le \max\limits_{y_0 \in \mathcal
Y}\hat J_0^{u^{(k+1)}}(s_0,y_0) = J_0^{u^{(k+1)}}(s_0).
\end{equation*}
Therefore, for each $s_0 \in \mathcal{S}$, the sequence
$\left\{J_0^{u^{(k)}}(s_0); k \ge 0 \right\}$ generated by
Algorithm~\ref{alg:local} is monotonically increasing, necessarily
$\left\{J_0^{u^{(k)}}(s_0); k \ge 0\right\}$ converges. Since
$\mathcal U^{\rm HD}$ is finite, the equality in (\ref{equ: mono})
holds within finite iterations. Thus, the convergence of
Algorithm~\ref{alg:local} is proved. Suppose
Algorithm~\ref{alg:local} converges to $u^*$ with corresponding
$y_0^*=\mathbbm{E}_{s_0}^{u^*}[R_{0:T}]$, the pair $(u^*,y_0^*)$ must
satisfy the fixed point equation (\ref{eq_fixedpoint}). Therefore,
Algorithm~\ref{alg:local} converges to a fixed point solution
$(u^*,y_0^*)$ to (\ref{eq_fixedpoint}). Below, we show that
$\mathcal{U}_{\rm valid}^{\rm HD}(u^*)$ is a valid pruned
deterministic policy space and further show that $u^*$ and $y_0^*$
are both local optima.

To prove part~(i), we show that
\begin{equation*}
J_0^u(s_0)=J_0^{u^*}(s_0), \quad\forall  u \in \mathcal{U}^{\rm HD}
\backslash \mathcal{U}_{\rm valid}^{\rm HD}(u^*).
\end{equation*}
For any deterministic policy $u \in \mathcal{U}^{\rm HD} \backslash
\mathcal{U}_{\rm valid}^{\rm HD}(u^*)$, we have
\begin{equation*}
\hat {J}^{u^*}_0(s_0,y_0^*) = \hat {J}^{u}_0(s_0,y_0^*),
\quad\forall s_0 \in \mathcal S.
\end{equation*}
We can prove $\mu_0^u(s_0)=y_0^*$ with contradiction as follows.
Assume $\mu_0^u(s_0) \neq y_0^*$, then we have
\begin{eqnarray*}
        J_0^u(s_0)&=&\hat {J}^{u}_0(s_0,\mu_0^u(s_0))\\
        &=&\hat {J}^{u}_0(s_0,y_0^*)+\lambda (\mu_0^u(s_0)-y_0^*)^2\\
        &>&\hat {J}^{u}_0(s_0,y_0^*)\\
        &=&\hat {J}^{u^*}_0(s_0,y_0^*) ~ = ~ J_0^{u^*}(s_0),
\end{eqnarray*}
which means that Algorithm~\ref{alg:local} will not stop at $u^*$.
This is a contradiction and the assumption $\mu_0^u(s_0) \neq y_0^*$
does not hold. Thus, we must have $\mu_0^u(s_0)=y_0^*$, which
implies $J_0^u(s_0)=J_0^{u^*}(s_0)$. Therefore, $\mathcal{U}_{\rm
valid}^{\rm HD}(u^*)$ is a valid pruned deterministic policy space
by Definition \ref{def:valid_policy}.

We then prove part~(ii).  When Algorithm~\ref{alg:local} stops at
policy $u^*$,  it indicates that the corresponding policy $\tilde
u^*$ attains the inner optimum with initial state $(s_0,y_0^*)$,
that is
\begin{equation*}
V_0^{\tilde u^*}(s_0,y_0^*)= V_0^*(s_0,y_0^*).
\end{equation*}
Applying Lemma~\ref{lem:opt_con} to MDP $\widetilde{\mathcal M} $, we
have for any $t \in \mathcal T$ and $(s,a) \in \mathcal K$,
\begin{align*}
        \tilde r_t(s,y,\tilde u_t^*(s,y))
        ~ + \sum\limits_{(s',y') \in \mathcal S \times \mathcal Y }\tilde P_t(s',y'|s,y,\tilde u_t^*(s,y))g_{t+1}^{\tilde u^*}(s',y') \nonumber\\
        \geq
        ~ \tilde  r_t(s,y,a) ~+ \sum\limits_{(s',y') \in \mathcal S \times \mathcal Y}\tilde P_t(s',y'|s,y,a)g_{t+1}^{\tilde u^*}(s',y'),
\end{align*}
which implies that $\frac{\partial J_0^\delta(s_0)}{\partial
\delta}\Big|_{\delta=0} \le 0$, since $\tilde{\bm P}_\tau$ is
non-negative.

For any deterministic policy $u \in \mathcal{U}_{\rm valid}^{\rm
HD}(u^*)$, since $\hat {J}^{u^*}_0(s_0,y_0^*) \neq \hat
{J}^{u}_0(s_0,y_0^*)$, the inequality is strict for some $t \in
\mathcal T ~\text{and}~ (s,a) \in \mathcal K$. Thus we have
$\frac{\partial J_0^\delta(s_0)}{\partial \delta}\Big|_{\delta=0} <
0$ for any directions, which indicates that $u^*$ is a strictly
locally optimal policy in the valid pruned mixed policy space
generated by $\mathcal{U}_{\rm valid}^{\rm HD}(u^*)$.

We finally prove part~(iii). Since $y_0^*$ is not a break point.
According to Definition~\ref{def:cri}, there exists a constant
$\delta'>0$ such that $u^*$ remains optimal for any pseudo MV-MDPs
$\left\{\hat{\mathcal M}(y_0): y_0 \in
(y_0^*-\delta',y_0^*+\delta')\right\}$. Then, we have
\begin{equation}\nonumber
        \hat J^*_0(s_0,y_0^*)=\hat J^{u^*}_0(s_0,y_0^*) \ge \hat
        J^{u^*}_0(s_0,y_0) = \hat J^*_0(s_0,y_0), \quad\forall y_0 \in
        (y_0^*-\delta',y_0^*+\delta').
\end{equation}
Therefore, $y_0^*$ is a local optimum of $\hat J^*_0(s_0,y_0)$ in
the real space $\mathcal Y$.
\end{proof}

\subsection{Proof of Theorem~\ref{thm:global}}
\begin{proof}
    We prove this result by showing that the optimal value function
    $\hat J_0^*(s_0,y_0)$ of the pseudo MV-MDP (\ref{equ:inner}) is a
    concave function with respect to $y_0$, which is equivalent to prove
    $V_0^*(s_0,y_0)$ is a concave function on $\mathcal S \times
    \mathcal Y$ by Theorems \ref{thm:relation} \& \ref{thm:opt}. To this
    end, we prove $V^*_t(s_t,y_t)$ defined in (\ref{equ:opt_dp}) is
    concave on $\mathcal S \times \mathcal Y$ for all $t \in \mathcal T$
    by induction.

    For $t=T$, $V^*_T(s_T,y_T)=-\lambda y_T^2$ is obviously concave on
    $\mathcal S \times \mathcal Y$. Suppose that
    $V^*_{t+1}(s_{t+1},y_{t+1})$ is concave on $\mathcal S \times
    \mathcal Y$ for some $t \in \mathcal T$, we aim to prove
    $V^*_{t}(s_{t},y_{t})$ is also concave on $\mathcal S \times
    \mathcal Y$. We first verify that $\int_{\mathcal
        X}V^*_{t+1}(f_{t+1}(s_t,a_t,x),y_t-r_t(s_t,a_t,x))q_t(dx)$ is concave
    on $(s_t,a_t,y_t) \in \mathcal K \times \mathcal Y \subseteq
    \mathcal S \times \mathcal A \times \mathcal Y$, where $x$ is the
    possible value of random variable $\xi_t$. Arbitrarily choose
    $(s_t,a_t), (s_t',a_t') \in \mathcal K, \ y_t,y_t' \in \mathcal Y$,
    and $\omega \in [0,1]$. By using the concavity of $V^*_{t+1}$, we
    have
    \begin{align*}
        &  \omega   \int_\mathcal{X} V^*_{t+1}(f_{t}(s_t,a_t,x),y_t-r_t(s_t,a_t,x))q_t(dx) + (1-\omega) \int_\mathcal{X}V^*_{t+1}(f_{t}(s_t',a_t',x),y_t'-r_t(s_t',a_t',x))q_t(dx)\\
        &=  \int_\mathcal{X} \left\{\omega V^*_{t+1}(f_{t}(s_t,a_t,x),y_t-r_t(s_t,a_t,x)) + (1-\omega)V^*_{t+1}(f_{t}(s_t',a_t',x),y_t'-r_t(s_t',a_t',x))\right\}q_t(dx)\\
        & \le \int_\mathcal{X} \left\{ V^*_{t+1}(\omega f_{t}(s_t,a_t,x)+(1-\omega)f_{t}(s_t',a_t',x),\omega (y_t-r_t(s_t,a_t,x))+(1-\omega)(y_t'-r_t(s_t',a_t',x))) \right\}q_t(dx)\\
        &=\int_\mathcal{X} \left\{V^*_{t+1}( f_{t}(\omega s_t + (1-\omega)s_t',\omega a_t+(1-\omega) a_t',x),\right.\\
        &\phantom{=\;\;}\left.
        \omega y_t + (1-\omega) y_t' -r_t(\omega s_t + (1-\omega)s_t',\omega a_t+(1-\omega) a_t',x))
        \right\}q_t(dx),
    \end{align*}
    where the inequality is ensured by the concavity of $V^*_{t+1}$, the
    last equality follows from Condition~(ii), and the feasibility of
    combined state-action pairs $(\omega s + (1-\omega)s',\omega
    a+(1-\omega) a') \in \mathcal K$ is guaranteed by the convexity of
    $\mathcal K$ in Condition~(i). Since $r_t(s_t,a_t,x)$ is linear to
    $s_t \in \mathcal S$ and $a_t \in \mathcal A$, we can obtain the
    concavity of
    \begin{equation}
        \nonumber
        \int_{\mathcal X} r_t(s_t,a_t,x)q_t(dx) + \int_{\mathcal X}V^*_{t+1}(f_{t}(s_t,a_t,x),y_t-r_t(s_t,a_t,x))q_t(dx)
    \end{equation}
    on $\mathcal K \times \mathcal Y$.

    Suppose $\tilde u^*=(\tilde u_t^*;t \in \mathcal T)\in\tilde{\mathcal U}^{\rm MD}$ is the optimal
    policy of the standard MDP (\ref{equ:newmdp}). Then we have
    \begin{align*}
        & \omega V^*_t(s_t,y_t) + (1-\omega) V^*_t(s_t',y_t') \\
        &= \omega \left\{  \int_{\mathcal X} r_t(s_t,\tilde u^*_t(s_t,y_t),x)q_t(dx)+\int_{\mathcal X}V^*_{t+1}(f_{t}(s_t,\tilde u^*_t(s_t,y_t),x),y_t-r_t(s_t,\tilde u^*_t(s_t,y_t),x))q_t(dx)\right\} \\
        &+(1-\omega) \left\{  \int_{\mathcal X} r_t(s_t',\tilde u^*_t(s_t',y_t'),x)q_t(dx)+\int_{\mathcal X}V^*_{t+1}(f_{t}(s_t',\tilde u^*_t(s_t',y_t'),x),y_t'-r_t(s_t',\tilde u^*_t(s_t',y_t'),x))q_t(dx)\right\} \\
        &\le \int_{\mathcal X} r_t(\omega s_t + (1-\omega)s_t',\omega \tilde u^*_t(s_t,y_t)+(1-\omega) \tilde u^*_t(s_t',y_t'),x)q_t(dx)\\
        & +\int_{\mathcal X}\bigg\{ V^*_{t+1}( f_{t}(\omega s_t + (1-\omega)s_t',\omega \tilde u^*_t(s_t,y_t)+(1-\omega) \tilde u^*_t(s_t',y_t'),x),\bigg.\\
        &\phantom{=\;\;}\bigg.
        \omega y_t + (1-\omega) y_t' -r_t(\omega s_t + (1-\omega)s_t',\omega \tilde u^*_t(s_t,y_t)+(1-\omega) \tilde u^*_t(s_t',y_t'),x)) \bigg\}q_t(dx)\\
        &\le  \max\limits_{a \in \mathcal A(\omega s_t+(1-\omega)s_t')}\bigg\{ \int_{\mathcal X} r_t(\omega s_t + (1-\omega)s_t',a,x)q_t(dx) +\bigg.\\
        &\phantom{=\;\;}\bigg. \int_{\mathcal X} \left\{ V^*_{t+1}( f_{t}(\omega s_t + (1-\omega)s_t',a,x),\omega y_t + (1-\omega) y_t' -r_t(\omega s_t + (1-\omega)s_t',a,x)) \right\}q_t(dx)\bigg\}\\
        &=V_t^*\big(\omega s_t + (1-\omega)s_t', \omega y_t + (1-\omega)y_t'\big),
    \end{align*}
    where the last inequality is ensured by the convexity of state-action pairs $\mathcal K$ in Condition (i), and thus
    $V^*_t(s_t,y_t)$ is concave on $\mathcal S \times
    \mathcal Y$. Therefore, we can recursively derive that $\hat
    J_0^*(s_0, y_0)=V^*_0(s_0,y_0)$ is also a concave function on
    $\mathcal S \times \mathcal Y$ by induction.

    Since Algorithm~\ref{alg:local} converges to $u^*$ and
    $y^*_0$, we can see that $(u^*,y^*_0)$
    is a fixed point, i.e.,
    \begin{equation}
        \nonumber
        \hat J_0^{u^*}(s_0,y^*_0) = \hat J_0^*(s_0, y^*_0), \quad s_0 \in \mathcal S,
    \end{equation}
    \begin{equation}
        \nonumber
        \hat J_0^{u^*}(s_0,y^*_0) = J_0^{u^*}(s_0), \quad s_0 \in \mathcal S.
    \end{equation}
    Since $\hat J_0^*(s_0,y_0)$ is quadratically concave in $y_0$, we
    know that $\hat J_0^*(s_0,y_0)$ has a unique local optimum in $y_0
    \in \mathcal Y$. Therefore, with Theorem~\ref{thm:converge}, we
    directly derive that the converged point $y^*_0$ of
    Algorithm~\ref{alg:local} is both the local and the global maximum
    point of $\hat J_0^*(s_0,y_0)$, i.e.,
    \begin{equation}
        \nonumber
        \hat J_0^*(s_0,y^*_0) = \max\limits_{y_0 \in \mathcal Y}\hat J_0^*(s_0,y_0), \quad s_0 \in \mathcal S.
    \end{equation}
    The above three equations imply that
    \begin{equation}
        \nonumber
        J_0^{u^*}(s_0) = \max\limits_{y_0 \in \mathcal Y}\hat J_0^*(s_0,y_0)=J_0^{*}(s_0), \quad s_0 \in \mathcal S.
    \end{equation}
    Therefore, $u^*$ is the globally optimal policy of the
    finite-horizon MV-MDP (\ref{equ:mv_mdp}). In summary,
    Algorithm~\ref{alg:local} converges to the global optimum with the
    conditions in Theorem~\ref{thm:global}.
\end{proof}

\subsection{Proof of Theorem~\ref{thm:structure_mv}}
    \begin{proof}
    We prove this theorem by showing that $V_t^*(s_t,y_t)$ has the following form,
    \begin{equation}\label{equ:quadratic}
        V_t^*(s_t,y_t)=b_{2,t} y_t^2 + (b_{0,t}+b_{1,t} s_t) y_t + g_t(s_t), \quad\forall (s_t,y_t) \in \tilde{\mathcal{S}}, t \in \mathcal T,
    \end{equation}
    where $b_{0,t},b_{1,t},b_{2,t}$ are real numbers and $g_t(s_t)$ is a quadratic function of $s_t$. We prove (\ref{equ:quadratic}) by induction.

    For $t=T$, $V_T^*(s_T,y_T)=-\lambda y_T^2$ obviously has the form (\ref{equ:quadratic}). Suppose that $(\ref{equ:quadratic})$ holds for some $t+1$, we analyze the property of $V_{t}^*(s_{t}, y_{t})$ by applying dynamic programming, that is,
    \begin{align}
        \nonumber
        V_{t}^*(s_{t}, y_{t})
        &=\max\limits_{a \in \mathcal A(s_t)}\int_{\mathcal X}\left\{r_t(s_t,a,x)+V_{t+1}^*(f_t(s_t,a,x),y_t-r_t(s_t,a,x))\right\}q_t(dx).
    \end{align}
    Since $V_{t+1}^*$ has the form (\ref{equ:quadratic}), easy to show
    \begin{equation}\label{equ:action_value}
        \tilde V_{t}^*(s_t,y_t,a):=\int_{\mathcal X}\left\{r_t(s_t,a,x)+V_{t+1}^*(f_t(s_t,a,x),y_t-r_t(s_t,a,x))\right\}q_t(dx)
    \end{equation}
    is quadratically concave with respect to $a \in \mathcal A$ combined with the proof of Theorem~\ref{thm:global}. Taking derivative of function $\tilde V_{t}^*(s_t,y_t,a)$ with respect to $a$ and let $\frac{\partial \tilde V_{t}^*}{\partial a}=0$, we obtain $a^*=c_{2,t}s_t + c_{1,t}y_t+c_{0,t}$ where $c_{0,t},c_{1,t},c_{2,t}$ are real numbers. Substituting $a^*=c_{2,t}s_t + c_{1,t}y_t+c_{0,t}$ into (\ref{equ:action_value}), we obtain $V_{t}^*(s_{t}, y_{t})=\tilde V_{t}^*(s_t,y_t,a^*)$ takes the form (\ref{equ:quadratic}). Therefore, (\ref{equ:quadratic}) holds.

    Take $t=0$ in (\ref{equ:quadratic}), since $V_{0}^*(s_{0}, y_{0})$ is concave with respect to $y_0$, the minimization of $V_{0}^*(s_{0}, \cdot)$ in $\mathcal Y$ is attained at $y^*_0=-\frac{b_{0,0}+b_{1,0}s_0}{b_{2,0}}:=k_0+k_1 s_0$.
\end{proof}

\subsection{Proof of Theorem~\ref{thm:dp_pf}}
Theorem~\ref{thm:dp_pf} is similar to Theorem~\ref{thm:opt}, but has
some differences due to the special reward function. We first give
some preliminaries. We define an operator $\mathbb{\hat L}_t^{\varphi}:
\mathcal B(\tilde{\mathcal{S}}) \rightarrow \mathcal B(\tilde{\mathcal{S}})$ for a
stochastic kernel $\varphi$ on $\mathcal{A}$ given ${\mathcal{S}}$
and $t \in \mathcal T$ by
\begin{equation}\label{equ:opera_pf}
   \mathbb{\hat L}_t^\varphi v(s,y_0) := \sum\limits_{\bm a \in \mathcal{A}(s)}
    \varphi(\bm a|s)\mathbbm E[v(e_t^0 s+\bm Q_t' \bm a,y_0)],\quad
    v \in \mathcal B(\tilde{\mathcal{S}}).
\end{equation}
Given $y_0 \in \mathcal Y$, we further denote by
\begin{equation}
    \nonumber
    \hat{J}_{t}^{u}(s_t,y_0) := \mathbbm{E}_{s_0}^{u}
    \big[s_t+R_{t:T}-\lambda\big(s_t+R_{t:T}-y_0\big)^2|s_t\big],
    \quad s_t \in \mathcal S,~t \in \mathcal T
\end{equation}
and
\begin{equation}
    \nonumber
    \hat{J}_{t}^*(s_t,y_0) := \sup\limits_{{u} \in {\mathcal U}^{\rm MR}}\hat{J}_{t}^{u}(s_t, y_0),
    \quad s_t \in \mathcal S,~t \in \mathcal T
\end{equation}
the expected total rewards under Markov randomized policy $u \in
\mathcal U^{\rm MR}$ and the optimal value function from stage $t$ to
terminal stage $T$, respectively. Next, we introduce a lemma.

\begin{lem}\label{lem:dp_policy}
    Given $y_0 \in \mathcal Y$ and $u = (u_t;t \in \mathcal T) \in \mathcal
    U^{\rm MR}$,  suppose the sequence $\left\{V_t^{u};
    t \in \mathcal T\right\}$ is generated by
    \begin{equation}\label{equ:pf_dp_policy}
        V_t^{u}(s_t,y_0)=\mathbb{\hat L}_t^{u_t}V_{t+1}^{u}(s_t,y_0), \quad\forall t \in \mathcal T~\text{and}~V_T^{u}(s_T,y_0) = s_T-\lambda (s_T-y_0)^2,
    \end{equation}
    then we have
    \begin{equation}\label{equ:dp_policy}
        V_t^{u}(s_t,y_0)=\hat{J}_t^{u}(s_t,y_0), ~\forall s_t \in \mathcal{S}, t \in \mathcal T.
    \end{equation}
\end{lem}

\begin{proof}
We prove this lemma by induction. Taking $t=T-1$ and using
(\ref{equ:opera_pf}), we have
    \begin{eqnarray*}
        &&V_{T-1}^u(s_{T-1},y_0) =\hat{\mathbb L}_{T-1}^{u_{T-1}}V_{T}^{u}(s_{T-1},y_0) \\
        &=&\sum\limits_{\bm a_{T-1} \in \mathcal{A}(s_{T-1})}u_{T-1}(\bm a_{T-1}|s_{T-1})\mathbbm{E}_{s_0}^u[V_T^u(e_{T-1}^0 s_{T-1}+\bm Q_{T-1}' \bm a_{T-1},y_0)]\\
        &=&\sum\limits_{\bm a_{T-1} \in \mathcal{A}(s_{T-1})}u_{T-1}(\bm a_{T-1}|s_{T-1})\mathbbm{E}_{s_0}^u[e_{T-1}^0 s_{T-1}+\bm Q_{T-1}' \bm a_{T-1}-\lambda (e_{T-1}^0 s_{T-1}+\bm Q_{T-1}' \bm a_{T-1}-y_0)^2]\\
        &=&\mathbb{E}_{s_0}^u [e_{T-1}^0 s_{T-1}+\bm Q_{T-1}' \bm a_{T-1}-\lambda (e_{T-1}^0 s_{T-1}+\bm Q_{T-1}' \bm a_{T-1}-y_0)^2|s_{T-1}]\\
        &=&\mathbbm{E}_{s_0}^u[s_T-\lambda (s_T-y_0)^2|s_{T-1}]\\
        &=&\mathbbm{E}_{s_0}^u[s_{T-1}+R_{T-1:T}-\lambda (s_{T-1}+R_{T-1:T}-y_0)^2|s_{T-1}]\\
        &=&\hat{J}_{T-1}^u(s_{T-1},y_0),
    \end{eqnarray*}
where the second and third equalities follow from
(\ref{equ:opera_pf}) and definition of $V_T^{u}(s_T,y_0)$, the fifth
and sixth equalities are guaranteed by the definitions of transition
function and reward function, respectively. Suppose that
(\ref{equ:dp_policy}) is true for $T-1,T-2,\ldots,t+1$, we show that
it is also true for $t$.
    \begin{align*}
        V_{t}^u(s_{t},y_0)
        =&\quad\mathbb{\hat L}_{t}^{u_{t}}V_{t+1}^{u}(s_{t},y_0) \\
        =&\sum\limits_{\bm a_{t} \in \mathcal{A}(s_{t})}u_{t}(\bm a_{t}|s_{t})\mathbbm{E}_{s_0}^u[V_{t+1}^u(e_{t}^0 s_{t}+\bm Q_t' \bm a_{t},y_0)]\\
        =&\sum\limits_{\bm a_{t} \in \mathcal{A}(s_{t})}u_{t}(\bm a_{t}|s_{t})\mathbbm{E}_{s_0}^u[\hat{J}_{t+1}^u(e_{t}^0 s_{t}+\bm Q_t' \bm a_{t},y_0)]\\
        =&\sum\limits_{\bm a_{t} \in \mathcal{A}(s_{t})}u_{t}(\bm a_{t}|s_{t})\mathbbm E_{s_0}^u[e_{t}^0 s_{t}+\bm Q_t' \bm a_{t}+R_{t+1:T}-\lambda (e_{t}^0 s_{t}+\bm Q_t' \bm a_{t}+R_{t+1:T}-y_0)^2|e_{t}^0 s_{t}+\bm Q_t' \bm a_{t}]\\
        =&\quad\mathbbm E_{s_0}^u[e_{t}^0 s_{t}+\bm Q_t' \bm a_{t}+R_{t+1:T}-\lambda (e_{t}^0 s_{t}+\bm Q_t' \bm a_{t}+R_{t+1:T}-y_0)^2|s_{t}]\\
        =&\quad\mathbbm E_{s_0}^u[s_{t}+R_{t:T}-\lambda (s_{t}+R_{t:T}-y_0)^2|s_{t}]\\
        =&\quad \hat{J}_{t}^u(s_{t},y_0),
    \end{align*}
where the second to last equality follows from the special form of
transition function and reward function. Therefore,
(\ref{equ:dp_policy}) holds by induction.
\end{proof}

Lemma~\ref{lem:dp_policy} shows that the value function
$\hat{J}_0^u(s_0,y_0)$ under a Markov randomized policy $u \in
\mathcal{U}^{\rm MR}$ can be computed by dynamic programming
(\ref{equ:pf_dp_policy}). Next we prove Theorem~\ref{thm:dp_pf},
which establishes the dynamic programming of the optimal value
function $\hat{J}_0^*(s_0,y_0)$.

\textbf{Proof of Theorem~\ref{thm:dp_pf}}. First, it is well known
from Theorem~5.5.1 of \cite{Puterman94} that there exists a
Markov randomized policy $u \in \mathcal U^{\rm MR}$ corresponds to
each history-dependent randomized policy $u' \in \mathcal U^{\rm HR}$
such that
\begin{equation*}
\mathbb{P}_{s_0}^{u}(s_t,\bm a_t)=\mathbb{P}_{s_0}^{u'}(s_t,\bm
a_t),\quad\forall  (s_t,\bm a_t) \in \mathcal{K}, t \in \mathcal T,
\end{equation*}
which implies
\begin{equation}\nonumber
\hat{J}_0^{*}(s_0,y_0)=\sup\limits_{u \in \mathcal U^{\rm HR}}
\hat{J}_0^{u}(s_0,y_0)=\sup\limits_{u \in \mathcal U^{\rm MR}}
\hat{J}_0^{u}(s_0,y_0), \quad s_0 \in \mathcal S.
\end{equation}
We next prove $V_0^*={\hat{J}}_0^{*}$ by establishing two inequalities
of opposing directions: $V_0^* \le \hat{J}_0^{*}$ and $V_0^* \ge
\hat{J}_0^{*}$.

For the former, by the definition of operator $\mathbb{\hat{L}}^*_t$ in
(\ref{equ:opt_opera_pf}), there exists a policy $u = (u_t; t \in
\mathcal T) \in \mathcal U^{\rm {MD}}$ such that
\begin{equation}\label{equ:exi}
V_t^*(s_t,y_0)=\mathbb{\hat{L}}^*_t V_{t+1}^*(s_t,y_0)=\mathbb{\hat{L}}_t^{u_t}
V_{t+1}^*(s_t,y_0), \quad\forall s_t \in \mathcal S, t \in \mathcal
T.
\end{equation}
Using the result of Lemma~\ref{lem:dp_policy} and noting that
$V_T^*(s_T,y_0)=s_T-\lambda (s_T-y_0)^2=V_T^{u}(s_T,y_0)$, we have
\begin{equation}
        \nonumber
        V_0^*(s_0,y_0)=V_0^{u}(s_0,y_0)= \hat{J}_0^{u}(s_0,y_0),\quad s_0 \in \mathcal{S},
\end{equation}
thus $V_0^*(s_0,y_0) \le \sup\limits_{u \in \mathcal U^{\rm MR}}
\hat{J}_0^{u}(s_0,y_0)= \hat{J}_0^{*}(s_0,y_0)$.

For the latter, we just need to show that $V_0^* \ge  \hat{J}_0^{u}$
for each $u=(u_t; t \in \mathcal T) \in \mathcal U^{\rm MR}$. This
statement  holds by using the same argument of the former case, just
with (\ref{equ:forall}) in lieu of (\ref{equ:exi}),
\begin{equation}\label{equ:forall}
V_t^*(s_t,y_0)=\mathbb{\hat{L}}^*_t V_{t+1}^*(s_t,y_0) \ge \mathbb{\hat{L}}_t^{u_t}
V_{t+1}^*(s_t,y_0), \quad\forall u \in \mathcal U^{\rm MR},  s_t \in
\mathcal S, t \in \mathcal T,
\end{equation}
and noting that $\mathbb{\hat{L}}_t^{u_t}$ is a monotonically increasing
operator. Therefore, we have $V_0^*=\hat{J}^*_0$.

Furthermore, if $\bm a_t^* \in \mathcal A(s_t)$ attains the maximum in the
operation $\mathbb{\hat{L}}^*_t  V^*_{t+1}(s_t,y_0)$, then we have
$\hat{J}^*_0(s_0,y_0)=V_0^*(s_0,y_0)=\hat{J}^{\hat{u}^*}_0(s_0,y_0)$
where $\hat{u}^*=(\hat{u}^*_t; t \in \mathcal T) \in {\mathcal
U}^{\rm {MD}}$ with $\hat{u}_{t}^*(s_t)=\bm a_t^*$, which
implies that $\hat{u}^*$ is an optimal policy for the inner pseudo
MV-MDP (\ref{equ:inner_pf}).

\end{appendices}

\end{document}